\documentclass[12pt,leqno,oneside]{amsart}
\usepackage{amssymb}
\usepackage{mathrsfs}
\usepackage{graphicx,color}
\usepackage{newtxtext}
\usepackage[varg]{newtxmath}
\theoremstyle{plain} 
\newtheorem{theorem}{Theorem}[section]
\newtheorem{corollary}[theorem]{Corollary}
\newtheorem{proposition}[theorem]{Proposition}
\newtheorem{lemma}[theorem]{Lemma}

\theoremstyle{definition} 
\newtheorem{definition}[theorem]{Definition}
\newtheorem{remark}[theorem]{Remark}

\newcommand{\R}{\mathbb{R}}

\usepackage{constants}
\newconstantfamily{eps}{symbol=\varepsilon}
\numberwithin{equation}{section}
\usepackage{todonotes}

\usepackage{bm}
\renewcommand{\vec}[1]{\bm{#1}}
\usepackage{hyperref}

\title{A stochastic model of grain boundary dynamics: A Fokker-Planck perspective}

\author{Yekaterina Epshteyn}
\address[Yekaterina Epshteyn]%
{Department of Mathematics,
The University of Utah,
Salt Lake City, UT 84112, USA}
\email{epshteyn@math.utah.edu}

\author{Chun Liu}
\address[Chun Liu]%
{Department of Applied Mathematics, Illinois Institute of Technology.
Chicago, IL 60616, USA}
\email{cliu124@iit.edu}

\author{Masashi Mizuno}
\address[Masashi Mizuno]%
{Department of Mathematics, College of Science
and Technology, Nihon University, Tokyo 101-8308 JAPAN}
\email{mizuno.masashi@nihon-u.ac.jp}


\allowdisplaybreaks[1]
\begin{document}

%
%


%
%
%

\begin{abstract}
Many technologically useful materials are polycrystals composed of
 small monocrystalline grains that are separated by grain boundaries
 of crystallites with different lattice orientations. The energetics and connectivities
of the grain boundaries play an essential role in defining the effective
properties of materials across multiple scales.
In this paper we derive a Fokker-Planck model for the evolution
of the planar grain boundary network. The proposed model considers
anisotropic grain boundary energy which depends on lattice misorientation and
takes into account mobility of the triple junctions, as well as
independent dynamics of the misorientations. We establish long time
asymptotics of the Fokker-Planck solution, namely  
the joint probability density function of misorientations and triple
junctions, and closely related the marginal probability density of misorientations.
Moreover, for an equilibrium configuration of a boundary network, we derive explicit
local algebraic relations, a generalized Herring Condition formula, as well as
formula that connects grain boundary energy density with the geometry
of the grain boundaries that share a triple junction. Although the stochastic model neglects
the explicit interactions and correlations among triple junctions, the considered
specific form of the noise, under the fluctuation-dissipation assumption,
provides partial information about evolution of a grain boundary
network, and is consistent with presented results of extensive grain growth simulations. 

\end{abstract}

\maketitle

\section{Introduction}\label{sec1}
{Most technologically useful materials are polycrystalline
microstructures composed of a myriad of small monocrystalline grains
separated by grain boundaries. The energetics and connectivities of
grain boundaries play an important role in defining the main
properties of materials across multiple scales. }
More recent
mesoscale experiments and simulations provide large amounts of
information about both geometric features and crystallography of the
grain boundary network in material microstructures.  

A classical model, due to Mullins and
Herring~\cite{doi:10.1007/978-3-642-59938-5_2,
doi:10.1063/1.1722511,doi:10.1063/1.1722742}, for the evolution of grain
boundaries in polycrystalline materials is based on the motion by mean
curvature as the local evolution law. Mathematical analysis of the
motion by mean curvature can be found, for instance in~\cite{MR1100211,MR2024995,MR1100206,MR3155251}, and the
study of the curvature flow for networks can be found
in, e.g.~\cite{MR1833000,MR3967812,MR2075985,MR3495423,MR3612327,MR0485012}.  In
addition, to have a well-posed model of the evolution of the grain
boundary network, one has to impose a separate condition at the triple
junctions where three grain boundaries meet \cite{MR1240580,MR1833000}.

Grain growth is a very complex multiscale process. It involves, for example,
dynamics of grain boundaries, triple junctions (triple junctions are where three grain boundaries
meet) and the dynamics of lattice misorientations (difference in the orientation between two neighboring grains
that share the grain boundary)/possibility of
grains rotations. Recently, there are some studies that
consider interactions among grain boundaries and triple junctions, e.g.,
\cite{upmanyu1999triple,upmanyu2002molecular,barmak_grain_2013,ZHAO2017345,zhang2017equation,zhang2018motion}. In
our very recent work \cite{MR4263432,Katya-Chun-Mzn2}, we developed a new model for the evolution of the
2D grain-boundary network with finite mobility of the triple junctions
and with dynamic lattice misorientations (possibility of grain
rotations). In \cite{MR4263432,Katya-Chun-Mzn2},  using  the
energetic  variational  approach, we derived a system of geometric
differential equations to describe the motion of such grain
boundaries. Under assumption of no curvature effect,  we established a
local well-posedness result, as well as large time asymptotic behavior
for the model. Our results included obtaining explicit energy decay
rate for the system in terms of mobility of the triple junction and
the misorientation parameter (grains rotation relaxation time scale). In addition,  we conducted several numerical experiments for the 2D grain boundary network in order to further understand/illustrate the effect of relaxation time scales, e.g. of the curvature of grain
boundaries, mobility of triple junctions, and dynamics of misorientations on how the grain boundary system decays energy and coarsens with time \cite{Katya-Chun-Mzn2, Katya-Chun-Mzn4}. In \cite{Katya-Chun-Mzn4}, we also presented
and discussed relevant experimental results of grain growth in thin
films. Note that in the work \cite{MR4263432,Katya-Chun-Mzn2}, the
mathematical analysis of the model was done under assumption of no critical
events/no disappearance events, e.g., grain
disappearance, facet/grain boundary disappearance, facet interchange,
splitting of unstable junctions (however, numerical simulations were
performed with critical events).


The current work is motivated and is closely related to the work in
\cite{DK:gbphysrev,MR2772123,MR3729587} where a reduced 1D
coarsening model based
on the dynamical system was studied for texture evolution and was used
to identify texture evolution as a gradient flow, see also the article
\cite{BobKohn} for a perspective on the problem. In addition,  this
paper is a further
extension of our work in \cite{MR4263432,Katya-Chun-Mzn2,
  Katya-Chun-Mzn4}, and  the work
\cite{BERDICHEVSKY201250} is also relevant. In this paper, we study a stochastic model for the evolution of planar grain boundary
network in order to be able to incorporate and model the effect of the critical
events. 
We start with a simplified model and, hence,  consider the Langevin
equation analog of the model from \cite{MR4263432},  with the interactions among
triple junctions and misorientations modeled as white noise.  Next,  we use
the energetic variational approach to establish the associated
fluctuation-dissipation theorem.  The fluctuation-dissipation property ensures
that the free energy of the corresponding Fokker-Planck system is
dissipative. Moreover, the fluctuation-dissipation theorem also gives
the sufficient condition for the steady-state solution of the
Fokker-Planck equation to be given by the Boltzmann distribution.

Next, we study the well-posedness of the derived Fokker-Planck system
under assumption of the fluctuation-dissipation relation. In
particular,  we show that the solution of the Fokker-Planck equation
converges exponentially fast to the Boltzmann distribution for grain
boundary energy of the system.  Note that, the grain boundary energy
has degeneracy with respect to the misorientations (due to constraints on
misorientations)  and singularity with respect to the triple
junction. To overcome these difficulties, based on the idea of
\cite{MR1812873} (and see also a relevant work \cite{info:hdl/2115/71067}), we consider Fokker-Planck equation in a weighted
$L^2$ space, and we use the semigroup theory and the Poincar\'e
inequality to obtain  well-posedness and long-time asymptotics of the
solution.

Finally, for an equilibrium configuration of a boundary network, we derive explicit
local algebraic relations, a generalized Herring Condition formula, as well as
formula that connects grain boundary energy density with the geometry
of the grain boundaries that share a triple junction. The later local
algebraic relation gives the condition for a steady-state solution of
marginal probability density of misorientations to be the Boltzmann distribution
with respect to a grain boundary energy density. Such a steady-state
solution for marginal  probability density of misorientations is
related to the
observed Boltzmann distribution for the steady-state Grain Boundary Character
Distribution (GBCD) statistical metric of grain growth, e.g.  \cite{MR3729587, DK:gbphysrev, MR2772123}.
Although the investigated simplified stochastic model neglects
the explicit interactions and correlations among triple junctions, the considered
specific form of the noise, under the fluctuation-dissipation assumption,
provides partial information about evolution of a grain boundary
network, and is consistent with extensive grain growth simulations
presented in this paper. 

\par The paper is organized as follows. In
Section~\ref{sec3}, we discuss important details and
properties of the model for the grain boundary motion from \cite{MR4263432}. In
Sections~\ref{sec4}, we introduce the stochastic model of the grain
boundary system. In Section~\ref{sec5} we establish well-posedness
results of the associated Fokker-Planck equation and obtain the
long-time asymptotic behavior of the solution. In Section~\ref{sec10}
we derive Fokker-Planck type equation for the marginal probability
density of the misorientations and study long-time asymptotics of its solution.
Finally, in Section~\ref{sec13}, we present extensive numerical
experiments to show consistency among the obtained results for the
simplified stochastic model of a grain boundary network and the
results of 2D grain growth simulations based on sharp-interface
approach \cite{Katya-Chun-Mzn2} (and an earlier work \cite{MR2772123,
  MR3729587}), including numerical
investigation of the derived local algebraic relations for an equilibrium configuration of a boundary network.
\section{The Fokker-Planck equation and the fluctuation-dissipation principle}\label{sec2}

In this section, we first derive a Langevin equation, a stochastic
differential equation for the dynamics of misorientations and the triple
junctions using the deterministic model of grain boundary motion
obtained in
\cite{MR4263432} and see also \cite{Katya-Chun-Mzn2}. After that, we establish the fluctuation-dissipation theorem from
associated Fokker-Planck equation. Note, we use below notation $|\cdot|$ for a
standard Euclidean vector norm.

\subsection{Review of the deterministic model and the gradient flow structure}\label{sec3}
First, we review here the deterministic grain boundary motion model from
\cite{MR4263432}. Assume a single triple junction and consider
the following grain boundary energy of the system,
\begin{equation}
 \label{eq:2.1}
  \tilde{E}
  =
  \sum_{j=1}^3
  \sigma(\Delta^{(j)}\alpha)|\Gamma_t^{(j)}|,
\end{equation}
where $\sigma:\mathbb{R}\rightarrow\mathbb{R}$ is a given surface
tension of a grain boundary, $\alpha^{(j)}=\alpha^{(j)}(t):[0,\infty)\rightarrow\mathbb{R}$
is a time-dependent orientation of the grains,
$\theta=\Delta^{(j)}\alpha:=\alpha^{(j-1)}-\alpha^{(j)}$ is a lattice
misorientation of the grain boundary $\Gamma^{(j)}_t$, and
$|\Gamma_t^{(j)}|$ is the length of $\Gamma_t^{(j)}$ for $j=1,2,3$. Here
we put $\alpha^{(0)}=\alpha^{(3)}$. In this work, we assume a grain
boundary energy density $\sigma(\Delta^{(j)}\alpha)$ is only a
function of misorientation $\Delta^{(j)}\alpha$. In addition,  we  assume that $\sigma$ is a $C^1$
function on $\R$.

\begin{figure}[h]
 \centering
 \includegraphics[width=8cm]{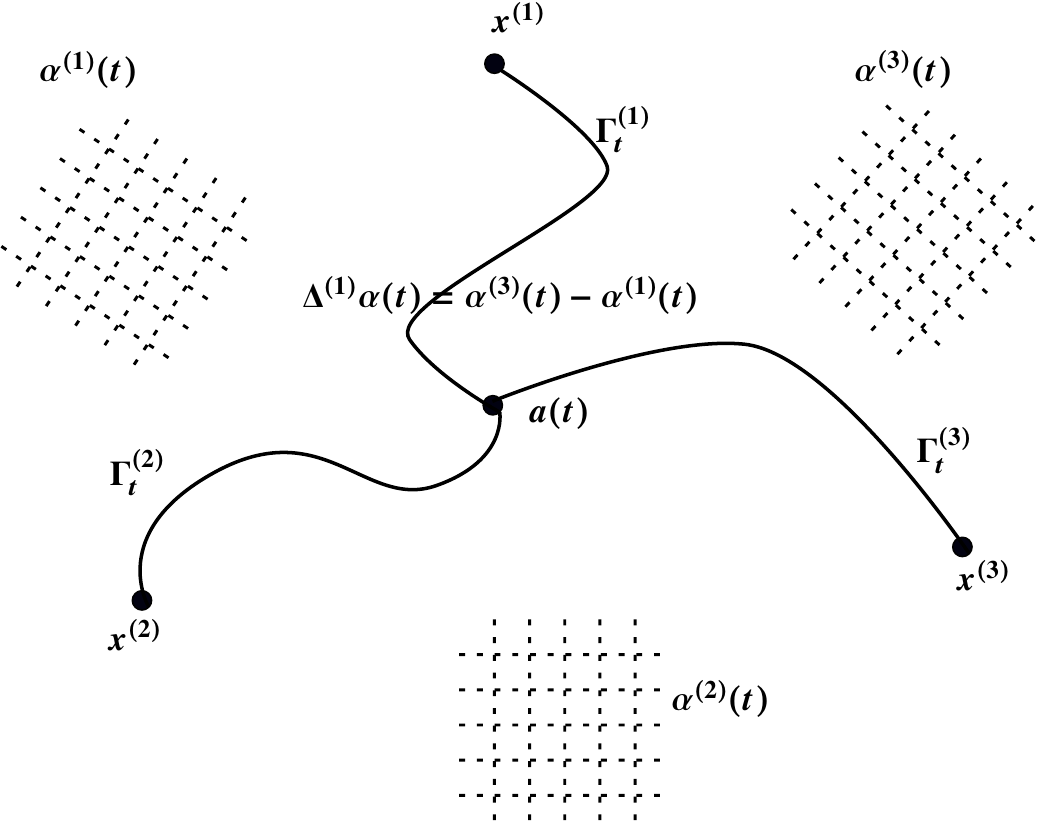}

 \caption{Grain boundaries/curves $\Gamma_t^{(j)}$ that meet at a
 triple junction $\vec{a}(t)$.  Lattice orientations are angles
 (scalars) $\alpha^{(j)}(t)$. Misorientation $\Delta^{(j)}\alpha(t)$
 of $\Gamma^{(j)}_t$ is
 the difference between two lattice orientations of grains
 that share grain boundary $\Gamma^{(j)}_t$.}\label{figTJ2}
\end{figure}

Then, as a result of applying the maximal dissipation principle for
the energy (\ref{eq:2.1}),  the
following model was derived in \cite{MR4263432}, 
\begin{equation}
 \label{eq:2.2}
  \left\{
  \begin{aligned}
   v_n^{(j)}
   &=
   \mu
   \sigma(\Delta^{(j)}\alpha)
   \kappa^{(j)},\quad\text{on}\ \Gamma_t^{(j)},\ t>0,\quad j=1,2,3, \\
   \frac{d\alpha^{(j)}}{dt}
   &=
   -\gamma
   \Bigl(
   \sigma_\theta(\Delta^{(j+1)}\alpha)
   |\Gamma_t^{(j+1)}|
   -
   \sigma_\theta(\Delta^{(j)}\alpha)
   |\Gamma_t^{(j)}|
   \Bigr)
   ,\quad
   j=1,2,3,
   \\
   \frac{d\vec{a}}{dt}(t)
   &=
   -\eta
   \sum_{k=1}^3
   \sigma(\Delta^{(k)}\alpha)
   \frac{\vec{b}^{(k)}}{|\vec{b}^{(k)}|},
   \quad t>0,\ \text{at}\ \vec{a}, \\
   \Gamma_t^{(j)}
   &:
   \vec{\xi}^{(j)}(s,t),\quad
   0\leq s\leq 1,\quad
   t>0,\quad
   j=1,2,3, \\
   \vec{a}(t)
   &=
   \vec{\xi}^{(1)}(1,t)
   =
   \vec{\xi}^{(2)}(1,t)
   =
   \vec{\xi}^{(3)}(1,t),
   \quad
   \text{and}
   \quad
   \vec{\xi}^{(j)}(0,t)=\vec{x}^{(j)},\quad
   j=1,2,3,
  \end{aligned}
  \right.
\end{equation}
where $\Delta^{(4)}\alpha=\Delta^{(1)}\alpha$. In \eqref{eq:2.2},
$v_n^{(j)}$, $\kappa^{(j)}$, and $\vec{b}^{(j)}=\vec{\xi}_s^{(j)}$
denote a normal velocity, a curvature and a tangent vector of the grain
boundary $\Gamma_t^{(j)}$, respectively. Note that $s$ is \emph{not} an
arc length parameter of $\Gamma_t^{(j)}$, namely, $\vec{b}^{(j)}$ is
\emph{not} necessarily a unit tangent vector. The vector
$\vec{a}=\vec{a}(t):[0,\infty)\rightarrow\mathbb{R}^2$ denotes a
position of the triple junction, $\vec{x}^{(j)}$, in a current
context (see also numerical Section \ref{sec13}),  is a position of the
end point of the grain boundary. The three independent relaxation time
scales $\mu,\gamma,\eta>0$ (length, misorientation and position of the
triple junction) are considered in this work as positive
constants.

Recall that  in \cite{MR4263432}, to derive 
\eqref{eq:2.2},   we first computed energy dissipation rate,
\begin{equation}
 \label{eq:2.21}
 \begin{split}
  \frac{d\tilde{E}}{dt}
  &=
  \frac{d}{dt}
  \sum_{j=1}^3
 \int_0^1\sigma(\Delta^{(j)}\alpha(t))|\vec{b}^{(j)}(s,t)|\,ds \\
  &=
  \sum_{j=1}^3
  \int_0^1\sigma(\Delta^{(j)}\alpha(t))\frac{\vec{b}^{(j)}(s,t)}{|\vec{b}^{(j)}(s,t)|}\cdot \vec{\xi}_{ts}^{(j)}(s,t)\,ds \\
  &\quad +
  \sum_{j=1}^3\int_0^1\sigma_\theta(\Delta^{(j)}\alpha(t))(\Delta^{(j)}\alpha(t))_t|\vec{b}^{(j)}(s,t)|\,ds.
 \end{split}
\end{equation}
Using $\vec{\xi}^{(j)}_{ts}=(\vec{\xi}^{(j)}_{t})_s$, integration by
parts, $\vec{\xi}^{(j)}(0,t)=\vec{x}^{(j)}$ is independent of $t$,
and $\vec{\xi}^{(j)}(1,t)=\vec{a}(t)$, we have,
\begin{equation}
 \label{eq:2.22}
 \begin{split}
 &\quad
 \sum_{j=1}^3
  \int_0^1\sigma(\Delta^{(j)}\alpha(t))\frac{\vec{b}^{(j)}(s,t)}{|\vec{b}^{(j)}(s,t)|}\cdot \vec{\xi}_{ts}^{(j)}(s,t)\,ds \\
  &=
  -
  \sum_{j=1}^3
  \sigma(\Delta^{(j)}\alpha(t))
  \int_0^1
  \left(
  \frac{\vec{b}^{(j)}(s,t)}{|\vec{b}^{(j)}(s,t)|}
  \right)_s
  \cdot
  \vec{\xi}_{t}^{(j)}(s,t)\,ds 
  +
  \sum_{j=1}^3
  \sigma(\Delta^{(j)}\alpha(t))
  \frac{\vec{b}^{(j)}(1,t)}{|\vec{b}^{(j)}(1,t)|}
  \cdot \vec{a}_{t}(t).
 \end{split}
\end{equation}
Recall that $\Delta^{(j)}\alpha=\alpha^{(j-1)}-\alpha^{(j)}$, we replace index of sum for the misorientations and we obtain,
\begin{equation}
 \label{eq:2.23}
 \begin{split}
  &\quad
  \sum_{j=1}^3\int_0^1\sigma_\theta(\Delta^{(j)}\alpha(t))(\Delta^{(j)}\alpha(t))_t|\vec{b}^{(j)}(s,t)|\,ds \\
  &=
  \sum_{j=1}^3\int_0^1
  \left(
  \sigma_\theta(\Delta^{(j+1)}\alpha(t))|\Gamma_t^{(j+1)}|
  -
  \sigma_\theta(\Delta^{(j)}\alpha(t))|\Gamma_t^{(j)}|
  \right)
  \alpha_t^{(j)}(t).
 \end{split}
\end{equation}
Inserting \eqref{eq:2.22} and \eqref{eq:2.23} into \eqref{eq:2.21}, we have,
\begin{equation}
 \begin{split}
  \frac{d\tilde{E}}{dt}
    &=
    - 
    \sum_{j=1}^3
    \sigma(\Delta^{(j)}\alpha(t))
    \int_0^1
    \left(
    \frac{\vec{b}^{(j)}(s,t)}{|\vec{b}^{(j)}(s,t)|}
    \right)_s
    \cdot
    \vec{\xi}_{t}^{(j)}(s,t)\,ds 
    +
    \sum_{j=1}^3
    \sigma(\Delta^{(j)}\alpha(t))
    \frac{\vec{b}^{(j)}(1,t)}{|\vec{b}^{(j)}(1,t)|}
    \cdot \vec{a}_{t}(t) \\
    &\quad
    +\sum_{j=1}^3\int_0^1
    \left(
    \sigma_\theta(\Delta^{(j+1)}\alpha(t))|\Gamma_t^{(j+1)}|
    -
    \sigma_\theta(\Delta^{(j)}\alpha(t))|\Gamma_t^{(j)}|
    \right)
    \alpha_t^{(j)}(t).
  \end{split}
\end{equation}
After that,  we ensured that the entire system is dissipative, that is
$\frac{d\tilde{E}}{dt}\leq0$.  We obtained \eqref{eq:2.2} with a
help of the Frenet-Serret formulas and with a help of the energy
dissipation principle which took a form as presented below,
\begin{equation*}
 \frac{d\tilde{E}}{dt}
  =
  -
  \sum_{j=1}^3
  \left(
  \frac{1}{\mu}
  \int_{\Gamma_t^{(j)}}
  |v_n^{(j)}|^2 d\mathcal{H}^1
  +
  \frac{1}{\gamma}
  \left|
   \frac{d\alpha^{(j)}}{dt}(t)
   \right|^2
  \right)
  -
  \frac{1}{\eta}
  \left|
   \frac{d\vec{a}}{dt}
   \right|^2,
\end{equation*}
where $\mathcal{H}^1$ is the $1$-dimensional Hausdorff measure.
More in-depth discussion and complete details of the derivation of the model \eqref{eq:2.2} can be found in our
earlier work \cite[Section 2]{MR4263432}.

Next, in \cite{MR4263432},  we relaxed curvature effect, by taking the limit
$\mu\rightarrow\infty$, and we obtained the reduced model,
\begin{equation}
 \label{eq:2.3}
 \left\{
  \begin{aligned}
   \frac{d\alpha^{(j)}}{dt}
   &=
   -
   \gamma{
   \Bigl(
   \sigma_\theta(\Delta^{(j+1)}\alpha)|\vec{a}(t)-\vec{x}^{(j+1)}|
   -
   \sigma_\theta(\Delta^{(j)}\alpha)|\vec{a}(t)-\vec{x}^{(j)}|
   \Bigr)
   }
   ,
   \quad
   j=1,2,3, \\
   \frac{d\vec{a}}{dt}(t)
   &=
   -
   \eta
   \sum_{j=1}^3
   \sigma(\Delta^{(j)}\alpha)
   \frac{\vec{a}(t)-\vec{x}^{(j)}}{|\vec{a}(t)-\vec{x}^{(j)}|},
   \quad t>0,
  \end{aligned}
 \right.
\end{equation}
where $\vec{x}^{(4)}=\vec{x}^{(1)}$.  Note, in
\cite{MR4263432},  we first applied the maximal
dissipation principle for the grain boundary energy of the system with curvature
\eqref{eq:2.1},  and after that we took the relaxation limit
$\mu\rightarrow\infty$. In fact, as the following proposition shows, these
operations are interchangeable.
\begin{proposition}
 \label{prop:2.1}
 Let $E$ be a relaxation energy associated with  $\tilde{E}$ ($\mu\rightarrow\infty$), given by
 \begin{equation*}
  E
   (\Delta\vec{\alpha},\vec{a})
   =
   \sum_{j=1}^3
   \sigma(\Delta^{(j)}\alpha)|\vec{a}-\vec{x}^{(j)}|.
 \end{equation*}
 Then, equation \eqref{eq:2.3} is a gradient flow associated with the
 energy $E$, namely, we have,
 \begin{equation}
  \label{eq:2.4}
  \begin{split}
   \frac{\delta E}{\delta\alpha^{(j)}}
   &=
   \sigma_\theta(\Delta^{(j+1)}\alpha)|\vec{a}-\vec{x}^{(j+1)}|
   -
   \sigma_\theta(\Delta^{(j)}\alpha)|\vec{a}-\vec{x}^{(j)}|,
   \\
   \frac{\delta E}{\delta\vec{a}}
   &=
   \sum_{j=1}^3
   \sigma(\Delta^{(j)}\alpha)
   \frac{\vec{a}-\vec{x}^{(j)}}{|\vec{a}-\vec{x}^{(j)}|}.
  \end{split}
\end{equation}
\end{proposition}

\begin{proof}
 The first equation of \eqref{eq:2.4} can be obtained by taking the
 derivative of $E$ with respect to $\alpha^{(j)}$. The second equation of \eqref{eq:2.4} can
 be deduced from,
 \begin{equation*}
  \frac{d}{d\varepsilon}\bigg|_{\varepsilon=0}
   E(\Delta\vec{\alpha},\vec{a}+\varepsilon\vec{p})
   =
   \sum_{j=1}^3
   \sigma(\Delta^{(j)}\alpha)
   \frac{\vec{a}-\vec{x}^{(j)}}{|\vec{a}-\vec{x}^{(j)}|}
   \cdot
   \vec{p},
 \end{equation*}
 for $\vec{p}\in\R^2$.
\end{proof}

To analyze the grain boundary motion in this work, Sections \ref{sec4}-\ref{sec10}, it will be convenient
to use the misorientation
$\Delta^{(j)}\alpha$ as a state variable, instead of the orientation
$\alpha^{(j)}$. Thus, from the first equation of \eqref{eq:2.3}, we
can derive,
\begin{equation}
 \label{eq:2.5}
  \begin{split}
   &\quad\frac{d(\Delta^{(j)}\alpha)}{dt} \\
   &=
   -
   \gamma{
   \Bigl(
   2
   \sigma_\theta(\Delta^{(j)}\alpha)|\vec{a}(t)-\vec{x}^{(j)}|
   -
   \sigma_\theta(\Delta^{(j+1)}\alpha)|\vec{a}(t)-\vec{x}^{(j+1)}|
   -
   \sigma_\theta(\Delta^{(j-1)}\alpha)|\vec{a}(t)-\vec{x}^{(j-1)}|
   \Bigr)
   },
  \end{split}
\end{equation}
where $\Delta^{(0)}\alpha=\Delta^{(3)}\alpha$,
$\Delta^{(1)}\alpha=\Delta^{(4)}\alpha$,
$\Delta^{(2)}\alpha=\Delta^{(5)}\alpha$, and $\vec{x}^{(j)}$ is defined
similarly. From the definition of the misorientation,
$\Delta^{(j)}\alpha=\alpha^{(j-1)}-\alpha^{(j)}$, it is easy to find the
constraint,
\begin{equation}
\label{eq:2.6}
\Delta^{(1)}\alpha
  +
  \Delta^{(2)}\alpha
  +
  \Delta^{(3)}\alpha
  =
  0.
\end{equation}
To consider \eqref{eq:2.5} to be a gradient flow, we introduce the
2-dimensional plane,
\begin{equation}
 \Omega
  :=
  \left\{
   \left(
    \Delta^{(1)}\alpha,\Delta^{(2)}\alpha,\Delta^{(3)}\alpha
   \right)
   \in \left(
	-\frac{\pi}{4},\frac\pi4
       \right)^3
   :
   \Delta^{(1)}\alpha+\Delta^{(2)}\alpha+\Delta^{(3)}\alpha=0
  \right\}
  \subset\R^3.
\end{equation}
For planar grain boundary network, it is reasonable to consider such range of the misorientations.
Next, we show that \eqref{eq:2.4} is also a gradient flow of the
energy $E$ with respect to
the misorientation $\Delta^{(j)}\alpha$ and the triple junction $\vec{a}$.

\begin{proposition}
 \label{prop:2.2}
 The system of equations \eqref{eq:2.3} is a gradient flow of the
 energy $E$ with respect to the misorientation $\Delta^{(j)}\alpha$ and
 the triple junction $\vec{a}$.
\end{proposition}

\begin{proof}
 We need to show that the right hand side of \eqref{eq:2.5} is a
gradient of $E$ with respect to misorientation $\Delta\vec{\alpha}$ on
$\Omega$.  Using the fact that one of the normal vectors of $\Omega$ is
$\mathstrut^t(1,1,1)$, the tangential derivative for an arbitrary
function $\phi=\phi\left(
\Delta^{(1)}\alpha,\Delta^{(2)}\alpha,\Delta^{(3)}\alpha \right)$ on
$\Omega$ is given by,
 \begin{equation}
  \nabla^\Omega_{\Delta\vec{\alpha}}
   \phi
   =
   \left(
    \begin{pmatrix}
     1 & 0 & 0 \\
     0 & 1 & 0 \\
     0 & 0 & 1
    \end{pmatrix}
    -
    \frac13
    \begin{pmatrix}
     1 & 1 & 1 \\
     1 & 1 & 1 \\
     1 & 1 & 1
    \end{pmatrix}
   \right)
   \begin{pmatrix}
    \phi_{\Delta^{(1)}\alpha} \\
    \phi_{\Delta^{(2)}\alpha} \\
    \phi_{\Delta^{(3)}\alpha} \\
   \end{pmatrix}
   =
   \frac13
   \begin{pmatrix}
    2 & -1 & -1 \\
    -1 & 2 & -1 \\
    -1 & -1 & 2
   \end{pmatrix}
   \begin{pmatrix}
    \phi_{\Delta^{(1)}\alpha} \\
    \phi_{\Delta^{(2)}\alpha} \\
    \phi_{\Delta^{(3)}\alpha} \\
   \end{pmatrix}.
 \end{equation}
 Thus, we have that,
 \begin{equation}
  \nabla^\Omega_{\Delta\vec{\alpha}} E
   =
   \frac13
   \begin{pmatrix}
    +2\sigma_\theta(\Delta^{(1)}\alpha)|\vec{a}-\vec{x}^{(1)}|
    -\sigma_\theta(\Delta^{(2)}\alpha)|\vec{a}-\vec{x}^{(2)}|
    -\sigma_\theta(\Delta^{(3)}\alpha)|\vec{a}-\vec{x}^{(3)}| \\
    -\sigma_\theta(\Delta^{(1)}\alpha)|\vec{a}-\vec{x}^{(1)}|
    +2\sigma_\theta(\Delta^{(2)}\alpha)|\vec{a}-\vec{x}^{(2)}|
    -\sigma_\theta(\Delta^{(3)}\alpha)|\vec{a}-\vec{x}^{(3)}| \\
    -\sigma_\theta(\Delta^{(1)}\alpha)|\vec{a}-\vec{x}^{(1)}|
    -\sigma_\theta(\Delta^{(2)}\alpha)|\vec{a}-\vec{x}^{(2)}|
    +2\sigma_\theta(\Delta^{(3)}\alpha)|\vec{a}-\vec{x}^{(3)}| \\
   \end{pmatrix}.
 \end{equation}
 Hence,  the equation \eqref{eq:2.5} can be regarded as a gradient
 flow of the
 energy $E$, that is,
 \begin{equation}
  \frac{d}{dt}
   \begin{pmatrix}
   \Delta^{(1)}\alpha \\
   \Delta^{(2)}\alpha \\
   \Delta^{(3)}\alpha
   \end{pmatrix}
   =
   -3\gamma\nabla^\Omega_{\Delta\vec{\alpha}} E.
 \end{equation}
\end{proof}

From Propositions \ref{prop:2.1} and \ref{prop:2.2}, we have that,
\begin{equation*}
 \frac{d(\Delta\vec{\alpha})}{dt}
  =
  -3\gamma\nabla^\Omega_{\Delta\vec{\alpha}} E,\qquad
 \frac{d\vec{a}}{dt}
  =
  -\eta\nabla_{\vec{a}} E,
\end{equation*}
where $\Delta\vec{\alpha}=(\Delta^{(1)}\alpha, \Delta^{(2)}\alpha,
\Delta^{(3)}\alpha)$. Thus,  we obtain the following energy dissipation,
\begin{equation*}
 \frac{dE}{dt}
  =
  -
  \frac{1}{3\gamma}
  \left|
 \frac{d(\Delta\vec{\alpha})}{dt}
  \right|^2
  - 
  \frac{1}{\eta}
  \left|
   \frac{d\vec{a}}{dt}
  \right|^2.
\end{equation*}

\subsection{Stochastic model and the fluctuation-dissipation theorem}\label{sec4}
Many technologically useful materials are polycrystals composed of a
myriad of small monocrystalline grains separated by grain
boundaries. An interaction among the grain boundaries and the triple
junctions in a grain boundary network (including modeling of critical/disappearance events, e.g., grain
disappearance, facet/grain boundary disappearance, facet interchange,
splitting of unstable junctions) is a very complex process. Here, we propose a simplified
stochastic model to develop better understanding of dynamics of misorientations and triple
junctions in a network. In our model, we consider ensemble of triple
junctions and misorientations (without curvature effect), and we use white noise
to describe interactions among them. Therefore, we consider the following Langevin equations, or stochastic
differential equations,
\begin{equation}
 \label{eq:2.9}
 \begin{aligned}
  d(\Delta\vec{\alpha})
  &=
  \vec{v}_{\Delta\vec{\alpha}}\,dt
  +\beta_{\Delta\vec{\alpha}}\,dB,&\qquad
  \vec{v}_{\Delta\vec{\alpha}}
  &=
  -3\gamma\nabla^\Omega_{\Delta\vec{\alpha}} E,
  \\
  d\vec{a}
  &=
  \vec{v}_{\vec{a}}\,dt
  +\beta_{\vec{a}}\,dB,
  &\qquad
  \vec{v}_{\vec{a}}
  &=
  -\eta\frac{\delta E}{\delta \vec{a}}
  =
  -\eta\nabla_{\vec{a}}E.
 \end{aligned}
\end{equation}
Here $B$ denotes a Brownian motion, and $\beta_{\Delta\vec{\alpha}},
\beta_{\vec{a}}>0$ are fluctuation parameters for misorientation
$\Delta\vec{\alpha}$ and triple junction $\vec{a}$,
respectively.  The proposed model \eqref{eq:2.9} can be viewed as a
stochastic analog of the ``vertex model''  \eqref{eq:2.3}. Thus,  the associated probability density function or joint
distribution function of misorientations $\Delta\vec{\alpha}$ and
positions of the triple junctions $\vec{a}$, $f=f(\Delta\vec{\alpha},\vec{a},t)$ obeys the
following Fokker-Planck equation,
\begin{equation}
 \label{eq:2.7}
 \frac{\partial f}{\partial t}
  +
  \nabla^\Omega_{\Delta\vec{\alpha}}
  \cdot
  \left(
   \vec{v}_{\Delta\vec{\alpha}}f
  \right)
  +
  \nabla_{\vec{a}}
  \cdot
  \left(
  \vec{v}_{\vec{a}}f
  \right)
  =
  \frac{\beta^2_{\Delta\vec{\alpha}}}{2}
  \Delta_{\Delta\vec{\alpha}}^{\Omega} f
  +
  \frac{\beta^2_{\vec{a}}}{2}
  \Delta_{\vec{a}} f,
\end{equation}
here
$\Delta_{\Delta\vec{\alpha}}^{\Omega}=\nabla^\Omega_{\Delta\vec{\alpha}}\cdot\nabla^\Omega_{\Delta\vec{\alpha}}$,
and $\Delta_{\vec{a}}$ is the standard Laplacian on
$\Omega_{\mathrm{TJ}}$. Hereafter a bounded domain
$\Omega_{\mathrm{TJ}}\subset\R^2$ denotes the state space for the triple
junction $\vec{a}$. In addition, we impose the natural boundary conditions,
\begin{equation}
 \label{eq:2.12}
 f
  \nabla^\Omega_{\Delta\vec{\alpha}}
  \left(
   \log f
   +
   \frac{6\gamma}{\beta^2_{\Delta\vec{\alpha}}}
   E
  \right)
  \cdot
  \vec{\nu_{\Delta\vec{\alpha}}}
  \bigg|_{\partial\Omega\times\Omega_{\mathrm{TJ}}}
  =0, \quad 
  f
  \nabla_{\vec{a}}
  \left(
  \log f
  +
  \frac{2\eta}{\beta^2_{\vec{a}}}
  E
  \right)
  \cdot
  \vec{\nu_{\vec{a}}}
  \bigg|_{\Omega\times\partial\Omega_{\mathrm{TJ}}}
  =
  0,
\end{equation}
where $\vec{\nu}_{\Delta\vec{\alpha}}$ and $\vec{\nu}_{\vec{a}}$ are an outer unit normal vector to
$\partial\Omega$ and $\partial\Omega_{\mathrm{TJ}}$, respectively. Next,  we state a condition
for the fluctuation parameters  $\beta_{\Delta\vec{\alpha}}$ and 
$\beta_{\vec{a}}$ under which the system described by the
Fokker-Planck equation \eqref{eq:2.7}-\eqref{eq:2.12}
is dissipative.

\begin{theorem}
 Let $f$ be a solution of the Fokker-Planck equation \eqref{eq:2.7}-\eqref{eq:2.12} with
 velocities $\vec{v}_{\Delta\vec{\alpha}}$, $\vec{v}_{\vec{a}}$ as defined in
 \eqref{eq:2.9}. If in addition, the relaxation time scales and the
 fluctuation parameters satisfy the relation,
 \begin{equation}
  \label{eq:2.8}
   \frac{6\gamma}{\beta^2_{\Delta\vec{\alpha}}}
   =
  \frac{2\eta}{\beta^2_{\vec{a}}},
 \end{equation}
 which in turn, determines the parameter $D$ as,
\begin{equation}
 \label{eq:2.10}
  D :=
  \frac{\beta^2_{\Delta\vec{\alpha}}}{6\gamma}
  =
  \frac{\beta^2_{\vec{a}}}{2\eta},
\end{equation}
 then, the  Fokker-Planck equation \eqref{eq:2.7}-\eqref{eq:2.12}  satisfies the following
 energy law,
 \begin{equation}
  \label{eq:2.11}
   \begin{split}
    &\quad
    \frac{d}{dt}
    \iint_{\Omega\times\Omega_{\mathrm{TJ}}}
    (Df\log f+fE)
    \,d\Delta\vec{\alpha}d\vec{a} \\
    &=
    -
    \frac{\beta^2_{\Delta\vec{\alpha}}}{2D}
    \iint_{\Omega\times\Omega_{\mathrm{TJ}}}
    f
    \left|
    \nabla^\Omega_{\Delta\vec{\alpha}}
    \left(
    D\log f
    +
    E
    \right)
    \right|^2
    \,d\Delta\vec{\alpha}d\vec{a}
    -
    \frac{\beta^2_{\vec{a}}}{2D}
    \iint_{\Omega\times\Omega_{\mathrm{TJ}}}
    f
    \left|
    \nabla_{\vec{a}}
    \left(
    D
    \log f
    +
    E
    \right)
    \right|^2
    \,d\Delta\vec{\alpha}d\vec{a}.
   \end{split}
 \end{equation} Here, $\iint_{\Omega\times\Omega_{\mathrm{TJ}}}
    (Df\log f+fE) \,d\Delta\vec{\alpha}d\vec{a}$ represents the (scaled) free energy of the
 Fokker-Planck system \eqref{eq:2.7}-\eqref{eq:2.12}.
\end{theorem}

\begin{proof}
First, we use expression \eqref{eq:2.9} for the velocities $\vec{v}_{\Delta\vec{\alpha}}$ and $\vec{v}_{\vec{a}}$ in the
 Fokker-Planck equation \eqref{eq:2.7}, and we have,
 \begin{equation*}
  \frac{\partial f}{\partial t}
   =
   \frac{\beta^2_{\Delta\vec{\alpha}}}{2}
   \nabla^\Omega_{\Delta\vec{\alpha}}
   \cdot
   \left(
    f
    \nabla^\Omega_{\Delta\vec{\alpha}}
    \left(
     \log f
     +
     \frac{6\gamma}{\beta^2_{\Delta\vec{\alpha}}}
     E
    \right)
   \right)
   +
   \frac{\beta^2_{\vec{a}}}{2}
   \nabla_{\vec{a}}
   \cdot
   \left(
    f
    \nabla_{\vec{a}}
    \left(
     \log f
     +
     \frac{2\eta}{\beta^2_{\vec{a}}}E
    \right)
   \right)
   .
 \end{equation*}
 Next, we multiply the Fokker-Planck equation
 \eqref{eq:2.7} by $D(1+\log f)+E$ and integrate over $\Omega\times\Omega_{\mathrm{TJ}}$.
Note that, 
 \begin{equation}
  \frac{\partial f}{\partial t}
   \left(
    D(1+\log f)+E
   \right)
   =
  \frac{\partial}{\partial t}(Df\log f+fE).
 \end{equation}
 Hence, using the natural boundary conditions \eqref{eq:2.12}, we have,
 \begin{equation}
  \begin{split}
   &\quad
   \frac{d}{dt}
   \iint_{\Omega\times\Omega_{\mathrm{TJ}}}
   (Df\log f+fE)
   \,d\Delta\vec{\alpha}d\vec{a} \\
   &=
   -
   \frac{\beta^2_{\Delta\vec{\alpha}}}{2}
   \iint_{\Omega\times\Omega_{\mathrm{TJ}}}
   \left(
   f
   \nabla^\Omega_{\Delta\vec{\alpha}}
   \left(
   \log f
   +
   \frac{6\gamma}{\beta^2_{\Delta\vec{\alpha}}}
   E
   \right)
   \right)
   \cdot
   \nabla^\Omega_{\Delta\vec{\alpha}}
   \left(
   D\log f+E\right)
   \,d\Delta\vec{\alpha}d\vec{a} \\
   &\qquad
   -
   \frac{\beta^2_{\vec{a}}}{2}
   \iint_{\Omega\times\Omega_{\mathrm{TJ}}}
   \left(
   f
   \nabla_{\vec{a}}
   \left(
   \log f
   +
   \frac{2\eta}{\beta^2_{\vec{a}}}
   E
   \right)
   \right)
   \cdot
   \nabla_{\vec{a}}
   \left(
   D\log f+E\right)
   \,d\Delta\vec{\alpha}d\vec{a}.
  \end{split}
 \end{equation}
 Finally, using \eqref{eq:2.10},
 we have energy dissipation,
 \begin{equation*}
  \begin{split}
   &\quad
   \frac{d}{dt}
   \iint_{\Omega\times\Omega_{\mathrm{TJ}}}
   (Df\log f+fE)
   \,d\Delta\vec{\alpha}d\vec{a} \\
   &=
   -
   \frac{\beta^2_{\Delta\vec{\alpha}}}{2D}
   \iint_{\Omega\times\Omega_{\mathrm{TJ}}}
   f
   \left|
   \nabla^\Omega_{\Delta\vec{\alpha}}
   \left(
   D\log f
   +
   E
   \right)
   \right|^2
   \,d\Delta\vec{\alpha}d\vec{a}
   -
   \frac{\beta^2_{\vec{a}}}{2D}
   \iint_{\Omega\times\Omega_{\mathrm{TJ}}}
   f
   \left|
   \nabla_{\vec{a}}
   \left(
   D
   \log f
   +
   E
   \right)
   \right|^2
   \,d\Delta\vec{\alpha}d\vec{a}.
  \end{split}
 \end{equation*}
\end{proof}

\begin{remark}
 The condition \eqref{eq:2.8} is related to the fluctuation-dissipation
 theorem ~\cite{PhysRev.83.34, Kubo_1966}. The system will
 approach the equilibrium state of the free energy $ \iint_{\Omega\times\Omega_{\mathrm{TJ}}}
   (Df\log f+fE)
   \,d\Delta\vec{\alpha}d\vec{a}$, which
 coincides with the Boltzmann distribution for the grain boundary
 energy $E$,
 \begin{equation}\label{BoltE}
  f_\infty(\Delta\vec{\alpha},\vec{a})
  =
  \Cr{const:3.1}
  \exp
  \left(-\frac{E(\Delta\vec{\alpha},\vec{a})}{D}\right),
 \end{equation}
 for some constant $\Cr{const:3.1}>0$. Relation \eqref{eq:2.8}, which
 is also called
 the {\it fluctuation-dissipation principle}, ensures not only the dissipation
 structure  \eqref{eq:2.11}, but also that the solution of the Fokker-Planck
 equation \eqref{eq:2.7}-\eqref{eq:2.12} converges to the Boltzmann
 distribution \eqref{BoltE}.
\end{remark}

\section{Well-posedness of the Fokker-Planck equation}\label{sec5}
In this section, we study well-posedness of the proposed Fokker-Planck model
\eqref{eq:2.7} under the fluctuation-dissipation relation
\eqref{eq:2.8}, and the natural boundary conditions,
\begin{equation}
 \label{eq:3.8}
 \left\{
  \begin{aligned}
   \frac{\partial f}{\partial t}
   +
   \nabla^\Omega_{\Delta\vec{\alpha}}
   \cdot
   \left(
   \vec{v}_{\Delta\vec{\alpha}}f
   \right)
   +
   \nabla_{\vec{a}}
   \cdot
   \left(
   \vec{v}_{\vec{a}}f
   \right)
   &=
   \frac{\beta^2_{\Delta\vec{\alpha}}}{2}
   \Delta_{\Delta\vec{\alpha}}^{\Omega} f
   +
   \frac{\beta^2_{\vec{a}}}{2}
   \Delta_{\vec{a}} f, \quad \Delta\alpha\in\Omega,\ \vec{a}\in\Omega_{\mathrm{TJ}},\ t>0, \\
   \vec{v}_{\Delta\vec{\alpha}}
   &=
   -3\gamma\nabla^\Omega_{\Delta\vec{\alpha}} E,
   \\
   \vec{v}_{\vec{a}}
   &=
   -\eta\frac{\delta E}{\delta \vec{a}}
   =
   -\eta\nabla_{\vec{a}}E, \\
   \left(
   \frac{\beta^2_{\Delta\vec{\alpha}}}{2}
   \nabla^\Omega_{\Delta\vec{\alpha}}f
   -
   \vec{v}_{\Delta\vec{\alpha}}f
   \right)
   \cdot
   \vec{\nu}_{\Delta\vec{\alpha}}
   \bigg|_{\partial\Omega\times\Omega_{\mathrm{TJ}}}
   &=
   0, \\
   \left(
   \frac{\beta^2_{\vec{a}}}{2}
   \nabla_{\vec{a}}f
   -
   \vec{v}_{\vec{a}}f
   \right)
   \cdot
   \vec{\nu}_{\vec{a}}
   \bigg|_{\Omega\times\partial\Omega_{\mathrm{TJ}}}
   &=
   0, \\
   f(\Delta\vec{\alpha},\vec{a},0)
   &=
   f_0(\Delta\vec{\alpha},\vec{a}).
  \end{aligned}
\right.
\end{equation}
Here, we assume that bounded domain $\Omega_{\mathrm{TJ}}\subset\R^2$
is a domain with
$C^2$ boundary and that,
\begin{equation}
 \Omega
  :=
  \left\{
   \Delta\vec{\alpha}
   =
   \left(
    \Delta^{(1)}\alpha,\Delta^{(2)}\alpha,\Delta^{(3)}\alpha
   \right)
   \in \left(
	-\frac{\pi}{4},\frac\pi4
       \right)^3
   :
   \Delta^{(1)}\alpha+\Delta^{(2)}\alpha+\Delta^{(3)}\alpha=0
  \right\}
  \subset\R^3.
\end{equation}
As in Section \ref{sec2}, the parameters $\beta_{\Delta\vec{\alpha}}$, $\beta_{\vec{a}}$, $\gamma$,
and $\eta$ are positive constants satisfying the fluctuation-dissipation
relation \eqref{eq:2.8}. The vectors $\vec{\nu}_{\Delta\vec{\alpha}}$ and $\vec{\nu}_{\vec{a}}$ are an outer unit normal vector to
$\partial\Omega$ and $\partial\Omega_{\mathrm{TJ}}$, respectively. Recall from Section
\ref{sec2} that the energy of the
system $E$ is given
by,
\begin{equation}
 E(\Delta\vec{\alpha},\vec{a})
  =
  \sum_{j=1}^3
  \sigma(\Delta^{(j)}\alpha)|\vec{a}-\vec{x}^{(j)}|,
\end{equation}
where grain boundary energy density $\sigma$ is a given $C^1$ function and $\vec{x}^{(j)}\in\R^2$ is a
fixed position for $j=1,2,3$. The initial data
$f_0:\Omega\times\Omega_{\mathrm{TJ}}\rightarrow\R$ is assumed to be
positive and,
\begin{equation}
 \label{eq:3.35}
 \int_{\Omega\times\Omega_{\mathrm{TJ}}}
  f_0(\Delta\vec{\alpha},\vec{a})
  \,d\Delta\vec{\alpha}d\vec{a}
  =
  1.
\end{equation}

From the energy law \eqref{eq:2.11}, one can expect
that the asymptotic profile
$f_\infty=f_\infty(\Delta\vec{\alpha},\vec{a})$ of \eqref{eq:3.8} is
given by,
\begin{equation}
\label{eq:3.2}
 f_\infty(\Delta\vec{\alpha},\vec{a})
  =
  \Cl{const:3.1}
  \exp
  \left(-\frac{E(\Delta\vec{\alpha},\vec{a})}{D}\right),
\end{equation}
for some constant $\Cr{const:3.1}>0$, where $D>0$ is defined by
\eqref{eq:2.10}. Since $f$ is a probability density function, the
constant $\Cr{const:3.1}$ satisfies,
\begin{equation}
 \label{eq:3.6}
 \frac{1}{\Cr{const:3.1}}
  =
  \iint_{\Omega\times\Omega_{\mathrm{TJ}}}
  \exp
  \left(-\frac{E(\Delta\vec{\alpha},\vec{a})}{D}\right)
  \,d\Delta\vec{\alpha}d\vec{a}.
\end{equation}
In order to show that solution of the Fokker-Planck equation
\eqref{eq:3.8} converges to $f_\infty$, we will introduce the change of
variable  $g(\Delta\vec{\alpha},\vec{a},t)$,
\begin{equation}
 \label{eq:3.1}
 f(\Delta\vec{\alpha},\vec{a},t)
  =
  g(\Delta\vec{\alpha},\vec{a},t)
  \exp
  \left(-\frac{E(\Delta\vec{\alpha},\vec{a})}{D}\right),
\end{equation}
and we will prove that $g$ converges to the constant $\Cr{const:3.1}$.

It is important to note, that the grain boundary energy $E$ may not belong to
$H^2(\Omega\times\Omega_{\mathrm{TJ}})$,  hence a solution of
\eqref{eq:3.8} will not be smooth in general. Thus, we will introduce the notion
of a weak solution of \eqref{eq:3.8}, similar to (cf.~\cite{MR0241822}).

\begin{definition}
 A function
 $f:\Omega\times\Omega_{\mathrm{TJ}}\times[0,\infty)\rightarrow\R$ is a
 weak solution of \eqref{eq:3.8} if,
 \begin{equation}
  f\in L^\infty(0,\infty;L^2(\Omega\times\Omega_{\mathrm{TJ}}))
   \
   \text{with}
   \
   \nabla^\Omega_{\Delta\vec{\alpha}} f,\ \nabla_{\vec{a}} f
   \in L^2(0,\infty;L^2(\Omega\times\Omega_{\mathrm{TJ}})),
 \end{equation}
 and
 \begin{equation}
  \label{eq:3.9}
   \begin{split}
    &\quad
    \iint_{\Omega\times\Omega_{\mathrm{TJ}}}
    f\phi
    \,d\Delta\vec{\alpha}d\vec{a}
    \bigg|_{t=T}
    -
    \int_0^T\,dt
    \iint_{\Omega\times\Omega_{\mathrm{TJ}}}
    f\phi_t
    \,d\Delta\vec{\alpha}d\vec{a}
    \\
    &\quad
    +
    \int_0^T\,dt
    \iint_{\Omega\times\Omega_{\mathrm{TJ}}}
    \left(
    \left(
    \frac{\beta^2_{\Delta\vec{\alpha}}}{2}
    \nabla^\Omega_{\Delta\vec{\alpha}} f
    -
    \vec{v}_{\Delta\vec{\alpha}}f
    \right)
    \cdot
    \nabla^\Omega_{\Delta\vec{\alpha}}\phi
    +
    \left(
    \frac{\beta^2_{\vec{a}}}{2}
    \nabla_{\vec{a}}f
    -
    \vec{v}_{\vec{a}}f
    \right)
    \cdot
    \nabla_{\vec{a}}\phi
    \right)
    \,d\Delta\vec{\alpha}d\vec{a} \\
    &=
   \iint_{\Omega\times\Omega_{\mathrm{TJ}}}
   f_0\phi
   \,d\Delta\vec{\alpha}d\vec{a}
   \bigg|_{t=0},
   \end{split}
 \end{equation}
 for all $\phi\in
 C^\infty(\overline{\Omega\times\Omega_{\mathrm{TJ}}\times[0,\infty)})$
 and almost every $T>0$.
\end{definition}

We also recall H\"older's inequality \cite[p.77]{MR1814364} and Gronwall's inequality
\cite[Appendix B]{MR1625845} that we will use in our
analysis below.
\begin{lemma}
 [H\"older's inequality] 
 For functions $u\in L^p(\Omega\times\Omega_{\mathrm{TJ}})$,  
 $v\in L^q(\Omega\times\Omega_{\mathrm{TJ}})$, $1/p+1/q=1$, we have that 
 \begin{equation*}
  \iint_{\Omega\times\Omega_{\mathrm{TJ}}}uv
   \,d\Delta\vec{\alpha}d\vec{a} 
   \leq
   \left(
    \iint_{\Omega\times\Omega_{\mathrm{TJ}}}|u|^p
    \,d\Delta\vec{\alpha}d\vec{a} 
   \right)^{1/p}
   \left(
    \iint_{\Omega\times\Omega_{\mathrm{TJ}}}|v|^q
    \,d\Delta\vec{\alpha}d\vec{a} 
   \right)^{1/q}.
 \end{equation*}
\end{lemma}

\begin{lemma}
 [Gronwall's inequality] Let $\zeta(\cdot)$ be a nonnegative, absolutely continuous function on
 $[0,T]$, which satisfies for a.e. t, the differential inequality
\begin{equation*}
 \zeta'(t)
  \leq \phi(t)\zeta(t),
\end{equation*}
 where $\phi(t)$ are summable function on $[0,T]$. Then,
 \begin{equation*}
  \zeta(t)\leq e^{\int_0^t\phi(s)\,ds}\zeta(0),
 \end{equation*}
 for all $0\leq t\leq T$.
\end{lemma}

\subsection{Uniqueness and existence of a weak solution to the Fokker-Planck equation}\label{sec6}

Here, we establish uniqueness and existence of a weak solution to \eqref{eq:3.8}.
First, uniqueness of a weak solution to \eqref{eq:3.8} is
considered. Since the Fokker-Planck equation \eqref{eq:3.8} is linear,
it is enough to deduce that the solution is zero provided that the initial
data is zero.

\begin{proposition}
 Let $f:\Omega\times\Omega_{\mathrm{TJ}}\times[0,\infty)\rightarrow\R$
 be a weak solution of~\eqref{eq:3.8} with $f_0=0$. Assume that $\sigma$
 is a $C^1$ function on $\R$. Then $f=0$ in
 $\Omega\times\Omega_{\mathrm{TJ}}\times[0,\infty)$.
\end{proposition}

\begin{proof}
 We give a formal proof. Take $f$ as a test function for
 \eqref{eq:3.8}, namely,  multiply \eqref{eq:3.8} by $f$ and integrate over
 $\Omega\times\Omega_{\mathrm{TJ}}$.
 Then, using integration by parts and the natural boundary conditions, we
 obtain that,
 \begin{equation}\label{equn1}
  \begin{split}
   \frac12\frac{d}{dt}
   \iint_{\Omega\times\Omega_{\mathrm{TJ}}}
   |f|^2
   \,d\Delta\vec{\alpha}d\vec{a}
   &=
   -
   \frac{\beta^2_{\Delta\vec{\alpha}}}{2}
   \iint_{\Omega\times\Omega_{\mathrm{TJ}}}
   |\nabla^\Omega_{\Delta\vec{\alpha}} f|^2
   \,d\Delta\vec{\alpha}d\vec{a}
   -
   \frac{\beta^2_{\vec{a}}}{2}
   \iint_{\Omega\times\Omega_{\mathrm{TJ}}}
   |\nabla_{\vec{a}} f|^2
   \,d\Delta\vec{\alpha}d\vec{a} \\
   &\quad +
   \iint_{\Omega\times\Omega_{\mathrm{TJ}}}
   f\vec{v}_{\Delta\vec{\alpha}}\cdot\nabla^\Omega_{\Delta\vec{\alpha}} f
   \,d\Delta\vec{\alpha}d\vec{a}
   +
   \iint_{\Omega\times\Omega_{\mathrm{TJ}}}
   f\vec{v}_{\vec{a}}\cdot\nabla_{\vec{a}}f
   \,d\Delta\vec{\alpha}d\vec{a}. \\
    \end{split}
 \end{equation}
To estimate the third and the fourth terms on the right
hand side of \eqref{equn1}, we use Young's
inequality,
\begin{equation}
  \begin{split}
   &\quad
   \iint_{\Omega\times\Omega_{\mathrm{TJ}}}
   f\vec{v}_{\Delta\vec{\alpha}}\cdot\nabla^\Omega_{\Delta\vec{\alpha}} f
   \,d\Delta\vec{\alpha}d\vec{a} \\
   &\leq
   \frac{\beta^2_{\Delta\vec{\alpha}}}{2}
   \iint_{\Omega\times\Omega_{\mathrm{TJ}}}
   |\nabla^\Omega_{\Delta\vec{\alpha}} f|^2
   \,d\Delta\vec{\alpha}d\vec{a}
   +
   \frac{1}{2\beta^2_{\Delta\vec{\alpha}}}
   \iint_{\Omega\times\Omega_{\mathrm{TJ}}}
   |f|^2|\vec{v}_{\Delta\vec{\alpha}}|^2
   \,d\Delta\vec{\alpha}d\vec{a} \\
   &\leq
   \frac{\beta^2_{\Delta\vec{\alpha}}}{2}
   \iint_{\Omega\times\Omega_{\mathrm{TJ}}}
   |\nabla^\Omega_{\Delta\vec{\alpha}} f|^2
   \,d\Delta\vec{\alpha}d\vec{a}
   +
   \frac{\|\vec{v}_{\Delta\vec{\alpha}}\|^2_\infty}{2\beta^2_{\Delta\vec{\alpha}}}
   \iint_{\Omega\times\Omega_{\mathrm{TJ}}}
   |f|^2
   \,d\Delta\vec{\alpha}d\vec{a},
  \end{split} 
 \end{equation}
 and similarly,
 \begin{equation}
  \iint_{\Omega\times\Omega_{\mathrm{TJ}}}
   f\vec{v}_{\vec{a}}\cdot\nabla_{\vec{a}}f
   \,d\Delta\vec{\alpha}d\vec{a}
   \leq
   \frac{\beta^2_{\vec{a}}}{2}
   \iint_{\Omega\times\Omega_{\mathrm{TJ}}}
   |\nabla_{\vec{a}} f|^2
   \,d\Delta\vec{\alpha}d\vec{a}
   +
   \frac{\|\vec{v}_{\vec{a}}\|^2_\infty}{2\beta^2_{\vec{a}}}
   \iint_{\Omega\times\Omega_{\mathrm{TJ}}}
   |f|^2
   \,d\Delta\vec{\alpha}d\vec{a},
 \end{equation}
 where
 $\|\vec{v}_{\Delta\vec{\alpha}}\|_\infty=\sup_{\Omega\times\Omega_{\mathrm{TJ}}}|\vec{v}_{\Delta\vec{\alpha}}|$
 and
 $\|\vec{v}_{\vec{a}}\|_\infty=\sup_{\Omega\times\Omega_{\mathrm{TJ}}}|\vec{v}_{\vec{a}}|$. Thus,
 we have that,
 \begin{equation}
  \frac{d}{dt}
   \iint_{\Omega\times\Omega_{\mathrm{TJ}}}
   |f|^2
   \,d\Delta\vec{\alpha}d\vec{a}
   \leq
   \left(
    \frac{\|\vec{v}_{\Delta\vec{\alpha}}\|^2_\infty}{\beta^2_{\Delta\vec{\alpha}}}
    +
    \frac{\|\vec{v}_{\vec{a}}\|^2_\infty}{\beta^2_{\vec{a}}}
   \right)
   \iint_{\Omega\times\Omega_{\mathrm{TJ}}}
   |f|^2
   \,d\Delta\vec{\alpha}d\vec{a}. 
 \end{equation}
 Therefore,  the assertion of the proposition follows from the
 application of Gronwall's inequality.
\end{proof}

Next, we show existence of a weak solution to the Fokker-Planck
equation \eqref{eq:3.8}. To do that, we use change of variable
\eqref{eq:3.1} and we
derive the equation of $g$. By direct calculation of the derivative of
$f$ and the fluctuation-dissipation relation \eqref{eq:2.8} and
\eqref{eq:2.10}, we obtain,
\begin{equation}\label{eqex1}
 \begin{split}
  \frac{\partial f}{\partial t}
  &=
  \frac{\partial g}{\partial t}
  \exp
  \left(-\frac{E}{D}\right), \\
  \nabla^\Omega_{\Delta\vec{\alpha}}
  \cdot
  \left(
  \vec{v}_{\Delta\vec{\alpha}}f
  \right)
  &=
  \left(
  g
  (   
  \nabla^\Omega_{\Delta\vec{\alpha}}
  \cdot
  \vec{v}_{\Delta\vec{\alpha}}
  )
  +
  \vec{v}_{\Delta\vec{\alpha}}
  \cdot
  \nabla^\Omega_{\Delta\vec{\alpha}} g
  -
  \frac{g}{D}
  \vec{v}_{\Delta\vec{\alpha}}
  \cdot
  \nabla^\Omega_{\Delta\vec{\alpha}} E
  \right)
  \exp
  \left(-\frac{E}{D}\right)  \\
  &=
  \left(
  -3\gamma g
  \Delta_{\Delta\vec{\alpha}}^\Omega E
  -
  3\gamma
  \nabla^\Omega_{\Delta\vec{\alpha}} E
  \cdot
  \nabla^\Omega_{\Delta\vec{\alpha}} g
  +
  \frac{3\gamma g}{D}
  |\nabla^\Omega_{\Delta\vec{\alpha}} E|^2
  \right)
  \exp
  \left(-\frac{E}{D}\right), \\
  \frac{\beta^2_{\Delta\vec{\alpha}}}{2}   
  \Delta_{\Delta\vec{\alpha}}^\Omega f
  &=
  \frac{\beta^2_{\Delta\vec{\alpha}}}{2}
  \left(
  \Delta_{\Delta\vec{\alpha}}^\Omega g
  -
  \frac{2}{D}
  \nabla^\Omega_{\Delta\vec{\alpha}} E
  \cdot
  \nabla^\Omega_{\Delta\vec{\alpha}} g
  -
  \frac{g}{D}
  \Delta_{\Delta\vec{\alpha}}^\Omega E
  +
  \frac{g}{D^2}
  |\nabla^\Omega_{\Delta\vec{\alpha}} E|^2
  \right)
  \exp
  \left(-\frac{E}{D}\right) \\
  &=
  \left(
  3\gamma D
  \Delta_{\Delta\vec{\alpha}}^\Omega g
  -
  6\gamma
  \nabla^\Omega_{\Delta\vec{\alpha}} E
  \cdot
  \nabla^\Omega_{\Delta\vec{\alpha}} g
  -
  3\gamma g
  \Delta_{\Delta\vec{\alpha}}^\Omega E
  +
  \frac{3\gamma g}{D}
  |\nabla^\Omega_{\Delta\vec{\alpha}} E|^2
  \right)
  \exp
  \left(-\frac{E}{D}\right),
 \end{split} 
\end{equation}
Similarly, we have that,
\begin{equation}\label{eqex2}
 \begin{split}
  \nabla_{\vec{a}}
  \cdot
  \left(
  \vec{v}_{\vec{a}}f
  \right)
  &=
  \left(
  -\eta
  g
  \Delta_{\vec{a}} E
  -
  \eta
  \nabla_{\vec{a}} E
  \cdot
  \nabla_{\vec{a}} g
  +
  \frac{\eta g}{D}
  |\nabla_{\vec{a}} E|^2
  \right)
  \exp
  \left(-\frac{E}{D}\right), \\
  \frac{\beta^2_{\vec{a}}}{2}   
  \Delta_{\vec{a}} f
  &=
  \left(
  \eta D \Delta_{\vec{a}} g
  -
  2\eta\nabla_{\vec{a}} E\cdot\nabla_{\vec{a}} g
  -
  \eta g
  \Delta_{\vec{a}} E
  +
  \frac{\eta g}{D}
  |\nabla_{\vec{a}} E|^2
  \right)
  \exp
  \left(-\frac{E}{D}\right).
 \end{split} 
\end{equation}
Thus, using \eqref{eq:3.1}, \eqref{eqex1}-\eqref{eqex2} in the Fokker-Planck equation
\eqref{eq:3.8}, we arrive at the equation for the function $g$,
\begin{equation}
 \label{eq:3.3}
  \frac{\partial g}{\partial t}
  =
  3\gamma D
  \Delta_{\Delta\vec{\alpha}}^{\Omega} g
  +
  \eta D
  \Delta_{\vec{a}} g
  -3\gamma
  \nabla^\Omega_{\Delta\vec{\alpha}}
  E
  \cdot
  \nabla^\Omega_{\Delta\vec{\alpha}}
  g
  -\eta
  \nabla_{\vec{a}}
  E
  \cdot
  \nabla_{\vec{a}}
  g.
\end{equation}
 We also derive the boundary condition for $g$ using expression for
 the boundary conditions \eqref{eq:3.8} for $f$. By direct computation
and the fluctuation-dissipation relation \eqref{eq:2.8} and
\eqref{eq:2.10}, we have that,
\begin{equation}
 \label{eq:3.22}
 \begin{split}
  \frac{\beta^2_{\Delta\vec{\alpha}}}{2}
  \nabla^\Omega_{\Delta\vec{\alpha}}f
  -
  \vec{v}_{\Delta\vec{\alpha}}f
  &=
  \left(
  \frac{\beta^2_{\Delta\vec{\alpha}}}{2}
  \nabla^\Omega_{\Delta\vec{\alpha}}g
  -
  \frac{\beta^2_{\Delta\vec{\alpha}}}{2D}
  g
  \nabla^\Omega_{\Delta\vec{\alpha}}E
  +3\gamma g\nabla^\Omega_{\Delta\vec{\alpha}}E
  \right)
  \exp\left(-\frac{E}{D}\right)
  \\
  &=3\gamma D
  \exp
  \left(-\frac{E}{D}\right)
  \nabla^\Omega_{\Delta\vec{\alpha}}g, \\
  \frac{\beta^2_{\vec{a}}}{2}
  \nabla_{\vec{a}}f
  -
  \vec{v}_{\vec{a}}f
  &=
  \left(
  \frac{\beta^2_{\vec{a}}}{2}
  \nabla_{\vec{a}}g
  -
  \frac{\beta^2_{\vec{a}}}{2D}
  g
  \nabla_{\vec{a}}E
  +\eta g\nabla_{\vec{a}}E
  \right)
  \exp\left(-\frac{E}{D}\right)
  \\
  &=
  \eta D
  \exp
  \left(-\frac{E}{D}\right)
  \nabla_{\vec{a}}g.
 \end{split}
\end{equation}
Thus,  the natural boundary conditions for $f$ is transformed into the Neumann
boundary conditions for $g$.  To study \eqref{eq:3.3}, we introduce
a differential operator,
\begin{equation}
 \label{eq:3.4}
  Lg:=
  3\gamma D
  \Delta_{\Delta\vec{\alpha}}^{\Omega}g
  +
  \eta D
  \Delta_{\vec{a}}g
  -3\gamma
  \nabla^\Omega_{\Delta\vec{\alpha}} 
  E
  \cdot
  \nabla^\Omega_{\Delta\vec{\alpha}} g
  -\eta
  \nabla_{\vec{a}}
  E
  \cdot
  \nabla_{\vec{a}} g
\end{equation}
subject to the Neumann boundary conditions, 
\begin{equation}
 \label{eq:3.25}
\begin{split}
 \nabla^\Omega_{\Delta\vec{\alpha}} g\cdot\vec{\nu}_{\Delta\vec{\alpha}}=0,\qquad
  \text{on}\ 
  \partial\Omega\times\Omega_{\mathrm{TJ}}, \\
  \nabla_{\vec{a}}g\cdot\vec{\nu}_{\vec{a}}=0, \qquad
  \text{on}\ 
  \Omega\times\partial\Omega_{\mathrm{TJ}}, \\
\end{split}
\end{equation}

 We will use the Lax-Milgram theorem below to show
that $L$ is a self-adjoint operator, (cf. \cite[Theorem
5.8]{MR1814364}). For the reader's convenience, we will state theorem
below.
\begin{lemma}
 [Lax-Milgram] \label{lem:3.2}
 Let $\mathbf{B}$ be a bounded, coercive bilinear form on a Hilbert
 space $H$. Then for every bounded linear functional $F\in H^*$, there
 exists a unique element $f\in H$ such that $\mathbf{B}(x,f) = F(x)$ for
 all $x\in H$.
\end{lemma}

Now, we proceed to show that $L$ is a self-adjoint operator on the weighted $L^2$
spaces,
\begin{equation}
 L^2(\Omega\times\Omega_{\mathrm{TJ}},\,d{\it m}),
  \quad
  d{\it m}=e^{-\frac{E}{D}}\,d\Delta\vec{\alpha}d\vec{a}.
\end{equation}

\begin{proposition}
 Let $d{\it m}=e^{-\frac{E}{D}}\,d\Delta\vec{\alpha}d\vec{a}$ and let $L$ be
 defined by \eqref{eq:3.4} on
 $L^2(\Omega\times\Omega_{\mathrm{TJ}},d{\it m})$ with a domain $ D(L) :=
 H^2(\Omega\times\Omega_{\mathrm{TJ}},d{\it m})$, and with the Neumann
 boundary conditions \eqref{eq:3.25}. Then, $L$ is a self-adjoint
 operator on $L^2(\Omega\times\Omega_{\mathrm{TJ}},d{\it m})$.
\end{proposition}

\begin{proof}
 For $g_1$, $g_2\in D(L)$, we have by the definition of $L$, that, 
 \begin{equation}
  \label{eq:3.31}
  \begin{split}
   &\quad
   (Lg_1,g_2)_{L^2(\Omega\times\Omega_{\mathrm{TJ}},d{\it m})} \\
   &=
   3\gamma D
   \iint_{\Omega\times\Omega_{\mathrm{TJ}}}
   \Delta^\Omega_{\Delta\vec{\alpha}} g_1
   g_2\exp\left(-\frac{E}{D}\right)
   \,d\Delta\vec{\alpha}d\vec{a}
   +
   \eta D
   \iint_{\Omega\times\Omega_{\mathrm{TJ}}}
   \Delta_{\vec{a}} g_1
   g_2
   \exp\left(-\frac{E}{D}\right)
   \,d\Delta\vec{\alpha}d\vec{a}
   \\
   &\quad
   -
   3\gamma
   \iint_{\Omega\times\Omega_{\mathrm{TJ}}}
   \nabla^\Omega_{\Delta\vec{\alpha}} E
   \cdot
   \nabla^\Omega_{\Delta\vec{\alpha}} g_1
   g_2\exp\left(-\frac{E}{D}\right)
   \,d\Delta\vec{\alpha}d\vec{a} \\
   &\quad
   -
   \eta 
   \iint_{\Omega\times\Omega_{\mathrm{TJ}}}
   \nabla_{\vec{a}} E
   \cdot
   \nabla_{\vec{a}} g_1
   g_2
   \exp\left(-\frac{E}{D}\right)
   \,d\Delta\vec{\alpha}d\vec{a}.
  \end{split}
 \end{equation}
 Using the integration by parts for the first and the second terms on the
 right hand side of \eqref{eq:3.31}, we obtain, 
 \begin{equation}
    \label{eq:3.31a}
  \begin{split}
   \iint_{\Omega\times\Omega_{\mathrm{TJ}}}
   \Delta^\Omega_{\Delta\vec{\alpha}} g_1
   g_2\exp\left(-\frac{E}{D}\right)
   \,d\Delta\vec{\alpha}d\vec{a}
   &=
   -
   \iint_{\Omega\times\Omega_{\mathrm{TJ}}}
   \nabla^\Omega_{\Delta\vec{\alpha}} g_1
   \cdot
   \nabla^\Omega_{\Delta\vec{\alpha}}
   \left(
   g_2\exp\left(-\frac{E}{D}\right)
   \right)
   \,d\Delta\vec{\alpha}d\vec{a}, \mbox{ and }\\
   \iint_{\Omega\times\Omega_{\mathrm{TJ}}}
   \Delta_{\vec{a}} g_1
   g_2
   \exp\left(-\frac{E}{D}\right)
   \,d\Delta\vec{\alpha}d\vec{a}
   &=
   -
   \iint_{\Omega\times\Omega_{\mathrm{TJ}}}
   \nabla_{\vec{a}} g_1
   \cdot
   \nabla_{\vec{a}}
   \left(
   g_2
   \exp\left(-\frac{E}{D}\right)
   \right)
   \,d\Delta\vec{\alpha}d\vec{a}.
  \end{split}
 \end{equation}
 Since,
 \begin{equation*}
  \begin{split}
   \nabla^\Omega_{\Delta\vec{\alpha}}
   \left(
   g_2\exp\left(-\frac{E}{D}\right)
   \right)
   &=
   \exp\left(-\frac{E}{D}\right)
   \nabla^\Omega_{\Delta\vec{\alpha}}
   g_2
   -
   \frac{g_2}{D}
   \exp\left(-\frac{E}{D}\right)
   \nabla^\Omega_{\Delta\vec{\alpha}}E, \mbox{ and }\\
   \nabla_{\vec{a}}
   \left(
   g_2\exp\left(-\frac{E}{D}\right)
   \right)
   &=
   \exp\left(-\frac{E}{D}\right)
   \nabla_{\vec{a}}
   g_2
   -
   \frac{g_2}{D}
   \exp\left(-\frac{E}{D}\right)
   \nabla_{\vec{a}}E, 
  \end{split} 
 \end{equation*}
 the integrals of $\nabla^\Omega_{\Delta\vec{\alpha}} E \cdot
 \nabla^\Omega_{\Delta\vec{\alpha}} g_1 g_2\exp\left(-\frac{E}{D}\right)
 $ and $\nabla_{\vec{a}} E \cdot \nabla_{\vec{a}} g_1 g_2
 \exp\left(-\frac{E}{D}\right)$ cancel out. Thus, by the definition of
 $d{\it m}$,
 \begin{equation}
  \label{eq:3.21}
  \begin{split}
   &\quad
   (Lg_1,g_2)_{L^2(\Omega\times\Omega_{\mathrm{TJ}},d{\it m})} \\
   &=
   -
   3\gamma D
   \iint_{\Omega\times\Omega_{\mathrm{TJ}}}
   \nabla^\Omega_{\Delta\vec{\alpha}} g_1
   \cdot
   \nabla^\Omega_{\Delta\vec{\alpha}}
   g_2
   \exp\left(-\frac{E}{D}\right)
   \,d\Delta\vec{\alpha}d\vec{a} \\
   &\quad
   -
   \eta D
   \iint_{\Omega\times\Omega_{\mathrm{TJ}}}
   \nabla_{\vec{a}} g_1
   \cdot
   \nabla_{\vec{a}}
   g_2
   \exp\left(-\frac{E}{D}\right)
   \,d\Delta\vec{\alpha}d\vec{a}
   \\
   &=
    -3\gamma D
   \iint_{\Omega\times\Omega_{\mathrm{TJ}}}
   \nabla^\Omega_{\Delta\vec{\alpha}} g_1
   \cdot
   \nabla^\Omega_{\Delta\vec{\alpha}} g_2\,d{\it m}
   -
   \eta D
   \iint_{\Omega\times\Omega_{\mathrm{TJ}}}
   \nabla_{\vec{a}} g_1
   \cdot
   \nabla_{\vec{a}} g_2\,d{\it m},
  \end{split}
 \end{equation}
 hence $L$ is a dissipative symmetric operator.

 Next, we show that $L$ is maximal operator: for fixed $F\in
 L^2(\Omega\times\Omega_{\mathrm{TJ}},\,d{\it m})$, we show that there is
 $g\in H^2(\Omega\times\Omega_{\mathrm{TJ}},\,d{\it m})$, such that
 $-Lg+g=F$.  
 Let us define for $g,\varphi\in
 H^1(\Omega\times\Omega_{\mathrm{TJ}},d{\it m}),$
 \begin{equation*}
  \langle -Lg+g,\varphi\rangle
   :=
   3\gamma D
   \iint_{\Omega\times\Omega_{\mathrm{TJ}}}
   \nabla^\Omega_{\Delta\vec{\alpha}} g
   \cdot
   \nabla^\Omega_{\Delta\vec{\alpha}} \varphi\,d{\it m}
   +
   \eta D
   \iint_{\Omega\times\Omega_{\mathrm{TJ}}}
   \nabla_{\vec{a}} g
   \cdot
   \nabla_{\vec{a}} \varphi\,d{\it m}
   +
   \iint_{\Omega\times\Omega_{\mathrm{TJ}}}
   g\varphi\,d{\it m}.
 \end{equation*}
 By H\"older's inequality and the definition of the Sobolev spaces,  there is a positive constant $\Cr{const:3.12}>0$, such that for,
 $g,\phi\in H^1(\Omega\times\Omega_{\mathrm{TJ}},d{\it m})$, we have that,
 \begin{equation*}
  \begin{split}
   |\langle -Lg+g,\varphi\rangle|
   &\leq
   3\gamma D
   \|\nabla^\Omega_{\Delta\vec{\alpha}}
   g\|_{L^2(\Omega\times\Omega_{\mathrm{TJ}},d{\it m})}
   \|\nabla^\Omega_{\Delta\vec{\alpha}}
   \varphi\|_{L^2(\Omega\times\Omega_{\mathrm{TJ}},d{\it m})} \\
   &\quad
   +
   \eta D
   \|\nabla_{\vec{a}} g\|_{L^2(\Omega\times\Omega_{\mathrm{TJ}},d{\it m})}
   \|\nabla_{\vec{a}}
   \phi\|_{L^2(\Omega\times\Omega_{\mathrm{TJ}},d{\it m})} \\
   &\quad
   +
   \|g\|_{L^2(\Omega\times\Omega_{\mathrm{TJ}},d{\it m} )}
   \|\phi\|_{L^2(\Omega\times\Omega_{\mathrm{TJ}},d{\it m})} \\
   &\leq
   \Cl{const:3.12}
   \|g\|_{H^1(\Omega\times\Omega_{\mathrm{TJ}},d{\it m})}
   \|\varphi\|_{H^1(\Omega\times\Omega_{\mathrm{TJ}},d{\it m)}}.
  \end{split} 
 \end{equation*}
Thus,  $\langle -Lg+g,\varphi\rangle$ is a bounded bilinear form in
 $H^1(\Omega\times\Omega_{\mathrm{TJ}},d{\it m})$. Also, for $g\in
 H^1(\Omega\times\Omega_{\mathrm{TJ}},d{\it m})$, we can obtain,
 \begin{equation*}
  \begin{split}
   \langle -Lg+g,g\rangle
  &=
  3\gamma D
  \|\nabla^\Omega_{\Delta\vec{\alpha}}
  g\|_{L^2(\Omega\times\Omega_{\mathrm{TJ}},d{\it m})}^2
  +
  \eta D
  \|\nabla_{\vec{a}} g\|_{L^2(\Omega\times\Omega_{\mathrm{TJ}},d{\it m)}}^2
  +
  \|g\|^2_{L^2(\Omega\times\Omega_{\mathrm{TJ}},d{\it m)}} \\
  &\geq
  \min\{3\gamma D,\eta D,1\}
   \|g\|^2_{H^1(\Omega\times\Omega_{\mathrm{TJ}},d{\it m})},
  \end{split}
 \end{equation*}
 which shows that $\langle -Lg+g,\varphi\rangle$ is a coercive bilinear
 form. In addition, $F$ can be regarded as a bounded linear functional in
 $H^1(\Omega\times\Omega_{\mathrm{TJ}},d{\it m})$, because for $\varphi\in
 H^1(\Omega\times\Omega_{\mathrm{TJ}},d{\it m})$, we have,
 \begin{equation*}
  |(F,\phi)_{L^2(\Omega\times\Omega_{\mathrm{TJ}},d{\it m})}|
   \leq
   \|F\|_{L^2(\Omega\times\Omega_{\mathrm{TJ}},d{\it m})}
   \|\varphi\|_{L^2(\Omega\times\Omega_{\mathrm{TJ}},d{\it m})}
   \leq
   \|F\|_{L^2(\Omega\times\Omega_{\mathrm{TJ}},d{\it m})}
   \|\varphi\|_{H^1(\Omega\times\Omega_{\mathrm{TJ}},d{\it m})}
 \end{equation*}
 by H\"older's inequality. Thus,  
 by the Lax-Milgram theorem, there exists $g\in
 H^1(\Omega\times\Omega_{\mathrm{TJ}},d{\it m})$ such that,
 \begin{equation*}
 \langle -Lg+g,\varphi\rangle
  =
  (F,\varphi)_{L^2(\Omega\times\Omega_{\mathrm{TJ}},\,d{\it m})}
 \end{equation*}
 for $\varphi\in H^1(\Omega\times\Omega_{\mathrm{TJ}},\,d{\it m})$. 
 Next, for
 arbitrary $\phi\in H^1(\Omega\times\Omega_{\mathrm{TJ}})$, take
 $\varphi=\phi\exp(\frac{E}{D})\in
 H^1(\Omega\times\Omega_{\mathrm{TJ}},\,d{\it m})$. Then, we find that $g$
 is a weak solution of $-Lg+g=F$ with the Neumann boundary condition
 $\nabla^\Omega_{\Delta\vec{\alpha}} g\cdot\vec{\nu}_{\Delta\vec{\alpha}}|_{\partial\Omega\times\Omega_{\mathrm{TJ}}}=0$,  and $\nabla_{\vec{a}}g\cdot\vec{\nu}_{\vec{a}}|_{\Omega\times\partial\Omega_{\mathrm{TJ}}}=0$. In a similar manner as
 for \cite{MR1814364}, we have $g\in
 H^2(\Omega\times\Omega_{\mathrm{TJ}})$. Since $\exp(-\frac{E}{D})$ is
 bounded, $g$ belongs to $H^2(\Omega\times\Omega_{\mathrm{TJ}},d{\it m})$.
\end{proof}

By the semigroup theory (cf. \cite{MR2759829}), for any $g_0\in
L^2(\Omega\times\Omega_{\mathrm{TJ}},d{\it m})$, there uniquely exists
$g\in C([0,\infty);L^2(\Omega\times\Omega_{\mathrm{TJ}},d{\it m}))
\cap C^1((0,\infty);L^2(\Omega\times\Omega_{\mathrm{TJ}},d{\it m}))
\cap C((0,\infty);H^2(\Omega\times\Omega_{\mathrm{TJ}},d{\it m}))$
such that,
\begin{equation}
 \label{eq:3.13}
 \left\{
 \begin{aligned}
  g_t
  &=
  Lg,\quad t>0, \\
  g(0)
  &=
  g_0.
 \end{aligned}
 \right.
\end{equation}
Furthermore $g$ belongs to $C^k((0,\infty);D(L^l))$ for any positive
integer $k,l$. Using the existence of a solution of \eqref{eq:3.13}, one
can obtain existence of a weak solution of \eqref{eq:3.8}.

\begin{proposition}
 Let $f_0\in L^2(\Omega\times\Omega_{\mathrm{TJ}})$. Assume that
 $\sigma$ is a $C^1$ function on $\R$. Then, there exists a weak solution
 $f$ of \eqref{eq:3.8}.
\end{proposition}

\begin{proof}
 Let $g_0=f_0\exp(\frac{E}{D})$. Then $g_0\in
 L^2(\Omega\times\Omega_{\mathrm{TJ}},d{\it m})$ hence there is a solution
 $g\in C([0,\infty);L^2(\Omega\times\Omega_{\mathrm{TJ}},d{\it m})) \cap
 C^1((0,\infty);L^2(\Omega\times\Omega_{\mathrm{TJ}},d{\it m})) \cap
 C((0,\infty);H^2(\Omega\times\Omega_{\mathrm{TJ}},d{\it m}))$ to
 \eqref{eq:3.13}. Then by \eqref{eq:3.21}, for any $\phi\in
 C^\infty(\overline{\Omega\times\Omega_{\mathrm{TJ}}\times[0,\infty)})$
 and almost every $T>0$, we have that,
 \begin{equation}
  \label{eq:3.23}
  \begin{split}
   &\quad
   \iint_{\Omega\times\Omega_{\mathrm{TJ}}}
   g\phi
   \exp
   \left(-\frac{E}{D}\right)
   \,d\Delta\vec{\alpha}d\vec{a}
   \bigg|_{t=T}
   -
   \int_0^T\,dt
   \iint_{\Omega\times\Omega_{\mathrm{TJ}}}
   g\phi_t
   \exp
   \left(-\frac{E}{D}\right)
   \,d\Delta\vec{\alpha}d\vec{a}
   \\
   &\quad
   +
   \int_0^T\,dt
   \iint_{\Omega\times\Omega_{\mathrm{TJ}}}
   \left(
   3\gamma D
   \nabla^\Omega_{\Delta\vec{\alpha}} g
   \cdot
   \nabla^\Omega_{\Delta\vec{\alpha}}\phi
   +
   \eta D
   \nabla_{\vec{a}}g
   \cdot
   \nabla_{\vec{a}}\phi
   \right)
   \exp
   \left(-\frac{E}{D}\right)
   \,d\Delta\vec{\alpha}d\vec{a} \\
   &=
   \iint_{\Omega\times\Omega_{\mathrm{TJ}}}
   g_0\phi
   \exp
   \left(-\frac{E}{D}\right)
   \,d\Delta\vec{\alpha}d\vec{a}
   \bigg|_{t=0}.
  \end{split}
 \end{equation}
 From \eqref{eq:3.31}, \eqref{eq:3.31a}, and \eqref{eq:3.21}
 with $g_1=\phi$, $g_2=g$, we deduce,
 \begin{equation}
  \label{eq:3.33}
   \begin{split}
    &\quad
    \iint_{\Omega\times\Omega_{\mathrm{TJ}}}
    \left(
    3\gamma D
    \nabla^\Omega_{\Delta\vec{\alpha}} g
    \cdot
    \nabla^\Omega_{\Delta\vec{\alpha}}\phi
    +
    \eta D
    \nabla_{\vec{a}}g
    \cdot
    \nabla_{\vec{a}}\phi
    \right)
    \exp
    \left(-\frac{E}{D}\right)
    \,d\Delta\vec{\alpha}d\vec{a} \\
    &=
    -(L\phi,g)_{L^2(\Omega\times\Omega_{\mathrm{TJ}},d{\it m})} \\
    &=
    3\gamma D
    \iint_{\Omega\times\Omega_{\mathrm{TJ}}}
    \nabla^\Omega_{\Delta\vec{\alpha}} \phi
    \cdot
    \nabla^\Omega_{\Delta\vec{\alpha}}
    \left(
    g\exp\left(-\frac{E}{D}\right)
    \right)
    \,d\Delta\vec{\alpha}d\vec{a} \\
    &\quad
    +
    \eta D
    \iint_{\Omega\times\Omega_{\mathrm{TJ}}}
    \nabla_{\vec{a}} \phi
    \cdot
    \nabla_{\vec{a}}
    \left(
    g
    \exp\left(-\frac{E}{D}\right)
    \right)
    \,d\Delta\vec{\alpha}d\vec{a}
    \\
    &\quad
    +
    3\gamma
    \iint_{\Omega\times\Omega_{\mathrm{TJ}}}
    \nabla^\Omega_{\Delta\vec{\alpha}} E
    \cdot
    \nabla^\Omega_{\Delta\vec{\alpha}} \phi
    g\exp\left(-\frac{E}{D}\right)
    \,d\Delta\vec{\alpha}d\vec{a} \\
    &\quad
    +
    \eta 
    \iint_{\Omega\times\Omega_{\mathrm{TJ}}}
    \nabla_{\vec{a}} E
    \cdot
    \nabla_{\vec{a}} \phi
    g
    \exp\left(-\frac{E}{D}\right)
    \,d\Delta\vec{\alpha}d\vec{a}.
   \end{split} 
 \end{equation}
 From the fluctuation-dissipation relation \eqref{eq:2.10}, $3\gamma
 D=\frac{\beta^2_{\Delta\vec{\alpha}}}{2}$ and $\eta
 D=\frac{\beta^2_{\vec{a}}}{2}$, we have that,
\begin{equation}
 \label{eq:3.34}
  \begin{split}
   &\quad
   3\gamma D
   \iint_{\Omega\times\Omega_{\mathrm{TJ}}}
   \nabla^\Omega_{\Delta\vec{\alpha}} \phi
   \cdot
   \nabla^\Omega_{\Delta\vec{\alpha}}
   \left(
   g\exp\left(-\frac{E}{D}\right)
   \right)
   \,d\Delta\vec{\alpha}d\vec{a} \\
   &\quad
   +
   \eta D
   \iint_{\Omega\times\Omega_{\mathrm{TJ}}}
   \nabla_{\vec{a}} \phi
   \cdot
   \nabla_{\vec{a}}
   \left(
   g
   \exp\left(-\frac{E}{D}\right)
   \right)
   \,d\Delta\vec{\alpha}d\vec{a}
   \\
   &\quad
   +
   3\gamma
   \iint_{\Omega\times\Omega_{\mathrm{TJ}}}
   \nabla^\Omega_{\Delta\vec{\alpha}} E
   \cdot
   \nabla^\Omega_{\Delta\vec{\alpha}} \phi
   g\exp\left(-\frac{E}{D}\right)
   \,d\Delta\vec{\alpha}d\vec{a} \\
   &\quad
   +
   \eta 
   \iint_{\Omega\times\Omega_{\mathrm{TJ}}}
   \nabla_{\vec{a}} E
   \cdot
   \nabla_{\vec{a}} \phi
   g
   \exp\left(-\frac{E}{D}\right)
    \,d\Delta\vec{\alpha}d\vec{a} \\
   &=
   \frac{\beta^2_{\Delta\vec{\alpha}}}{2}
   \iint_{\Omega\times\Omega_{\mathrm{TJ}}}
   \nabla^\Omega_{\Delta\vec{\alpha}} \phi
   \cdot
   \nabla^\Omega_{\Delta\vec{\alpha}}
   f
   \,d\Delta\vec{\alpha}d\vec{a}
   +
   \frac{\beta^2_{\vec{a}}}{2}
   \iint_{\Omega\times\Omega_{\mathrm{TJ}}}
   \nabla_{\vec{a}} \phi
   \cdot
   \nabla_{\vec{a}}
   f   
   \,d\Delta\vec{\alpha}d\vec{a} \\
   &\quad
   -
   \iint_{\Omega\times\Omega_{\mathrm{TJ}}}
   \vec{v}_{\Delta\vec{\alpha}}
   \cdot
   \nabla^\Omega_{\Delta\vec{\alpha}} \phi
   f
   \,d\Delta\vec{\alpha}d\vec{a} 
   -
   \iint_{\Omega\times\Omega_{\mathrm{TJ}}}
   \vec{v}_{\vec{a}}
   \cdot
   \nabla_{\vec{a}} \phi
   f    
   \,d\Delta\vec{\alpha}d\vec{a},
  \end{split} 
 \end{equation}
 where we used $\vec{v}_{\Delta\vec{\alpha}} =
 -3\gamma\nabla^\Omega_{\Delta\vec{\alpha}} E$, $\vec{v}_{\vec{a}} =
 -\eta\nabla_{\vec{a}}E$, and $f=g\exp(-E/D)$, \eqref{eq:3.1}. Plugging \eqref{eq:3.33},
 \eqref{eq:3.34}, and $f=g\exp(-E/D)$ again into \eqref{eq:3.23}, we obtain
 that $f$ is a weak solution \eqref{eq:3.9} to \eqref{eq:3.8}.
\end{proof}

\subsection{Exponential decay of $f$}\label{sec7}
We study the long-time asymptotics of the solution $f$ of the
Fokker-Planck equation \eqref{eq:3.8}. In order to derive that $f$
converges to $f_\infty$  \eqref{eq:3.2}, we will show that
$g=f\exp(E/D)$ converges to some constant. Hereafter, we assume the
2-Poincar\'e-Wirtinger inequality on $\Omega\times\Omega_{\mathrm{TJ}}$,
that is, there exists a positive constant $\Cl{const:3.2}>0$ such that
for $g\in C^\infty(\Omega\times\Omega_{\mathrm{TJ}})$,
\begin{equation}
 \label{eq:3.7}
 \iint_{\Omega\times\Omega_{\mathrm{TJ}}}|g-\bar{g}|^2
  \,d\Delta\vec{\alpha}d\vec{a}
  \leq
  \Cr{const:3.2}
  \iint_{\Omega\times\Omega_{\mathrm{TJ}}}
  \left(
   |\nabla^\Omega_{\Delta\vec{\alpha}}g|^2
   +
   |\nabla_{\vec{a}}g|^2
  \right)
  \,d\Delta\vec{\alpha}d\vec{a},
\end{equation}
where
\begin{equation}
 \bar{g}
  =
  \frac{1}{|\Omega\times\Omega_{\mathrm{TJ}}|}
  \iint_{\Omega\times\Omega_{\mathrm{TJ}}}g\,d\Delta\vec{\alpha}d\vec{a}
\end{equation}
is the integral mean on $\Omega\times\Omega_{\mathrm{TJ}}$. For example,
when $\Omega$ is a bounded convex domain, 2-Poincar\'e-Wirtinger
inequality \eqref{eq:3.7} holds~\cite[Lemma
6.12]{MR1465184}.

We now show that $\Omega\times\Omega_{\mathrm{TJ}}$ supports the
2-Poincar\'e-Wirtinger inequality \eqref{eq:3.7} in the weighted $L^2$
space $L^2(\Omega\times\Omega_{\mathrm{TJ}},d{\it m})$.

\begin{lemma}
 \label{lem:3.1}
 There exists $\Cl{const:3.3}>0$ such that for
 $g\in C^\infty(\Omega\times \Omega_{\mathrm{TJ}})$, we have that,
 \begin{equation}
  \label{eq:3.5}
  \|g-\bar{g}_{L^2(\Omega\times\Omega_{\mathrm{TJ}},d{\it
      m})}\|_{L^2(\Omega\times\Omega_{\mathrm{TJ}},d{\it m})}^2
   \leq
   \Cr{const:3.3}
   \iint_{\Omega\times\Omega_{\mathrm{TJ}}}
   \left(
   |\nabla^\Omega_{\Delta\vec{\alpha}}g|^2
   +
   |\nabla_{\vec{a}}g|^2
   \right)
   \,d{\it m},
 \end{equation}
 where a constant $\Cr{const:3.1}$ is defined in \eqref{eq:3.6} and,
 \begin{equation}
  \bar{g}_{L^2(\Omega\times\Omega_{\mathrm{TJ}},d{\it m})}
   =
   \Cr{const:3.1}
   \iint_{\Omega\times\Omega_{\mathrm{TJ}}}g\,d{\it m}.
 \end{equation}
\end{lemma}

\begin{proof}
 We let
 \begin{equation}
  \label{eq:3.32}
  \Cl{const:3.4}
   =
   \inf_{(\Delta\vec{\alpha},\vec{\alpha})}e^{-\frac{E(\Delta\vec{\alpha},\vec{a})}{D}},
   \quad
   \Cl{const:3.5}
   =
   \sup_{(\Delta\vec{\alpha},\vec{\alpha})}e^{-\frac{E(\Delta\vec{\alpha},\vec{a})}{D}},
 \end{equation}
 so that $\Cr{const:3.4}\leq e^{-\frac{E}{D}}\leq \Cr{const:3.5}$ on
 $\Omega\times\Omega_{\mathrm{TJ}}$. Thus, for $g\in
 C^\infty(\Omega\times\Omega_{\mathrm{TJ}})$, we obtain that,
 \begin{equation}
  \begin{split}
   \|g-\bar{g}_{L^2(\Omega\times\Omega_{\mathrm{TJ}},d{\it m})}
   \|^2_{L^2(\Omega\times\Omega_{\mathrm{TJ}},d{\it m})}
   &\leq
   \Cr{const:3.5}
   \iint_{\Omega\times\Omega_{\mathrm{TJ}}}
   |g-\bar{g}|^2\,d\Delta\vec{\alpha}d\vec{a} \\
   &\leq
   \Cr{const:3.5}
   \Cr{const:3.2}
   \iint_{\Omega\times\Omega_{\mathrm{TJ}}}
   \left(
   |\nabla^\Omega_{\Delta\vec{\alpha}}g|^2
   +
   |\nabla_{\vec{a}}g|^2
   \right)
   \,d\Delta\vec{\alpha}d\vec{a} \\
   &\leq
   \frac{\Cr{const:3.5}\Cr{const:3.2}}{\Cr{const:3.4}}
   \iint_{\Omega\times\Omega_{\mathrm{TJ}}}
   \left(
   |\nabla^\Omega_{\Delta\vec{\alpha}}g|^2
   +
   |\nabla_{\vec{a}}g|^2
   \right)
   \,d{\it m}.
  \end{split}
\end{equation}
 The inequality \eqref{eq:3.5} holds for
\begin{equation}\label{eqc3}
 \Cr{const:3.3}=\frac{\Cr{const:3.5}\Cr{const:3.2}}{\Cr{const:3.4}}.
\end{equation}
\end{proof}

Now we are in position to derive the long-time asymptotic behavior for
the solution of the Fokker-Planck equation \eqref{eq:3.8}.

\begin{theorem}
 \label{thm:3.1}
 Assume that $\sigma$ is a $C^1$ function on $\R$ and
 $\Omega\times\Omega_{\mathrm{TJ}}$ supports the 2-Poincar\'e-Wirtinger
 inequality \eqref{eq:3.7}. Let $f_0\in
 L^2(\Omega\times\Omega_{\mathrm{TJ}},e^{\frac{E}{D}}\,d\Delta\vec{\alpha}d\vec{a})$
 be a probability density function. Then, there exists a constant
 $\Cl{const:3.6}>0$ such that the associated solution $f$ of
 \eqref{eq:3.8} satisfies,
 \begin{equation}
  \label{eq:3.19}
   \iint_{\Omega\times\Omega_{\mathrm{TJ}}}
   |f(\Delta\vec{\alpha},\vec{a},t)-f_\infty(\Delta\vec{\alpha},\vec{a})|^2\,\exp \left(\frac{E(\Delta\vec{\alpha},\vec{a})}{D}\right)
   \,d\Delta\vec{\alpha}d\vec{a}
   \leq
   \Cr{const:3.6}
   e^{-\frac{2\min\{3\gamma,\eta\}D}{\Cr{const:3.3}}t}
 \end{equation}
 for $t>0$, where $f_\infty$, $\Cr{const:3.1}$, and $\Cr{const:3.3}$ are
 defined in \eqref{eq:3.2}, \eqref{eq:3.6}, and \eqref{eqc3}, 
 respectively.
\end{theorem}

\begin{proof}
 We multiply \eqref{eq:3.3}  by
 $(g-\bar{g}_{L^2(\Omega\times\Omega_{\mathrm{TJ}},d{\it m})})\exp
 \left(-\frac{E}{D}\right)$ and integrate over
 $\Omega\times\Omega_{\mathrm{TJ}}$, we obtain that,
 \begin{equation*}
  \frac{1}{2}
   \frac{d}{dt}
   \iint_{\Omega\times\Omega_{\mathrm{TJ}}}
   |g-\bar{g}_{L^2(\Omega\times\Omega_{\mathrm{TJ}},d{\it
       m})}|^2\,d{\it m}
   =
   (Lg,g-\bar{g}_{L^2(\Omega\times\Omega_{\mathrm{TJ}},d{\it
       m})})_{L^2(\Omega\times\Omega_{\mathrm{TJ}},d{\it m})}.
 \end{equation*}
 By \eqref{eq:3.21} we get,
 \begin{equation*}
  (Lg,g-\bar{g}_{L^2(\Omega\times\Omega_{\mathrm{TJ}},d{\it
      m})})_{L^2(\Omega\times\Omega_{\mathrm{TJ}},d{\it m})}
   =
   -3\gamma D
   \iint_{\Omega\times\Omega_{\mathrm{TJ}}}
   |\nabla_{\Delta\vec{\alpha}}^\Omega g|^2
   \,d{\it m}
   -
   \eta D
   \iint_{\Omega\times\Omega_{\mathrm{TJ}}}
   |\nabla_{\vec{a}} g|^2
   \,d{\it m}. 
 \end{equation*}
 Combining the above relations with the Poincar\'e inequality
 \eqref{eq:3.5}, we have that,
\begin{equation}
 \frac{1}{2}
  \frac{d}{dt}
  \iint_{\Omega\times\Omega_{\mathrm{TJ}}}
  |g-\bar{g}_{L^2(\Omega\times\Omega_{\mathrm{TJ}},d{\it
      m})}|^2\,d{\it m}
  \leq
  -\frac{\min\{3\gamma,\eta\}D}{\Cr{const:3.3}}
  \|g-\bar{g}_{L^2(\Omega\times\Omega_{\mathrm{TJ}},d{\it
      m})}\|_{L^2(\Omega\times\Omega_{\mathrm{TJ}},d{\it m})}^2. 
 \end{equation}
 Therefore, by Gronwall's inequality, we deduce that,
 \begin{equation}\label{eqest1}
  \iint_{\Omega\times\Omega_{\mathrm{TJ}}}
   |g-\bar{g}_{L^2(\Omega\times\Omega_{\mathrm{TJ}},d{\it
       m})}|^2\,d{\it m}
   \leq
   e^{-2\frac{\min\{3\gamma,\eta\}D}{\Cr{const:3.3}}t}
   \iint_{\Omega\times\Omega_{\mathrm{TJ}}}
  |g_0-\bar{g}_{L^2(\Omega\times\Omega_{\mathrm{TJ}},d{\it
      m})}|^2\,d{\it m}
  =:
  \Cr{const:3.6}
  e^{-\frac{2\min\{3\gamma,\eta\}D}{\Cr{const:3.3}}t}
 \end{equation}
 where $g_0=f_0\exp(-E/D)$.  Using that,  $g=f\exp(E/D)$,  we have,
 \begin{equation}
  \label{eq:3.36}
   \iint_{\Omega\times\Omega_{\mathrm{TJ}}}
   |g-\bar{g}_{L^2(\Omega\times\Omega_{\mathrm{TJ}},d{\it
       m})}|^2\,d{\it m}
   =
   \iint_{\Omega\times\Omega_{\mathrm{TJ}}}
   \left|
    f-\bar{g}_{L^2(\Omega\times\Omega_{\mathrm{TJ}},d{\it m})}
    \exp\left(-\frac{E}{D}\right)
   \right|^2
   \exp\left(\frac{E}{D}\right)\,d\Delta\vec{\alpha}d\vec{a}.
 \end{equation} 
 Integrating \eqref{eq:3.8} on
 $\Omega\times\Omega_{\mathrm{TJ}}$, applying the integration by parts
 and using boundary conditions \eqref{eq:3.8}, we obtain that,
 \begin{equation}
  \label{eq:3.37}
  \frac{d}{dt}
   \iint_{\Omega\times\Omega_{\mathrm{TJ}}}
   f\,d\Delta\vec{\alpha}d\vec{a}
   =
   \iint_{\Omega\times\Omega_{\mathrm{TJ}}}
   \frac{\partial f}{\partial t}\,d\Delta\vec{\alpha}d\vec{a}
   =
   0.
 \end{equation}
 Hence,  due to the assumption on the initial data \eqref{eq:3.35}, it
 follows that,
 \begin{equation*}
   \iint_{\Omega\times\Omega_{\mathrm{TJ}}}
    f(t,\Delta\vec{\alpha},\vec{a})\,d\Delta\vec{\alpha}d\vec{a}
    =
    \iint_{\Omega\times\Omega_{\mathrm{TJ}}}
    f_0(\Delta\vec{\alpha},\vec{a})\,d\Delta\vec{\alpha}d\vec{a}
    =
    1,
 \end{equation*}
 for $t>0$.  Since, $f=g\exp(-E/D)$ and
 $d{\it m}=\exp(-E/D)d\Delta\vec{\alpha}d\vec{a}$,  we have that,
 \begin{equation}\label{eqest2}
  \bar{g}_{L^2(\Omega\times\Omega_{\mathrm{TJ}},d{\it m})}
   =
   \Cr{const:3.1}
   \iint_{\Omega\times\Omega_{\mathrm{TJ}}}g\,d{\it m}
   =
   \Cr{const:3.1}
   \iint_{\Omega\times\Omega_{\mathrm{TJ}}}f\,d\Delta\vec{\alpha}d\vec{a}
   =
   \Cr{const:3.1}.
 \end{equation}
 Combining \eqref{eqest1}, \eqref{eqest2} and
 $f_\infty=\Cr{const:3.1}\exp(-E/D)$, we obtain \eqref{eq:3.19}.
\end{proof}

\subsection{Exponential decay for $f_t$}\label{sec8}
Next, we study finer asymptotics of the solution $f$ of the
Fokker-Planck equation \eqref{eq:3.8}. Due to the properties of
self-adjointness of $L$, the solution $f$ is smooth in time even though
$f$ may not be smooth in space. Thus, we consider long-time asymptotic
behavior of $f_t$.

\begin{theorem}
 \label{thm:3.2}
 Assume that $\sigma$ is a $C^1$ function on $\R$ and
 $\Omega\times\Omega_{\mathrm{TJ}}$ supports the 2-Poincar\'e-Wirtinger
 inequality \eqref{eq:3.7}. Let $f_0\in
 L^2(\Omega\times\Omega_{\mathrm{TJ}},e^{\frac{E}{D}}\,d\Delta\vec{\alpha}d\vec{a})$
 be a probability density function. Then, for any $t_0>0$, there exists a constant
 $\Cl{const:3.8}>0$,  such that the associated solution $f$ of
 \eqref{eq:3.8} satisfies,
 \begin{equation}
  \label{eq:3.20}
   \iint_{\Omega\times\Omega_{\mathrm{TJ}}}
   |f_t(\Delta\vec{\alpha},\vec{a},t)|^2
   \exp \left(\frac{E(\Delta\vec{\alpha},\vec{a})}{D}\right)
   \,d\Delta\vec{\alpha}d\vec{a}
   \leq
   \Cr{const:3.8}
   e^{-\frac{2\min\{3\gamma,\eta\}D}{\Cr{const:3.3}}t}
 \end{equation}
 for $ t>t_0$, where $\Cr{const:3.3}$ is a constant defined in \eqref{eqc3}.
\end{theorem}

\begin{proof}
The equation $g_t=Lg$, \eqref{eq:3.3}, \eqref{eq:3.4} can be written as,
\begin{equation}
 \label{eq:3.10}
 \exp\left(-\frac{E}{D}\right)
  g_t
  =
  3\gamma D
  \nabla^\Omega_{\Delta\vec{\alpha}}\cdot
  \left(
  \exp\left(-\frac{E}{D}\right)
  \nabla^\Omega_{\Delta\vec{\alpha}} g
  \right)
  +
  \eta D
  \nabla_{\vec{a}}\cdot
  \left(
  \exp\left(-\frac{E}{D}\right)
  \nabla_{\vec{a}} g
  \right).
\end{equation}
Note that,  $E(\Delta\vec{\alpha}, \vec{a})$ is a function of only 
 misorientations and the positions of the triple junctions. Take a derivative in time of
 \eqref{eq:3.10}, then,
 \begin{equation}
  \label{eq:3.38}
  \exp\left(-\frac{E}{D}\right)
  g_{tt}
  =
  3\gamma D
  \nabla^\Omega_{\Delta\vec{\alpha}}\cdot
  \left(
   \exp\left(-\frac{E}{D}\right)
   \nabla^\Omega_{\Delta\vec{\alpha}} g_t
  \right)
  +
  \eta D
  \nabla_{\vec{a}}\cdot
  \left(
   \exp\left(-\frac{E}{D}\right)
  \nabla_{\vec{a}} g_t
  \right).
 \end{equation}
 Multiplying \eqref{eq:3.38} by  $g_t$, integrating over
 $\Omega\times\Omega_{\mathrm{TJ}}$, integrating by parts and using
 the boundary conditions \eqref{eq:3.25}, it follows that,
 \begin{equation}
  \label{eq:3.11}
   \frac12
   \frac{d}{dt}
   \iint_{\Omega\times\Omega_{\mathrm{TJ}}}
   |g_t|^2
   \,d{\it m}
   =
   -
   \iint_{\Omega\times\Omega_{\mathrm{TJ}}}
   \left(
    3\gamma D
    |\nabla^\Omega_{\Delta\vec{\alpha}} g_t|^2
    +\eta D
    |\nabla_{\vec{a}} g_t|^2
   \right) 
   \,d{\it m}.
 \end{equation}
 Next, note that,
 \begin{equation*}
  \iint_{\Omega\times\Omega_{\mathrm{TJ}}}
   g_t
   \,d{\it m}
   =
   0,
 \end{equation*}
 thus, we obtain by the Poincar\'e inequality,
 \begin{equation*}
  \iint_{\Omega\times\Omega_{\mathrm{TJ}}}
   \left(
    3\gamma D
    |\nabla^\Omega_{\Delta\vec{\alpha}} g_t|^2
    +\eta D
    |\nabla_{\vec{a}} g_t|^2
   \right) 
   \,d{\it m}
   \geq
   \frac{\min\{3\gamma, \eta\}}{\Cr{const:3.3}}
   \iint_{\Omega\times\Omega_{\mathrm{TJ}}}
   |g_t|^2
   \,d{\it m}.
 \end{equation*}
 Hence,  one can obtain from \eqref{eq:3.11} that,
 \begin{equation*}
  \frac12
   \frac{d}{dt}
   \iint_{\Omega\times\Omega_{\mathrm{TJ}}}
   |g_t|^2
   \,d{\it m}
   \leq
   -
   \frac{\min\{3\gamma, \eta\}}{\Cr{const:3.3}}
   \iint_{\Omega\times\Omega_{\mathrm{TJ}}}
   |g_t|^2
   \,d{\it m}.
 \end{equation*} 
 Thus,  by Gronwall's inequality,
 \begin{equation}
  \label{eq:3.12}
   \iint_{\Omega\times\Omega_{\mathrm{TJ}}}
   |g_t(\Delta\vec{\alpha},\vec{a},t)|^2\,d{\it m}
   \leq
   \left(
   \iint_{\Omega\times\Omega_{\mathrm{TJ}}}
   |g_t(\Delta\vec{\alpha},\vec{a},t_0)|^2\,d{\it m}
   \right)
   e^{-\frac{2\min\{3\gamma,\eta\}D}{\Cr{const:3.3}}t}   
   =:
   \Cr{const:3.8}
   e^{-\frac{2\min\{3\gamma,\eta\}D}{\Cr{const:3.3}}t}.
 \end{equation}
 for $t>t_0$.
 Note again, that $f=g\exp(-E/D)$,
 $d{\it m}=\exp(-E/D)\,d\Delta\vec{\alpha}d\vec{a}$, and,
 \begin{equation}
  \iint_{\Omega\times\Omega_{\mathrm{TJ}}}
   |g_t|^2\,d{\it m}
   =
  \iint_{\Omega\times\Omega_{\mathrm{TJ}}}
   |f_t|^2\exp\left(
	       \frac{E}{D}
	      \right)
   \,d\Delta\vec{\alpha}d\vec{a}.
 \end{equation}
 Therefore, the estimate \eqref{eq:3.20} follows.
\end{proof}

\subsection{Exponential decay for the gradient of $f$}\label{sec9}
 Here we establish the exponential decay for the gradient of $f$. To derive
the asymptotics of the gradient of $f$, one may consider the
equation for the derivative of $f$. However, we cannot take a space
derivative of the Fokker-Planck equation \eqref{eq:3.8}, because of lack of
regularity for the solution $f$. Nevertheless, from the exponential
decay for $f_t$ in Theorem \ref{thm:3.2}, one can obtain a long time
asymptotics for the gradient of $f$.

\begin{theorem}
 \label{thm:3.3} 
 Assume that $\sigma$ is a $C^1$ function on $\R$ and
 $\Omega\times\Omega_{\mathrm{TJ}}$ supports the 2-Poincar\'e-Wirtinger
 inequality \eqref{eq:3.7}. Let $f_0\in
 L^2(\Omega\times\Omega_{\mathrm{TJ}},e^{\frac{E}{D}}\,d\Delta\vec{\alpha}d\vec{a})$
 be a probability density function. Then, for any $t_0>0$, there is a constant
 $\Cl{const:3.9}>0$,  such that the associated solution $f$ of
 \eqref{eq:3.8} satisfies,
 \begin{equation}
  \label{eq:3.16}
  \iint_{\Omega\times\Omega_{\mathrm{TJ}}}
   \left(
    3\gamma D
    |\nabla^\Omega_{\Delta\vec{\alpha}} (f-f_\infty)|^2
    +
    \eta D
    |\nabla_{\vec{a}}(f-f_\infty)|^2
   \right)
   \exp \left(\frac{E(\Delta\vec{\alpha},\vec{a})}{D}\right)
   \,d\Delta\vec{\alpha}d\vec{a}
   \leq
   \Cr{const:3.9}
   e^{-\frac{2\min\{3\gamma,\eta\}D}{\Cr{const:3.3}}t}
 \end{equation}
 for $ t>t_0$, where $\Cr{const:3.3}$ is a constant \eqref{eqc3}.
\end{theorem}

\begin{proof}
 Multiplying \eqref{eq:3.10} by $g_t$, integrating over
 $\Omega\times\Omega_{\mathrm{TJ}}$, and using the integration by parts with
 the boundary conditions~\eqref{eq:3.25}, one can show,
 \begin{equation}
  \begin{split}
   &\quad
   \iint_{\Omega\times\Omega_{\mathrm{TJ}}}
   |g_t|^2\,d{\it m}  \\
   &=
   \iint_{\Omega\times\Omega_{\mathrm{TJ}}}
   \left(
   3\gamma D
   \nabla^\Omega_{\Delta\vec{\alpha}}\cdot
   \left(
   \exp\left(-\frac{E}{D}\right)
   \nabla^\Omega_{\Delta\vec{\alpha}} g
   \right)
   +
   \eta D
   \nabla_{\vec{a}}\cdot
   \left(
  \exp\left(-\frac{E}{D}\right)
   \nabla_{\vec{a}} g
   \right)  
   \right)
   g_t
   \,d\Delta\vec{\alpha}d\vec{a} \\
   &=
   -
   \iint_{\Omega\times\Omega_{\mathrm{TJ}}}
   \left(
   3\gamma D
   \left(
   \exp\left(-\frac{E}{D}\right)
   \nabla^\Omega_{\Delta\vec{\alpha}} g
   \right)
   \cdot
   \nabla^\Omega_{\Delta\vec{\alpha}} g_t
   +
   \eta D
   \left(
   \exp\left(-\frac{E}{D}\right)
   \nabla_{\vec{a}} g
   \right)  
   \cdot
   \nabla_{\vec{a}}
   g_t
   \right)
   \,d\Delta\vec{\alpha}d\vec{a}. 
  \end{split}
 \end{equation}
 On the other hand, by direct computation and
 $d{\it m}=e^{-E/D}\,d\Delta\vec{\alpha}d\vec{a}$, we have,
 \begin{equation*}
  \begin{split}
   &\quad
   \frac12
   \frac{d}{dt}
   \iint_{\Omega\times\Omega_{\mathrm{TJ}}}
   \left(
   3\gamma D
   |\nabla^\Omega_{\Delta\vec{\alpha}} g|^2
   +
   \eta D
   |\nabla_{\vec{a}}g|^2
   \right)
   \,d{\it m }\\
   &=
  \iint_{\Omega\times\Omega_{\mathrm{TJ}}}
  \left(
  3\gamma D
  \left(
  \exp\left(-\frac{E}{D}\right)
  \nabla^\Omega_{\Delta\vec{\alpha}} g
  \right)
  \cdot
  \nabla^\Omega_{\Delta\vec{\alpha}} g_t
  +
  \eta D
  \left(
  \exp\left(-\frac{E}{D}\right)
  \nabla_{\vec{a}} g
  \right)  
  \cdot
  \nabla_{\vec{a}}
  g_t
  \right)
  \,d\Delta\vec{\alpha}d\vec{a}.
  \end{split}
 \end{equation*}
Thus,  we arrive at,
 \begin{equation*}
  \begin{split}
   \iint_{\Omega\times\Omega_{\mathrm{TJ}}}
   |g_t|^2\,d{\it m }
   =  
   -
   \frac12
   \frac{d}{dt}
   \iint_{\Omega\times\Omega_{\mathrm{TJ}}}
   \left(
   3\gamma D
   |\nabla^\Omega_{\Delta\vec{\alpha}} g|^2
   +
   \eta D
   |\nabla_{\vec{a}}g|^2
   \right)
   \,d{\it m}.
  \end{split}
 \end{equation*}
 Using \eqref{eq:3.12} and non-negativity of the integral of $|g_t|^2$,
 one can obtain, for $t>t_0$,
 \begin{equation}
 \label{eq:3.17}
 \begin{split}
  -2\Cr{const:3.8}
  e^{-\frac{2\min\{3\gamma,\eta\}D}{\Cr{const:3.3}}t}
  \leq
  \frac{d}{dt}
  \iint_{\Omega\times\Omega_{\mathrm{TJ}}}
  \left(
  3\gamma D
  |\nabla^\Omega_{\Delta\vec{\alpha}} g|^2
  +
  \eta D
  |\nabla_{\vec{a}}g|^2
  \right)
  \,d{\it m}
  \leq 
  0.
 \end{split}
 \end{equation}
 Specifically, $3\gamma D\|\nabla^\Omega_{\Delta\vec{\alpha}}
 g\|^2_{L^2(\Omega\times\Omega_{\mathrm{TJ}},d{\it m})} + \eta
 D\|\nabla_{\vec{a}} g\|^2_{L^2(\Omega\times\Omega_{\mathrm{TJ}},d{\it m})}$
 is monotone decreasing in time. On the other hand, multiplying
 \eqref{eq:3.10} by  $g$, integrating by parts and using the boundary conditions
 \eqref{eq:3.25}, we have,
 \begin{equation}
  \frac{d}{dt}\iint_{\Omega\times\Omega_{\mathrm{TJ}}}
   |g|^2
   \,d{\it m}
   +
   \iint_{\Omega\times\Omega_{\mathrm{TJ}}}
   \left(
    3\gamma D
    |\nabla^\Omega_{\Delta\vec{\alpha}} g|^2
    +
    \eta D
    |\nabla_{\vec{a}}g|^2
   \right)
   \,d{\it m}
   =
   0.
 \end{equation}
 Now,  integrating over $0\leq t\leq T$ for $T>0$, we arrive at,
 \begin{equation}
 \iint_{\Omega\times\Omega_{\mathrm{TJ}}}
  |g|^2
  \,d{\it m}
  \bigg|_{t=T}
  +
  \int_0^T
  \,dt
   \iint_{\Omega\times\Omega_{\mathrm{TJ}}}
   \left(
    3\gamma D
    |\nabla^\Omega_{\Delta\vec{\alpha}} g|^2
    +
    \eta D
    |\nabla_{\vec{a}}g|^2
   \right)
   \,d{\it m}
   =
   \iint_{\Omega\times\Omega_{\mathrm{TJ}}}
   |g|^2
   \,d{\it m}
   \bigg|_{t=0}.
 \end{equation}
 Thus,  there is a positive monotone increasing sequence $\{t_j\}$ such
 that $t_j\rightarrow\infty$ and,
 \begin{equation}
  \label{eq:3.15}
   \iint_{\Omega\times\Omega_{\mathrm{TJ}}}
   \left(
    3\gamma D
    |\nabla^\Omega_{\Delta\vec{\alpha}} g|^2
    +
    \eta D
    |\nabla_{\vec{a}}g|^2
   \right)
   \,d{\it m}
   \bigg|_{t=t_j}
   \rightarrow0
   \quad
   \text{as}\ 
   t_j\rightarrow\infty.
 \end{equation}
 Using the monotonicity in time of $3\gamma
 D\|\nabla^\Omega_{\Delta\vec{\alpha}}
 g\|^2_{L^2(\Omega\times\Omega_{\mathrm{TJ}},d{\it m})} + \eta
 D\|\nabla_{\vec{a}} g\|^2_{L^2(\Omega\times\Omega_{\mathrm{TJ}},d{\it m})}$,
 we can take a full limit in time of \eqref{eq:3.15}, namely,
 \begin{equation}
  \iint_{\Omega\times\Omega_{\mathrm{TJ}}}
   \left(
    3\gamma D
    |\nabla^\Omega_{\Delta\vec{\alpha}} g|^2
    +
    \eta D
    |\nabla_{\vec{a}}g|^2
   \right)
   \,d{\it m}
  \rightarrow0
  \quad
  \text{as}\ 
  t\rightarrow\infty.
 \end{equation}
 Next, for $0<T<T'$, we obtain,
 \begin{equation}
  \begin{split}
   &\quad
   \left|
   \iint_{\Omega\times\Omega_{\mathrm{TJ}}}
   \left(
   3\gamma D
   |\nabla^\Omega_{\Delta\vec{\alpha}} g|^2
   +
   \eta D
   |\nabla_{\vec{a}}g|^2
   \right)
   \,d{\it m}
   \bigg|_{t=T'}
   -
   \iint_{\Omega\times\Omega_{\mathrm{TJ}}}
   \left(
   3\gamma D
   |\nabla^\Omega_{\Delta\vec{\alpha}} g|^2
   +
   \eta D
   |\nabla_{\vec{a}}g|^2
   \right)
   \,d{\it m}
   \bigg|_{t=T}
   \right|
   \\
   &=
   \left|
   \int_{T}^{T'}
   \left(
   \frac{d}{dt}
   \iint_{\Omega\times\Omega_{\mathrm{TJ}}}
   \left(
   3\gamma D
   |\nabla^\Omega_{\Delta\vec{\alpha}} g|^2
   +
   \eta D
   |\nabla_{\vec{a}}g|^2
   \right)
   \,d{\it m}
   \right)
   \,dt
   \right| \\
   &\leq
   \int_{T}^{T'}
   \left|
   \frac{d}{dt}
   \iint_{\Omega\times\Omega_{\mathrm{TJ}}}
   \left(
   3\gamma D
   |\nabla^\Omega_{\Delta\vec{\alpha}} g|^2
   +
   \eta D
   |\nabla_{\vec{a}}g|^2
   \right)
   \,d{\it m}
   \right|
   \,dt.
  \end{split}
 \end{equation}
 Using \eqref{eq:3.17}, we deduce,
 \begin{equation*}
  \left|
   \frac{d}{dt}
   \iint_{\Omega\times\Omega_{\mathrm{TJ}}}
   \left(
    3\gamma D
    |\nabla^\Omega_{\Delta\vec{\alpha}} g|^2
    +
    \eta D
    |\nabla_{\vec{a}}g|^2
   \right)
   \,d{\it m}
  \right|
  \leq
  2\Cr{const:3.8}
  e^{-\frac{2\min\{3\gamma,\eta\}D}{\Cr{const:3.3}}t}.
 \end{equation*}
 Hence,  we arrive at,
 \begin{equation}
  \begin{split}
   &\quad
   \left|
   \iint_{\Omega\times\Omega_{\mathrm{TJ}}}
   \left(
   3\gamma D
   |\nabla^\Omega_{\Delta\vec{\alpha}} g|^2
   +
   \eta D
   |\nabla_{\vec{a}}g|^2
   \right)
   \,d{\it m}
   \bigg|_{t=T'}
   -
  \iint_{\Omega\times\Omega_{\mathrm{TJ}}}
   \left(
   3\gamma D
   |\nabla^\Omega_{\Delta\vec{\alpha}} g|^2
   +
   \eta D
   |\nabla_{\vec{a}}g|^2
   \right)
   \,d{\it m}
   \bigg|_{t=T}
   \right| \\
   &\leq
   \int_{T}^{T'}
   2\Cr{const:3.8}
   e^{-\frac{2\min\{3\gamma,\eta\}D}{\Cr{const:3.3}}t}
   \,dt
   =
   \frac{\Cr{const:3.8}\Cr{const:3.3}}{\min\{3\gamma,\eta\}D}
   \left(
   e^{-\frac{2\min\{3\gamma,\eta\}D}{\Cr{const:3.3}}T}
   -
   e^{-\frac{2\min\{3\gamma,\eta\}D}{\Cr{const:3.3}}T'}
   \right).
  \end{split}
 \end{equation}
Taking a limit $T'\rightarrow\infty$, we obtain that,
\begin{equation}
 \label{eq:3.18}
 \iint_{\Omega\times\Omega_{\mathrm{TJ}}}
  \left(
   3\gamma D
   |\nabla^\Omega_{\Delta\vec{\alpha}} g|^2
   +
   \eta D
   |\nabla_{\vec{a}}g|^2
  \right)
  \,d{\it m}
  \bigg|_{t=T}
  \leq
  \frac{\Cr{const:3.8}\Cr{const:3.3}}{\min\{3\gamma,\eta\}D}
  e^{-\frac{2\min\{3\gamma,\eta\}D}{\Cr{const:3.3}}T}.
\end{equation}
 In addition, by direct calculation,  we have that,
 \begin{equation*}
   \nabla (f-f_\infty)
   =
   \left(
    \nabla g 
    -
    \frac{1}{D}
    (g-\Cr{const:3.1})
    \nabla E
   \right)
   \exp \left(-\frac{E}{D}\right)
   =
   \nabla g 
    \exp \left(-\frac{E}{D}\right)
    -
    \frac{1}{D}
    (f-f_\infty)
    \nabla E,
 \end{equation*}
 where $\nabla$ is $\nabla^\Omega_{\Delta\vec{\alpha}}$ or
 $\nabla_{\vec{a}}$. Thus,
 \begin{equation*}
  \begin{split}
   &\quad
   \iint_{\Omega\times\Omega_{\mathrm{TJ}}}
   |\nabla^\Omega_{\Delta\vec{\alpha}} (f-f_\infty)|^2
   \exp \left(\frac{E}{D}\right)
   \,d\Delta\vec{\alpha}d\vec{a} \\
   &\leq
   2
   \iint_{\Omega\times\Omega_{\mathrm{TJ}}}
   |\nabla^\Omega_{\Delta\vec{\alpha}}g|^2
   \exp \left(-\frac{E}{D}\right)
   \,d\Delta\vec{\alpha}d\vec{a} 
   +
   \frac{2}{D^2}
   \iint_{\Omega\times\Omega_{\mathrm{TJ}}}
   |f-f_\infty|^2
   |\nabla^\Omega_{\Delta\vec{\alpha}}E|^2
   \exp \left(\frac{E}{D}\right)
   \,d\Delta\vec{\alpha}d\vec{a}. 
  \end{split} 
 \end{equation*}
 Therefore, from
 \eqref{eq:3.19}, \eqref{eq:3.18}, and boundedness of the gradient of
 $E$,  there is a constant $\Cl{const:3.10}>0$,  such that, 
 \begin{equation*}
  \iint_{\Omega\times\Omega_{\mathrm{TJ}}}
   |\nabla^\Omega_{\Delta\vec{\alpha}} (f-f_\infty)|^2
   \exp \left(\frac{E}{D}\right)
   \,d\Delta\vec{\alpha}d\vec{a} 
   \leq
   \Cr{const:3.10}
   e^{-\frac{2\min\{3\gamma,\eta\}D}{\Cr{const:3.3}}t}.
 \end{equation*}
 Similarly,  there is a constant $\Cl{const:3.11}>0$ such that, 
 \begin{equation*}
  \iint_{\Omega\times\Omega_{\mathrm{TJ}}}
   |\nabla_{\vec{a}} (f-f_\infty)|^2
   \exp \left(\frac{E}{D}\right)
   \,d\Delta\vec{\alpha}d\vec{a} 
   \leq
   \Cr{const:3.11}
   e^{-\frac{2\min\{3\gamma,\eta\}D}{\Cr{const:3.3}}t},
 \end{equation*}
 hence,  we obtain \eqref{eq:3.16}.
\end{proof}

\begin{remark}
 Since $E$ is not $C^2$ for $\vec{a}\in\Omega_{\mathrm{TJ}}$, it is not known
 that $g$ is in $C^3$ on $\Omega_{\mathrm{TJ}}$,  hence one cannot take a
 derivative in $\vec{a}$ of \eqref{eq:3.13}. However,  $g$ is smooth in
 time so we can take a derivative in time. Note that,  we do not use
 third derivative in $\vec{a}$ in the proof of Theorem
 \ref{thm:3.3} (cf. \cite{MR0254401,MR0603444}).
\end{remark}

\begin{remark}
 In Theorem \ref{thm:3.1}, \ref{thm:3.2}, and \ref{thm:3.3}, decay rate
 of solutions to \eqref{eq:3.8} may not be optimal. It will be part of a future work
 to obtain optimal decay orders and dependence on the relaxation time scales
 $\gamma, \eta>0$.
\end{remark}

\begin{remark}
In this paper we have used the Poincar\'e inequality to obtain the
large-time asymptotics of the solution in the weighted $L^2$ framework. The specific difficulties for our system are related to the fact that the potential $\nabla^\Omega_{\Delta\vec{\alpha}}E$ is degenerate and  $\nabla_{\vec{a}}E$
 is not smooth enough. When the potential has better properties, such
 as non-degeneracy and smoothness, one could try to employ the
 logarithmic-Sobolev inequality or the higher order energy
 estimates~\cite{MR1842428,MR0889476,MR3497125,MR1760620} to obtain
 the results in weaker spaces. This is currently under study, and  one
 of the subjects of our forthcoming work would be to study the
 logarithmic-Sobolev type of inequalities and Bakry-\'Emery theory to construct the $L^1$ theory of the system discussed in this paper.
\end{remark}

In this section, we obtained long-time asymptotics for joint distribution $f$ on
$\Omega\times\Omega_{\mathrm{TJ}}$ in the weighted $L^2$ space. In
particular, we established that distribution $f$ converges to the Boltzmann distribution $
f_\infty(\Delta\vec{\alpha},\vec{a}) = \Cr{const:3.1} \exp
\left(-\frac{E(\Delta\vec{\alpha},\vec{a})}{D}\right)$ with respect to the
grain boundary energy $E$ on
$\Omega\times\Omega_{\mathrm{TJ}}$. In the next section, we will study
long-time asymptotics of the marginal probability density.

\section{Marginal probability distribution}\label{sec10}
In this section, for a solution $f$ of the Fokker-Planck equation
\eqref{eq:3.8}, which is a joint distribution on
$\Omega\times\Omega_{\mathrm{TJ}}$, we consider the marginal probability
density of misorientations,  $\rho_1$ of $\Omega$. The probability
density $\rho_1$ is related to the Grain Boundary Character
Distribution (GBCD). The GBCD (in 2D context
and with the grain boundary energy density which only depends on the
misorientation $\Delta \alpha$) is an empirical
statistical measure of the relative length (in 2D) of the grain
boundary interface with a given lattice misorientation. GBCD can be viewed as a primary statistical descriptor to characterize texture of the grain boundary
network, and is inversely related to the grain boundary energy density
as discovered in experiments and simulations.
The
reader can consult, for instance, \cite{DK:BEEEKT,DK:gbphysrev,
  MR2772123, MR3729587} for more details about GBCD and the theory of
the GBCD, and see also Section~\ref{sec13}. 

In this section,  we compare the long-time asymptotics for the marginal
distribution $\rho_{1,\infty}$ and the Boltzmann distributions on
$\Omega$. Hence, let us define the marginal distributions for a misorientation
$\Delta\vec{\alpha}=(\Delta\alpha^{(1)},\Delta\alpha^{(2)},\Delta\alpha^{(3)})\in\Omega$,
and for a position of the triple junction
$\vec{a}\in\Omega_{\mathrm{TJ}}$,
\begin{equation}
 \label{eq:4.1}
  \rho_1(\Delta\vec{\alpha},t)
  =
  \int_{\Omega_{\mathrm{TJ}}}
  f(\Delta\vec{\alpha},\vec{a},t)
  \,d\vec{a}, \quad
  \rho_2(\vec{a},t)
  =
  \int_{\Omega}
  f(\Delta\vec{\alpha},\vec{a},t)
  \,d\Delta\vec{\alpha},
\end{equation}
and
\begin{equation}
 \label{eq:4.2}
  \rho_{1,\infty}(\Delta\vec{\alpha})
  =
  \int_{\Omega_{\mathrm{TJ}}}
  f_\infty(\Delta\vec{\alpha},\vec{a})
  \,d\vec{a}, \quad
  \rho_{2,\infty}(\vec{a})
  =
  \int_{\Omega}
  f_\infty(\Delta\vec{\alpha},\vec{a})
  \,d\Delta\vec{\alpha}.
\end{equation}
From Theorems \ref{thm:3.1}, \ref{thm:3.2}, and \ref{thm:3.3} in Section~\ref{sec5}, we can
obtain long time asymptotics of $\rho_1$ and $\rho_2$.

\begin{proposition}
 \label{prop:4.1}
 Assume that $\sigma$ is a $C^1$ function on $\R$ and
 $\Omega\times\Omega_{\mathrm{TJ}}$ supports the 2-Poincar\'e-Wirtinger
 inequality \eqref{eq:3.7}. Let $f_0\in
 L^2(\Omega\times\Omega_{\mathrm{TJ}},e^{\frac{E}{D}}\,d\Delta\vec{\alpha}d\vec{a})$
 be a probability density function. Let $\rho_1$ be defined in
 \eqref{eq:4.1}. Then, for any $t_0>0$,  there are positive constants $\Cl{const:4.1}$,
 $\Cl{const:4.2}$, and $\Cl{const:4.3}>0$, such that for $ t>t_0$,
 \begin{equation}
  \begin{split}
   \int_{\Omega}
   |\rho_1(\Delta\vec{\alpha},t)-\rho_{1,\infty}(\Delta\vec{\alpha})|^2
   \,d\Delta\vec{\alpha} 
   &\leq
   \Cr{const:4.1}
   e^{-\frac{2\min\{3\gamma,\eta\}D}{\Cr{const:3.3}}t}, \\
   \int_{\Omega}
   |(\rho_1)_t(\Delta\vec{\alpha},t)|^2
   \,d\Delta\vec{\alpha} 
   &\leq
   \Cr{const:4.2}
   e^{-\frac{2\min\{3\gamma,\eta\}D}{\Cr{const:3.3}}t}, \\
   \int_{\Omega}
   |\nabla^{\Omega_{\mathrm{TJ}}}_{\Delta\vec{\alpha}}
   (\rho_1(\Delta\vec{\alpha},t)-\rho_{1,\infty}(\Delta\vec{\alpha}))
   |^2
   \,d\Delta\vec{\alpha} 
   &\leq
   \Cr{const:4.3}
   e^{-\frac{2\min\{3\gamma,\eta\}D}{\Cr{const:3.3}}t}, 
  \end{split}
 \end{equation}
 where $\Cr{const:3.3}>0$ is a constant defined in \eqref{eqc3}.
\end{proposition}

\begin{proof}
We get by H\"older's inequality that,
 \begin{equation}
  \begin{split}
   |\rho_1(\Delta\vec{\alpha},t)-\rho_{1,\infty}(\Delta\vec{\alpha},t)|^2
   &=
   \left|
   \int_{\Omega_{\mathrm{TJ}}}
   \left(
   f(\Delta\vec{\alpha},\vec{a},t)
   -
   f_\infty(\Delta\vec{\alpha},\vec{a})
   \right)
   \,d\vec{a}
   \right|^2
   \\
   &\leq
   |\Omega_{\mathrm{TJ}}|
   \int_{\Omega_{\mathrm{TJ}}}
   \left|
   f(\Delta\vec{\alpha},\vec{a},t)
   -
   f_\infty(\Delta\vec{\alpha},\vec{a})
   \right|^2
   \,d\vec{a}.
  \end{split}
 \end{equation}
 Next, note that $\Cr{const:3.4}\leq e^{-\frac{E}{D}}\leq
 \Cr{const:3.5}$ on $\Omega\times\Omega_{\mathrm{TJ}}$, where the
 constants $\Cr{const:3.4}, \Cr{const:3.5}>0$ are defined in
 \eqref{eq:3.32}. Thus, we obtain,
 \begin{equation}
  \int_{\Omega_{\mathrm{TJ}}}
   \left|
    f(\Delta\vec{\alpha},\vec{a},t)
    -
    f_\infty(\Delta\vec{\alpha},\vec{a})
   \right|^2
   \,d\vec{a}
   \leq
   \Cr{const:3.5}
   \int_{\Omega_{\mathrm{TJ}}}
   \left|
    f(\Delta\vec{\alpha},\vec{a},t)
    -
    f_\infty(\Delta\vec{\alpha},\vec{a})
   \right|^2
   \exp \left(\frac{E(\Delta\vec{\alpha},\vec{a})}{D}\right)
   \,d\vec{a}.
 \end{equation}
 Then, using \eqref{eq:3.19}, we have,
 \begin{equation}
  \begin{split}
   &\quad
   \int_{\Omega}
   |\rho_1(\Delta\vec{\alpha},t)-\rho_{1,\infty}(\Delta\vec{\alpha},t)|^2
   \,d\Delta\vec{\alpha} \\
   &\leq
   \Cr{const:3.5}|\Omega_{\mathrm{TJ}}|
   \iint_{\Omega\times\Omega_{\mathrm{TJ}}}
   \left|
   f(\Delta\vec{\alpha},\vec{a},t)
   -
   f_\infty(\Delta\vec{\alpha},\vec{a})
   \right|^2
   \exp \left(\frac{E(\Delta\vec{\alpha},\vec{a})}{D}\right)
   \,d\vec{a}d\Delta\vec{\alpha} \\
   &\leq
   \Cr{const:3.6}
   \Cr{const:3.5}|\Omega_{\mathrm{TJ}}|
   e^{-\frac{2\min\{3\gamma,\eta\}D}{\Cr{const:3.3}}t},
  \end{split}
 \end{equation}
 hence the exponential decay estimate for $\rho_1$ is derived.
 Similarly, the estimates for $(\rho_1)_t$ and
 $\nabla^\Omega_{\Delta\vec{\alpha}}\rho_1$ can be deduced.
\end{proof}

\begin{remark}
 Using the same argument as in the proof of Proposition \ref{prop:4.1}, one
 can obtain similar long-time asymptotics for the probability density
 $\rho_2$. In this work, we are more interested in the analysis of the
 marginal probability density of the misorientations $\Delta\vec{\alpha}$,
 $\rho_1 =\rho_1(\Delta\vec{\alpha}, t)$ due to the relation to the GBCD
 statistical metric.
\end{remark}

Next, we compare $\rho_{1,\infty}$ and the Boltzmann distribution of the
misorientations $\Delta\vec{\alpha}$. We first derive the evolution
equation for the the marginal distribution $\rho_1$.

\begin{proposition}
 Let $f$ be a solution of \eqref{eq:3.8}, and let
 $\rho_1=\rho_1(\Delta\vec{\alpha},t)$ be a marginal distribution defined
 by \eqref{eq:4.1}. Then, $\rho_1$ satisfies,
 \begin{equation}
  \label{eq:4.13} 
  \frac{\partial\rho_1}{\partial t}
   =
   \frac{\beta^2_{\Delta\vec{\alpha}}}{2}
   \Delta_{\Delta\vec{\alpha}}^{\Omega} \rho_1
   -
   \nabla^\Omega_{\Delta\vec{\alpha}}
   \cdot
   \left(
    \int_{\Omega_{\mathrm{TJ}}}
    \left(
     \vec{v}_{\Delta\vec{\alpha}}f
    \right)
    \,d\vec{a}
   \right),
   \quad \vec{a}\in\Omega,\ t>0.
 \end{equation}
\end{proposition}

\begin{proof}
Integrate \eqref{eq:3.8} in
 $\vec{a}\in\Omega_{\mathrm{TJ}}$, hence we obtain,
 \begin{equation}
  \label{eq:4.11} 
   \frac{\partial \rho_1}{\partial t}
   +
   \nabla^\Omega_{\Delta\vec{\alpha}}
   \cdot
   \left(
    \int_{\Omega_{\mathrm{TJ}}}
    \left(
     \vec{v}_{\Delta\vec{\alpha}}f
    \right)
    \,d\vec{a}
   \right)
   +
   \int_{\Omega_{\mathrm{TJ}}}
   \nabla_{\vec{a}}
   \cdot
    \left(
     \vec{v}_{\vec{a}}f
    \right)
   \,d\vec{a}
   =
   \frac{\beta^2_{\Delta\vec{\alpha}}}{2}
   \Delta_{\Delta\vec{\alpha}}^{\Omega} \rho_1
   +
   \frac{\beta^2_{\vec{a}}}{2}
   \int_{\Omega_{\mathrm{TJ}}}
   \Delta_{\vec{a}} f 
   \,d\vec{a}.
 \end{equation}
 Due to the boundary conditions of \eqref{eq:3.8}, it follows,
 \begin{equation}
  \label{eq:4.12}
   \frac{\beta^2_{\vec{a}}}{2}
   \int_{\Omega_{\mathrm{TJ}}}
   \Delta_{\vec{a}} f 
   \,d\vec{a}
   -
   \int_{\Omega_{\mathrm{TJ}}}
   \nabla_{\vec{a}}
   \cdot
    \left(
     \vec{v}_{\vec{a}}f
    \right)
   \,d\vec{a}
   =
   \int_{\partial\Omega_{\mathrm{TJ}}}
   \left(
    \frac{\beta^2_{\vec{a}}}{2}
    \nabla_{\vec{a}} f 
    -
    \vec{v}_{\vec{a}}f
   \right)
   \cdot
   \vec{\nu}_{\vec{a}}
   \,dS_{\vec{a}}
   =
   0
 \end{equation}
 for $\Delta\vec{\alpha}\in\Omega$, where $\vec{\nu}_{\vec{a}}$ is an outer unit
 normal on $\partial\Omega_{\mathrm{TJ}}$, and
 $dS_{\vec{a}}$ is a length element on
 $\partial\Omega_{\mathrm{TJ}}$. From \eqref{eq:4.11} and
 \eqref{eq:4.12}, one can obtain \eqref{eq:4.13}.
\end{proof}
To proceed with the analysis of $\rho_1$, we first consider the Taylor
expansion of the grain boundary energy
$E$ around arbitrarily selected point $\vec{a}_{*}\in\Omega_{\mathrm{TJ}}$, namely,

\begin{equation}\label{eqe1e2}
 E(\Delta\vec{\alpha},\vec{a})
  =
  \sum_{j=1}^3
  \sigma(\Delta^{(j)}\alpha)|\vec{a}-\vec{x}^{(j)}|
  =
  E_1(\Delta\vec{\alpha})
  +
  E_2(\Delta\vec{\alpha},\vec{a}),
\end{equation}
where,
\begin{equation}
 E_1(\Delta\vec{\alpha})
  =
  \sum_{j=1}^3
  \sigma(\Delta^{(j)}\alpha)|\vec{a}_*-\vec{x}^{(j)}|
  , \mbox{ and }
  E_2(\Delta\vec{\alpha},\vec{a})
  =
  E(\Delta\vec{\alpha},\vec{a})
  -
  E_1(\Delta\vec{\alpha}).
\end{equation}
Note that, we formulated the grain boundary energy $E$ in the form above
\eqref{eqe1e2} to investigate effect of the position of the triple
junction $\vec{a}=\vec{a}_{*}$ on the distribution of the
misorientations $\rho_1(\Delta\vec{\alpha},t)$ and its steady-state
distribution $\rho_{1, \infty}(\Delta\vec{\alpha})$.
 
\begin{remark}
 From Proposition \ref{prop:4.1}, marginal distribution $\rho_1$ may not
 converge to the Boltzmann distribution, in general. This is because,
 \begin{equation}
  \label{eq:4.19}
  \rho_{1,\infty}(\Delta\vec{\alpha})
   =
   \left(
    \Cr{const:3.1}
    \int_{\Omega_{\mathrm{TJ}}}
    \exp
    \left(-\frac{E_2(\Delta\vec{\alpha},\vec{a})}{D}\right)
    \,d\vec{a}
   \right)
   \exp
   \left(
    -\frac{E_1(\Delta\vec{\alpha})}{D}
   \right),
 \end{equation}
 and the coefficient of $\exp(-E_1/D)$ generally depend on
 $\Delta\vec{\alpha}$.
\end{remark}

 Using \eqref{eqe1e2}, equation \eqref{eq:4.13} becomes,
\begin{equation}
 \label{eq:4.22}
 \frac{\partial\rho_1}{\partial t}
  =
  \frac{\beta^2_{\Delta\vec{\alpha}}}{2}
  \Delta_{\Delta\vec{\alpha}}^{\Omega} \rho_1
  -
  \nabla^\Omega_{\Delta\vec{\alpha}}
  \cdot
  \left(
   (-3\gamma\nabla_{\Delta\vec{\alpha}}^\Omega E_1(\Delta\vec{\alpha}))
   \rho_1
  \right)
  +
  3\gamma
  \nabla^\Omega_{\Delta\vec{\alpha}}
  \cdot
  \left(
   \int_{\Omega_{\mathrm{TJ}}}
   \left(
    (\nabla_{\Delta\vec{\alpha}}^\Omega E_2(\Delta\vec{\alpha},\vec{a})))
    f
    \right)
   \,d\vec{a}
  \right),
\end{equation}
hence $\rho_1$ satisfies the Fokker-Planck type equation with an
extra term.
Next, we explore the effects of the triple junction position, $\vec{a}=\vec{a}_*$ on \eqref{eq:4.13}.

\begin{remark}
 In \cite{DK:BEEEKT,DK:gbphysrev,
  MR2772123, MR3729587}, Fokker-Planck equation was derived for the
evolution of the GBCD using a novel implementation of the iterative scheme for the Fokker-Planck
equation in terms of the system free energy and a
Kantorovich-Rubinstein-Wasserstein metric. Equation for probability
density of misorientations $\rho_1$,  \eqref{eq:4.13} or
 \eqref{eq:4.22} is a Fokker-Planck type equation which also takes
 into account the
 effect of the mobility of the triple junctions. 
\end{remark}

\begin{remark}
Because of 
\begin{equation}
 \partial_{a_k}|\vec{a}-\vec{x}^{(j)}|
  =
  \frac{a_k-x_k^{(j)}}{|\vec{a}-\vec{x}^{(j)}|},\quad
  \partial_{a_l}\partial_{a_k}|\vec{a}-\vec{x}^{(j)}|
  =
  \frac{1}{|\vec{a}-\vec{x}^{(j)}|}
  \left(
   \delta_{kl}
   -
   \frac{a_k-x_k^{(j)}}{|\vec{a}-\vec{x}^{(j)}|}
   \frac{a_l-x_l^{(j)}}{|\vec{a}-\vec{x}^{(j)}|}
  \right),
\end{equation}
where $\delta_{kl}$ is the Kronecker delta, $\vec{a}=(a_1,a_2)$, and
$\vec{x}^{(j)}=(x_1^{(j)},x_2^{(j)})$, by the Taylor expansion for
$|\vec{a}-\vec{x}^{(j)}|$ around $\vec{a}_*$ we obtain the following
expansion for $E_2$;
\begin{equation}
 \label{eq:4.10}
 \begin{split}
  E_2(\Delta\vec{\alpha},\vec{a})
  &=
  \sum_{j=1}^3
  \sigma(\Delta\alpha^{(j)})
  \left(
  |\vec{a}-\vec{x}^{(j)}|
  -
  |\vec{a}_*-\vec{x}^{(j)}|
  \right) \\
  &=
  \sum_{j=1}^3
  \sigma(\Delta\alpha^{(j)})
  \Bigg(
  \frac{(\vec{a}_*-\vec{x}^{(j)})}{|\vec{a}_*-\vec{x}^{(j)}|}
  \cdot(\vec{a}-\vec{a}_*) \\
  &\qquad
  +
  \frac{1}{2|\vec{a}_*-\vec{x}^{(j)}|}
  \left(
  |\vec{a}-\vec{a}_*|^2
  -
  \left(
  \frac{(\vec{a}_*-\vec{x}^{(j)})}{|\vec{a}_*-\vec{x}^{(j)}|}
  \cdot(\vec{a}-\vec{a}_*)
  \right)^2
  \right)
  +o(|\vec{a}-\vec{a}_*|^2)
  \Bigg)
 \end{split}
\end{equation}
 as $\vec{a}\rightarrow\vec{a}_*$. 
\end{remark}


\subsection{The weighted Fermat-Torricelli point as a triple junction
  point}\label{sec11}
Let $\vec{a}_*$ be the minimizer of $E(\Delta\vec{\alpha},\vec{a})$,
which is called the weighted Fermat-Torricelli point
$\vec{a}_{\mathrm{wFT}}$, for fixed $\Delta\vec{\alpha}\in\Omega$
(cf.~\cite{MR1677397}), that is
\begin{equation}
 \sum_{j=1}^3\sigma(\Delta^{(j)}\alpha)
  |\vec{a}_{\mathrm{wFT}}-\vec{x}^{(j)}|
  =
  \inf_{\vec{a}\in\Omega_{\mathrm{TJ}}}
  \sum_{j=1}^3\sigma(\Delta^{(j)}\alpha)
  |\vec{a}-\vec{x}^{(j)}|
  =
  \inf_{\vec{a}\in\Omega_{\mathrm{TJ}}}
  E(\Delta\vec{\alpha},\vec{a}).
\end{equation}
Let $\psi^{(i)}$ be an angle formed by $\vec{a}-\vec{x}^{(i)}$ and
$\vec{a}-\vec{x}^{(i+1)}$ at the triple junction $\vec{a}$. Now we give
an equivalent condition that the triple junction coincides with
$\vec{a}_{\mathrm{wFT}}$.

\begin{figure}[h]
 \centering
 \includegraphics[width=8cm]{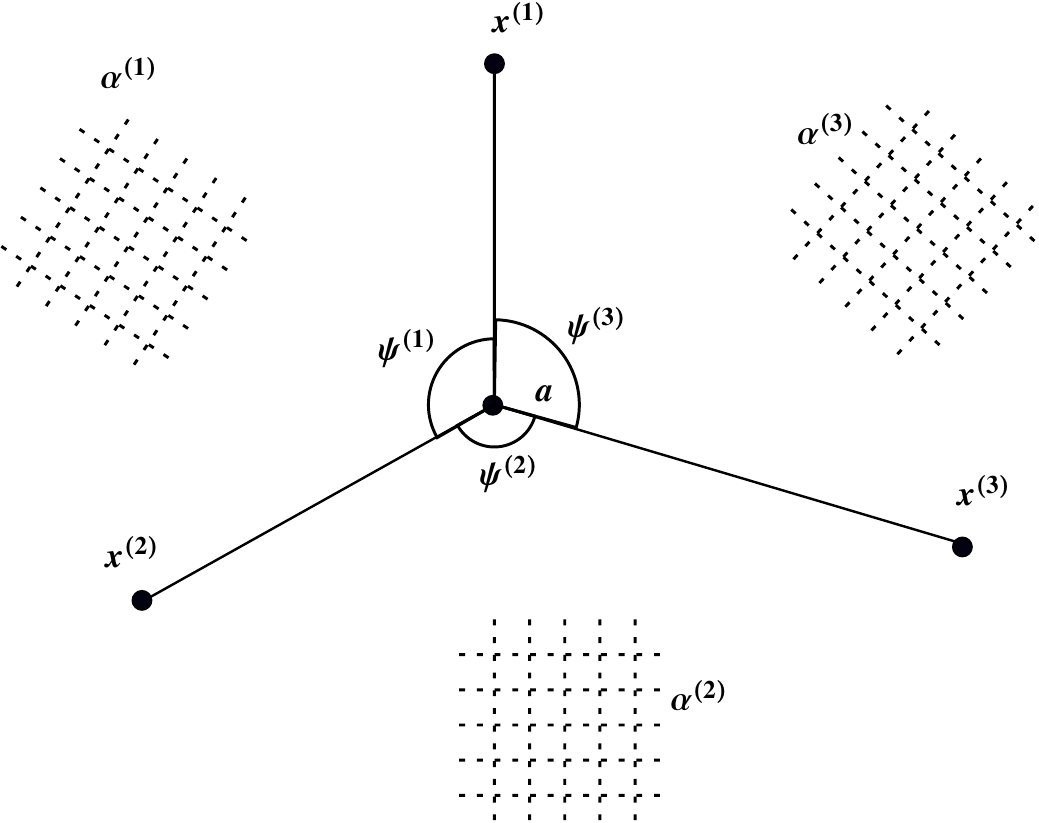}
 \caption{The angles $\psi_1$, $\psi_2$, $\psi_3$ are defined as the above figure.}\label{figTJ1}
\end{figure}

\begin{proposition}
 \label{prop:4.2} Assume that weighted Fermat-Torricelli point
 $\vec{a}_{\mathrm{wFT}}$ does not coincide with $\vec{x}^{(j)}$ for
 $j=1,2,3$. Then, the triple junction coincides with
 $\vec{a}_{\mathrm{wFT}}$, if and only if,
 \begin{equation}
  \label{eq:4.8}
   1-\cos\psi^{(i)}
   =
   \frac{(\sigma(\Delta^{(i)}\alpha)+\sigma(\Delta^{(i+1)}\alpha))^2
   -
   \sigma(\Delta^{(i+2)}\alpha)^2}
   {2\sigma(\Delta^{(i)}\alpha)\sigma(\Delta^{(i+1)}\alpha)},   
 \end{equation}
 for $i=1,2,3$.
\end{proposition}

\begin{remark}
 A condition that for $k=1,2,3$,
 \begin{equation}
  \label{eq:4.9}
  \left|
   \sum_{\substack{1\leq j\leq 3, \\j\neq k}}
   \sigma(\Delta^{(j)}\alpha)
   \frac{\vec{x}^{(j)}-\vec{x}^{(k)}}{|\vec{x}^{(j)}-\vec{x}^{(k)}|}
  \right|
   >
   \sigma(\Delta^{(k)}\alpha)
 \end{equation}
 is equivalent to the condition that $\vec{a}_{\mathrm{wFT}}$ does not
 coincide with $\vec{x}^{(j)}$ for $j=1,2,3$ and
 \begin{equation}
  \label{eq:4.3}
   \vec{0}
   =
   \nabla_{\vec{a}}E(\Delta\vec{\alpha},\vec{a}_{\mathrm{wFT}})
   =
   \sum_{j=1}^3
   \sigma(\Delta^{(j)}\alpha)
   \vec{e}^{(j)},
   \quad
   \vec{e}^{(j)}
   :=
   \frac{\vec{a}_{\mathrm{wFT}}-\vec{x}^{(j)}}{|\vec{a}_{\mathrm{wFT}}-\vec{x}^{(j)}|}
 \end{equation}
 holds (See \cite[Theorem 18.37]{MR1677397}). 

 Because, by \eqref{eq:4.9}, 
 \begin{equation*}
  \sigma(\Delta^{(k)}\alpha)
   <
   \left|
    \sigma(\Delta^{(i)}\alpha)
    \frac{\vec{x}^{(i)}-\vec{x}^{(k)}}{|\vec{x}^{(i)}-\vec{x}^{(k)}|}
    +
    \sigma(\Delta^{(j)}\alpha)
    \frac{\vec{x}^{(j)}-\vec{x}^{(k)}}{|\vec{x}^{(j)}-\vec{x}^{(k)}|}
   \right| 
 \end{equation*}
 for different $1\leq i,j,k\leq 3$, 
 one
 can obtain the following condition,
 \begin{equation}
  \label{eq:4.15}
   \sigma(\Delta^{(k)}\alpha)
   <
   \sigma(\Delta^{(i)}\alpha)
   +
   \sigma(\Delta^{(j)}\alpha).
 \end{equation}

\end{remark}

\begin{proof}[Proof of Proposition \ref{prop:4.2}]
 When $\vec{a}_{\mathrm{wFT}}$ does not coincide with $\vec{x}^{(j)}$ for
 $j=1,2,3$, the weighted Fermat-Torricelli point $\vec{a}_{\mathrm{wFT}}$
 satisfies \eqref{eq:4.3}.
 Taking the inner product of \eqref{eq:4.3}  with $\vec{e}^{(k)}$ for
 $k=1,2,3$, we obtain,
\begin{equation}
\label{eq:4.4}
 \left\{
  \begin{aligned}
   \sigma(\Delta^{(2)}\alpha)
   (\vec{e}^{(1)},\vec{e}^{(2)})
   +
   \sigma(\Delta^{(3)}\alpha)
   (\vec{e}^{(3)},\vec{e}^{(1)})
   &=
   -\sigma(\Delta^{(1)}\alpha), \\
   \sigma(\Delta^{(1)}\alpha)
   (\vec{e}^{(1)},\vec{e}^{(2)})
   +
   \sigma(\Delta^{(3)}\alpha)
   (\vec{e}^{(2)},\vec{e}^{(3)})
   &=
   -\sigma(\Delta^{(2)}\alpha),\\
   \sigma(\Delta^{(2)}\alpha)
   (\vec{e}^{(2)},\vec{e}^{(3)})
   +
   \sigma(\Delta^{(1)}\alpha)
   (\vec{e}^{(3)},\vec{e}^{(1)})
   &=
   -\sigma(\Delta^{(3)}\alpha).
  \end{aligned}
 \right.
\end{equation}
 Next, we can solve $(\vec{e}^{(k)},\vec{e}^{(l)})$ from
 \eqref{eq:4.4} and, thus,
 obtain,
 \begin{equation}
  \begin{split}
   (\vec{e}^{(1)},\vec{e}^{(2)})
   &=
   \frac{-(\sigma(\Delta^{(1)}\alpha))^2-(\sigma(\Delta^{(2)}\alpha))^2+(\sigma(\Delta^{(3)}\alpha))^2}{2\sigma(\Delta^{(1)}\alpha)\sigma(\Delta^{(2)}\alpha)}, \\
   (\vec{e}^{(2)},\vec{e}^{(3)})
   &=
   \frac{-(\sigma(\Delta^{(2)}\alpha))^2-(\sigma(\Delta^{(3)}\alpha))^2+(\sigma(\Delta^{(1)}\alpha))^2}{2\sigma(\Delta^{(2)}\alpha)\sigma(\Delta^{(3)}\alpha)}, \\
   (\vec{e}^{(3)},\vec{e}^{(1)})
   &=
   \frac{-(\sigma(\Delta^{(3)}\alpha))^2-(\sigma(\Delta^{(1)}\alpha))^2+(\sigma(\Delta^{(2)}\alpha))^2}{2\sigma(\Delta^{(3)}\alpha)\sigma(\Delta^{(1)}\alpha)}.
  \end{split}
 \end{equation}
 Note that, $(\vec{e}^{(k)},\vec{e}^{(l)})$ is the cosine of the angle
 at $\vec{a}_{\mathrm{wFT}}$ formed by $\vec{e}^{(k)}$ and
 $\vec{e}^{(l)}$. Thus, we have,
 \begin{equation}
  \begin{split}
   \cos\psi^{(1)}
   &=
   \frac{-(\sigma(\Delta^{(1)}\alpha))^2-(\sigma(\Delta^{(2)}\alpha))^2+(\sigma(\Delta^{(3)}\alpha))^2}{2\sigma(\Delta^{(1)}\alpha)\sigma(\Delta^{(2)}\alpha)}, \\
   \cos\psi^{(2)}
   &=
   \frac{-(\sigma(\Delta^{(2)}\alpha))^2-(\sigma(\Delta^{(3)}\alpha))^2+(\sigma(\Delta^{(1)}\alpha))^2}{2\sigma(\Delta^{(2)}\alpha)\sigma(\Delta^{(3)}\alpha)}, \\
   \cos\psi^{(3)}
   &=
   \frac{-(\sigma(\Delta^{(3)}\alpha))^2-(\sigma(\Delta^{(1)}\alpha))^2+(\sigma(\Delta^{(2)}\alpha))^2}{2\sigma(\Delta^{(3)}\alpha)\sigma(\Delta^{(1)}\alpha)},
  \end{split}
 \end{equation}
 hence we arrive at \eqref{eq:4.8} by direct calculation of $1-\cos\psi^{(i)}$. 

 Conversely, when \eqref{eq:4.8} holds, then \eqref{eq:4.4} also
 holds. Using \eqref{eq:4.15}, and any nonparallel pair of
 $\vec{e}^{(j)}$, one can obtain \eqref{eq:4.3}. Since $\vec{a}_{\mathrm{wFT}}$
 is unique~\cite[Theorem 18.37]{MR1677397}, $\vec{a}_{\mathrm{wFT}}$
 coincides with the triple junction.
\end{proof}

\begin{remark}
 The relation \eqref{eq:4.8} is a force balance condition at the
 triple junction,  the generalized Herring condition. When $\sigma\equiv1$, then
 from \eqref{eq:4.8} we have $\cos\psi^{(i)}=-1/2$, hence three angles
 at the triple junction are the same, $\frac{2\pi}{3}$.
\end{remark}

Next, we study the behavior of the reminder term $E_2$ when
$\vec{a}_*=\vec{a}_{\mathrm{wFT}}$. Thanks to
$\nabla_{\vec{a}}E(\Delta\vec{\alpha},\vec{a}_*)=\vec{0}$, one can
obtain the following result,

\begin{proposition}
 \label{prop:4.4}
 Assume that weighted Fermat-Torricelli point $\vec{a}_{\mathrm{wFT}}$
 does not coinside with $\vec{x}^{(j)}$ for $j=1,2,3$. Let
 $\vec{a}_*=\vec{a}_{\mathrm{wFT}}$. Then,
 \begin{equation}
  \label{eq:4.17}
  E_2(\Delta\vec{\alpha},\vec{a})
   =
   \sum_{j=1}^3
   \sigma(\Delta^{(j)}\alpha)
   \left(
    \frac{1}{2|\vec{a}_*-\vec{x}^{(j)}|}
    \left(
     |\vec{a}-\vec{a}_*|^2
     -
     \left(
      \frac{(\vec{a}_*-\vec{x}^{(j)})}{|\vec{a}_*-\vec{x}^{(j)}|}
      \cdot(\vec{a}-\vec{a}_*)
     \right)^2
    \right)
    +o(|\vec{a}-\vec{a}_*|^2)
   \right)
 \end{equation} 
 as $\vec{a}\rightarrow\vec{a}_*$. 
\end{proposition}

\begin{proof}
 Since $\nabla_{\vec{a}}E(\Delta\vec{\alpha},\vec{a}_*)=\vec{0}$, by
 \eqref{eq:4.3} we have that,
 \begin{equation*}
  \sum_{j=1}^3
   \sigma(\Delta\alpha^{(j)})
   \frac{(\vec{a}_*-\vec{x}^{(j)})}{|\vec{a}_*-\vec{x}^{(j)}|}
   =
   \vec{0}.
 \end{equation*}
 Using this in \eqref{eq:4.10}, we obtain \eqref{eq:4.17}.
\end{proof}

The above Proposition \ref{prop:4.4} is a reason of why we choose
$\vec{a}_{\mathrm{wFT}}$ as $\vec{a}_*$, namely, we can show that $E_2$
is asymptotically of order $|\vec{a}-\vec{a}_*|^2$ as
$\vec{a}\rightarrow\vec{a}_*$.

\subsection{The circumcenter as a triple junction point}\label{sec12}

Next, we introduce the circumcenter $\vec{a}_{\mathrm{cc}}$ of
$\vec{x}^{(j)}$. The
circumcircle of $\vec{x}^{(j)}$ is the unique circle that passes through
all $\vec{x}^{(j)}$, and the circumcenter $\vec{a}_{\mathrm{cc}}$ of
$\vec{x}^{(j)}$ is the center of the circumcircle, namely
\begin{equation}
 \label{eq:4.5}
 |\vec{a}_{\mathrm{cc}}-\vec{x}^{(1)}|
  =
  |\vec{a}_{\mathrm{cc}}-\vec{x}^{(2)}|
  =
  |\vec{a}_{\mathrm{cc}}-\vec{x}^{(3)}|.
\end{equation}
If a triple junction $\vec{a}$ coincides with the circumcenter then, Boltzmann
distribution $\exp(-\frac{E_1}{D})$  becomes Boltzmann
distribution for a grain boundary energy density
$\sigma(\Delta^{(j)}\alpha)$ (instead of Boltzmann distribution for the
grain boundary energy $E$), namely, 
\begin{equation}
 \label{eq:4.21}
  E_1(\Delta\vec{\alpha})
  =
  |\vec{a}_{\mathrm{cc}}-\vec{x}^{(1)}|
  \sum_{j=1}^3\sigma(\Delta^{(j)}\alpha).
\end{equation}
This is reminiscent of the result for the steady-state GBCD which is
given by the
Boltzmann distribution for the grain boundary energy density, see for
instance, \cite{MR3729587, DK:gbphysrev, DK:BEEEKT, MR2772123}. When
$\vec{a}_*$ coincides with $\vec{a}_{\mathrm{cc}}$, from \eqref{eq:4.19}
and \eqref{eq:4.21}, $\rho_{1,\infty}$ is similar to $\exp \left(
-\frac{|\vec{a}_{\mathrm{cc}}-\vec{x}^{(1)}|}{D}
\sum_{j=1}^3\sigma(\Delta^{(j)}\alpha) \right)$.

We now give a relation between the angle at the circumcenter point and
the point
$\vec{x}^{(j)}$.
\begin{proposition}
 If the triple junction coincides with the circumcenter
 $\vec{a}_{\mathrm{cc}}$, then
 \begin{equation}
  \label{eq:4.18}
\frac{1-\cos\psi^{(1)}}{|\vec{x}^{(1)}-\vec{x}^{(2)}|^2}
  =
  \frac{1-\cos\psi^{(2)}}{|\vec{x}^{(2)}-\vec{x}^{(3)}|^2}
  =
  \frac{1-\cos\psi^{(3)}}{|\vec{x}^{(3)}-\vec{x}^{(1)}|^2}.
 \end{equation}
\end{proposition}

\begin{proof}
 By the cosine formula, for $i=1,2,3$,
 \begin{equation*}
  \begin{split}
   |\vec{x}^{(i)}-\vec{x}^{(i+1)}|^2
   &=
   |\vec{a}_{\mathrm{cc}}-\vec{x}^{(i)}|^2
   +
   |\vec{a}_{\mathrm{cc}}-\vec{x}^{(i+1)}|^2
   -
   2
   |\vec{a}_{\mathrm{cc}}-\vec{x}^{(i)}|
   |\vec{a}_{\mathrm{cc}}-\vec{x}^{(i+1)}|
   (\vec{e}^{(i)},\vec{e}^{(i+1)}) \\
   &=
   2R^2
   (1-\cos\psi^{(i)}),
  \end{split} 
 \end{equation*}
 where
 $R=|\vec{a}_{\mathrm{cc}}-\vec{x}^{(1)}|=|\vec{a}_{\mathrm{cc}}-\vec{x}^{(2)}|=|\vec{a}_{\mathrm{cc}}-\vec{x}^{(3)}|$. Thus, 
\begin{equation*}
 \frac{1}{2R^2}
  =
  \frac{1-\cos\psi^{(1)}}{|\vec{x}^{(1)}-\vec{x}^{(2)}|^2}
  =
  \frac{1-\cos\psi^{(2)}}{|\vec{x}^{(2)}-\vec{x}^{(3)}|^2}
  =
  \frac{1-\cos\psi^{(3)}}{|\vec{x}^{(3)}-\vec{x}^{(1)}|^2},
\end{equation*}
hence \eqref{eq:4.18} holds.
\end{proof}

Next we look at a necessary condition for
$\vec{a}_{\mathrm{wFT}}=\vec{a}_{\mathrm{cc}}$. By combining the relations
\eqref{eq:4.8} and \eqref{eq:4.18}, we have the following corollary,

\begin{corollary}
 Assume that weighted Fermat-Torricelli point $\vec{a}_{\mathrm{wFT}}$
 does not coincide with $\vec{x}^{(j)}$ for $j=1,2,3$. If the triple
 junction, $\vec{a}_{\mathrm{wFT}}$, and circumcenter
 $\vec{a}_{\mathrm{cc}}$ are all the same, then,
 \begin{equation}
  \label{eq:4.16}
  \begin{split}
   \frac{(\sigma(\Delta^{(1)}\alpha)+\sigma(\Delta^{(2)}\alpha))^2
   -
   \sigma(\Delta^{(3)}\alpha)^2}
   {2\sigma(\Delta^{(1)}\alpha)\sigma(\Delta^{(2)}\alpha)
   |\vec{x}^{(1)}-\vec{x}^{(2)}|^2
   }   
   &=
  \frac{(\sigma(\Delta^{(2)}\alpha)+\sigma(\Delta^{(3)}\alpha))^2
   -
   \sigma(\Delta^{(1)}\alpha)^2}
   {2\sigma(\Delta^{(2)}\alpha)\sigma(\Delta^{(3)}\alpha)
   |\vec{x}^{(2)}-\vec{x}^{(3)}|^2
   } 
   \\  
   &=
   \frac{(\sigma(\Delta^{(3)}\alpha)+\sigma(\Delta^{(1)}\alpha))^2
   -
   \sigma(\Delta^{(2)}\alpha)^2}
   {2\sigma(\Delta^{(3)}\alpha)\sigma(\Delta^{(1)}\alpha)
   |\vec{x}^{(3)}-\vec{x}^{(1)}|^2
   }.   
  \end{split} 
 \end{equation} 
\end{corollary}

\begin{remark}
 When $\sigma\equiv1$, then the relation \eqref{eq:4.16} gives
 \begin{equation*}
  |\vec{x}^{(1)}-\vec{x}^{(2)}|
   =
   |\vec{x}^{(2)}-\vec{x}^{(3)}|
   =
   |\vec{x}^{(3)}-\vec{x}^{(1)}|,
 \end{equation*}
 hence $\vec{x}^{(1)}$, $\vec{x}^{(2)}$, $\vec{x}^{(3)}$ are vertices of
 some equilateral triangle. Thus, \eqref{eq:4.16} is a more general
  ``geometric'' condition on $\vec{x}^{(j)}$ and the grain boundary
  energy density $\sigma(\Delta \vec{\alpha})$ to observe Boltzmann
  distribution for a grain boundary energy density as a steady-state
  distribution for $\rho_1(\Delta\vec{\alpha} , t)$.
\end{remark}

\section{Numerical Experiments}\label{sec13}
Here, we present several numerical experiments to illustrate
consistency of the proposed stochastic model (\ref{eq:2.9}) with a
grain growth model (\ref{eq:2.2}) applied to a grain boundary network
that undergoes critical/disappearance events, e.g., grain
disappearance, facet/grain boundary disappearance, facet interchange,
splitting of unstable junctions.  We define the total grain
boundary energy of the network, like,
\begin{equation} \label{eq:6.4e}
 E(t)
  =
  \sum_{j}
  \sigma
  (
  \Delta^{(j)}\alpha
  )
  |\Gamma_t^{(j)}|,
\end{equation}
where $\Delta^{(j)}\alpha$ is a misorientation, a difference between the
lattice orientation of the two neighboring grains which form the grain
boundary $\Gamma^{(j)}_t$. Then, the energetic variational principle
implies,
\begin{equation}
 \label{eq:6.4n}
  \left\{
  \begin{aligned}
   v_n^{(j)}
   &=
   \mu
   \sigma
   (
   \Delta^{(j)}\alpha
   )
   \kappa^{(j)},\quad\text{on}\ \Gamma_t^{(j)},\ t>0, \\
   \frac{d\alpha^{(k)}}{dt}
   &=
   -\gamma
   \frac{\delta E}{\delta \alpha^{(k)}},
   \\
   \frac{d\vec{a}^{(l)}}{dt}
   &=
   \eta
   \sum_{\vec{a}^{(l)}\in\Gamma^{(j)}_t}
   \left(\sigma(\Delta^{(j)}\alpha)\frac{\vec{b}^{(j)}}{|\vec{b}^{(j)}|}\right),
   \quad t>0.
  \end{aligned}
  \right.
\end{equation}
First, we will test ``generalized'' Herring condition
(\ref{eq:4.8}), as well as relations (\ref{eq:4.18}) and
(\ref{eq:4.16}) for the grain boundary network (\ref{eq:6.4n}). Next,
in our numerics, using grain boundary character
distribution (GBCD) statistics (see for example, \cite{DK:BEEEKT,DK:gbphysrev,
  MR2772123, MR3729587}),  we will
illustrate  that the grain
growth system (\ref{eq:6.4n}) exhibits some fluctuation-dissipation principles (see
Section \ref{sec4}). 
\par Therefore, to verify first ``generalized'' Herring condition (\ref{eq:4.8}), we define
ratio $R_1$,
\begin{equation}\label{eqR1}
  R_1: =
   \frac{(\sigma(\Delta^{(i)}\alpha)+\sigma(\Delta^{(i+1)}\alpha))^2
   -
   \sigma(\Delta^{(i+2)}\alpha)^2}
   {2
     (1-\cos\psi^{(i)})\sigma(\Delta^{(i)}\alpha)\sigma(\Delta^{(i+1)}\alpha)},
   i=1, 2, 3, ...   
 \end{equation}
To verify relations  (\ref{eq:4.18}) and (\ref{eq:4.16}), we define
ratio $R_2$ and $R_3$ respectively for each triple junction $\vec{a}$,
\begin{equation}\label{eqR2}
  R_2: = \frac{ |\vec{x}^{(3)}-\vec{x}^{(1)}|^2\sqrt{(1-\cos\psi^{(1)})(1-\cos\psi^{(2)})}}{(1-\cos\psi^{(3)}) |\vec{x}^{(1)}-\vec{x}^{(2)}||\vec{x}^{(2)}-\vec{x}^{(3)}|},
\end{equation}
and 
\begin{equation}\label{eqR3}
  R_3: =\frac{\sqrt{\mathcal{R}_1 \cdot \mathcal{R}_2}}{\mathcal{R}_3} ,
\end{equation}
where $\mathcal{R}_i:=\frac{(\sigma(\Delta^{(i)}\alpha)+\sigma(\Delta^{(i+1)}\alpha))^2
   -
   \sigma(\Delta^{(i+2)}\alpha)^2}
   {2|\vec{x}^{(i)}-\vec{x}^{(i+1)}|^2\sigma(\Delta^{(i)}\alpha)\sigma(\Delta^{(i+1)}\alpha)}$
   and $\vec{x}^{(i)}\neq \vec{a}$ are any node along grain boundary
   with triple junction $\vec{a}$. Note that formulas for $R_2$ and
   $R_3$ (\ref{eqR2})-(\ref{eqR3}) require selection of the node $x_i$
   along the grain boundary different from the triple junction
   $\vec{a}$, see for example Figs.~\ref{figTJ2} and \ref{figTJ1}.
   Note also that for $j=1, 2,3$, $R_j$ is a dimensionless quantity
   with respect to the length of grain boundaries. If the formula
   \eqref{eq:4.8}, \eqref{eq:4.18}, or \eqref{eq:4.16} holds, then
   $R_j=1$ ($j=1,2,3$), respectively. Since 
   \eqref{eq:4.8}, \eqref{eq:4.18}, and \eqref{eq:4.16} are local
   relations (and not the property of the network), in
   our numerical experiments we compute probability densities for
   $R_1$, as well as for
   $R_2$ and $R_3$ (using two choices of the node $x_i$ to compute $R_2$ and $R_3$). In
   Figs.~\ref{fig2}--\ref{fig3}, \ref{fig5}, \ref{fig7}, \ref{fig8}
   (left plot) and \ref{fig10} (left and middle plots) we selected
   $x_i$ to be a mesh node on the grain boundary which is the closest
   to the triple junction $\vec{a}$ (note, we discretize each grain
   boundary using linear line segments, hence, end points of these
   line segments form mesh nodes on each grain boundary). As a second
   choice for the node $x_i$, see Fig.~\ref{fig4}, we selected $x_i$
   to be the other end point of the grain boundary/the ``other triple
   junction'' (different from the triple junction of $\vec{a}$) of the
   considered grain boundary that shares $\vec{a}$. As our results show,
   choice of $x_i$ affects the distributions for $R_2$ and
   $R_3$. However, the choice of $x_i$ does not affect consistency
   property  reflected by distributions for $R_2$ and $R_3$ between
   developed stochastic model (\ref{eq:2.9}) and the simulated grain
   growth system (\ref{eq:6.4n}), see Figs.~\ref{fig2}--\ref{fig3}, \ref{fig5}, \ref{fig7}, \ref{fig8}
   (left plot), \ref{fig10} (left and middle plots) and Fig.~\ref{fig4}.
\par Further,
we will investigate the distribution of the grain boundary character
distribution (GBCD) $\rho(\Delta ^{(j)} \alpha)$ at $T_{\infty}$
($T_{\infty}$ is defined below),  and we will use GBCD
to illustrate  that the grain
growth system (\ref{eq:6.4n}) exhibits some fluctuation-dissipation principles (see
Section \ref{sec4}).
The GBCD (in our context) is an empirical
statistical measure of the relative length (in 2D) of the grain
boundary interface with a given lattice misorientation,
\begin{eqnarray}
&\rho(\Delta ^{(j)} \alpha, t) =\mbox{ relative length of interface of
  lattice misorientation } \Delta ^{(j)} \alpha \mbox{ at time }t,\nonumber\\
&\mbox{ normalized so that } \int_{\Omega_{\Delta ^{(j)} \alpha}} \rho
  d \Delta ^{(j)} \alpha=1, \label{eq:7.3}
\end{eqnarray}
where we consider $\Omega_{\Delta ^{(j)} \alpha}=[-\frac{\pi}{4},
\frac{\pi}{4}]$ in the numerical experiments below (for planar grain
boundary network, it is reasonable to consider such range for the
misorientations). For more details, see for example \cite{DK:gbphysrev}. In all our tests below,
we compare  GBCD at $T_{\infty}$ to the stationary solution of the
Fokker-Planck equation, the  Boltzmann distribution for the grain
boundary energy density $\sigma
  (
  \Delta^{(j)}\alpha
  )$,
\begin{equation} \begin{aligned} \label{eq:7.4}
& \rho_D(\Delta^{(j)}\alpha) = \frac{1}{Z_D}e^{-\frac{\sigma(\Delta^{(j)}\alpha)}{D}}, \\
& \textrm{with partition function, i.e.,normalization factor} \\
&Z_{D} = \int_{\Omega_{\Delta^{(j)}\alpha}}
e^{-\frac{\sigma(\Delta^{(j)}\alpha)}{D}}d\Delta^{(j)} \alpha,
\end{aligned}\end{equation}
\cite{DK:BEEEKT,DK:gbphysrev,MR2772123,MR3729587} and see Section~\ref{sec10}. We employ
Kullback-Leibler relative entropy test to obtain a unique
``temperature-like'' parameter $D$ and to construct the corresponding
Boltzmann distribution for the GBCD at $T_{\infty}$ as it was originally
done in
\cite{DK:BEEEKT,DK:gbphysrev,MR2772123,MR3729587}. Kullback-Leibler
(KL) relative entropy test
\cite{DK:BEEEKT,DK:gbphysrev,MR2772123,MR3729587} is based on the idea
that if we know that  the GBCD $\rho(\Delta ^{(j)} \alpha, t)$ evolves according to the Fokker-Planck
equation, then it must converge
exponentially fast to  $\rho_D(\Delta ^{(j)}\alpha)$  in KL relative entropy as $t\to \infty$.
Note, that the GBCD is a primary candidate to characterize texture of the grain boundary
network, and is inversely related to the grain boundary energy density
as discovered in experiments and simulations.
The
reader can consult, for example, \cite{DK:BEEEKT,DK:gbphysrev,
  MR2772123, MR3729587} for more details about GBCD and the theory of
the GBCD.
 In the numerical experiments in this paper, we consider the
grain boundary energy density as plotted in Fig.~\ref{gbend} and given below,
\[\sigma(\Delta^{(j)} \alpha)=1+0.25\sin^2(2\Delta^{(j)} \alpha).\]
\par  Our simulation of 2D grain
boundary network \cite{Katya-Chun-Mzn2} is a further extension of the algorithm based on sharp interface
approach  \cite{MR2772123,MR3729587} (note, that in
\cite{MR2772123,MR3729587},  only Herring conditions at
triple junctions were considered, i.e., $\eta\to \infty$, and dynamic
orientations/misorientations (``rotation of grains'') was absent, i.e.,
$\gamma=0$). We recall that in the numerical scheme
we work with a variational principle. 
\begin{figure}[hbtp]
\centering
\vspace{-2.0cm}
\includegraphics[width=3.0in]{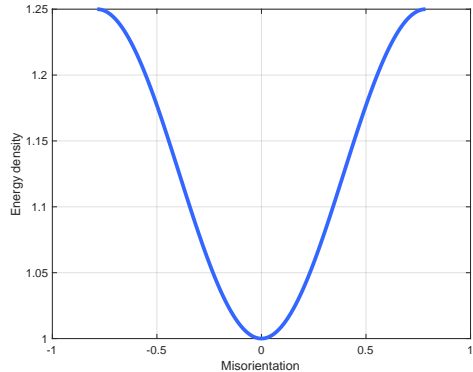}
\vspace{-2.0cm}
\caption{\footnotesize Grain boundary energy density function $\sigma(\Delta \alpha)$.}\label{gbend}
\end{figure}
The cornerstone of the algorithm, which assures its
stability, is the discrete dissipation inequality for the total grain boundary energy that holds when
either the discrete Herring boundary condition ($\eta \to \infty$) or  discrete ``dynamic boundary
condition'' (finite mobility $\eta$ of the triple junctions,  third
equation of (\ref{eq:6.4n})) is
satisfied at the triple junctions. We also recall that in the numerical algorithm we impose
Mullins theory (first equation of (\ref{eq:6.4n})) as the local
evolution law for the grain boundaries (and the relaxation time scale $\mu$ is
kept finite).  For more details about computational model
based on Mullins equations (curvature driven growth), the reader can
consult, for example \cite{MR2772123,MR3729587,Katya-Chun-Mzn2}. In addition, in our
final test Fig.~\ref{fig10}, we also
compare results of ``curvature model''  ($\mu$ is finite)
(\ref{eq:6.4n}) with a
results of ``vertex model'' ($\mu \to \infty$), grain boundaries are
straight lines, and hence, only second and third
evolution equations of  (\ref{eq:6.4n}) are considered for the vertex
model, namely model (\ref{eq:2.3}) which is applied to the grain boundary
network is studied.

\par In all the numerical tests below we initialized our system with
$\mathcal{N}$ grains cells/grains with normally distributed misorientation angles at
initial time $t=0$. We also assume that the final time of the
simulations $T_{\infty}$ is the time when
approximately $80\%$ of grains disappeared from the system. The final time is
selected based on the system with no dynamic misorientations
($\gamma=0$) and with Herring condition at the triple junctions ($\eta
\to \infty$)
and, it is selected to ensure that statistically significant number of grains still remain
in the system and
the system reached its statistical stead-state. Therefore, all the
numerical results which are presented below are for the grain boundary
system that undergoes critical/disappearance events. We also denote by
$T_0$ the initial time (before first time step) and by $T_1$ we denote a
time after a first time step.
\par First, we consider grain growth model with curvature
(\ref{eq:6.4n}) and we study three systems with $\mathcal{N}=10000$
initial grains, the first system has $\gamma=10$ and $\eta=100$, the
second system has  $\gamma=100$ and $\eta=1000$, and the third system
has $\gamma=1000$ and $\eta\to \infty$ (Herring condition). We check ``generalized'' Herring condition formula
(\ref{eq:4.8}) by computing probability density for ratio $R_1$,
(\ref{eqR1}) and by computing time evolution of frequency of dihedral
angles that satisfy ratio $R_1$ with $0.01$ accuracy.  The
results for $R_1$ are plotted on Fig.~\ref{fig1}. We observe that all three
distributions of $R_1$ (left and middle plots) for all three grain growth systems have peak at $1$
which is consistent with the ``generalized'' Herring condition formula
(\ref{eq:4.8}). In addition, larger values of $\gamma$ and of $\eta$
provide a higher accuracy for ratio $R_1$ and, in addition, 
produce a higher peak of the distribution at $1$. The distribution
of $R_1$ for system with $\gamma=1000$ and $\eta\to \infty$ (Herring
condition) looks like a delta function positioned at $1$ which is
again consistent with results for the developed stochastic model
Sections \ref{sec4}-\ref{sec10}. Next, we check relations
(\ref{eq:4.18}) and (\ref{eq:4.16}) for the same three grain growth
systems (\ref{eq:6.4n}) by computing probability densities for ratio
$R_2$ and $R_3$, (\ref{eqR2})-(\ref{eqR3}). The results are presented
in Figs.~\ref{fig2}--\ref{fig5}. Again, we observe that the peaks of
the distributions for $R_2$ and $R_3$ for all three systems are near
$1$. Moreover, the agreement between distributions $R_2$ and $R_3$ is
better for grain growth systems with larger values of $\gamma$ and
$\eta$ ( for $\gamma=1000$ and $\eta\to \infty$, the plots for $R_2$
and $R_3$ are almost indistinguishable, see Figs.~\ref{fig3} (left
plot) and \ref{fig4} (right plot)), which is again consistent with a
developed theory, see Section \ref{sec10}. In addition, on
Fig.~\ref{fig5}, we illustrate how distribution for ratio $R_3$
evolves with time for grain growth system with $\gamma=10$ and
$\eta=100$, Fig.~\ref{fig5} (left plot) and with $\gamma=1000$ and
$\eta\to\infty$ (Herring condition) Fig.~\ref{fig5} (right
plot). The results illustrate that the distributions are ``defined'' by the
grain growth evolution equations and not by the initial
distribution. Finally, in the last test for the considered three grain
growth systems, we compute GBCD statistics at time
$T_{\infty}$. First, we
observe that the GBCD at $T_{\infty}$ is well-approximated by the Boltzmann
distribution for the grain boundary energy density see
Fig.~\ref{fig6}, which is consistent with the theory developed
in the work \cite{DK:BEEEKT,DK:gbphysrev,
  MR2772123, MR3729587} and is consistent with the stochastic model and theory
developed in this work, Sections \ref{sec4}-\ref{sec10}. Furthermore, as concluded from our numerical
results Fig.~\ref{fig6},  grain growth systems with larger values
of $\gamma$ and $\eta$, give smaller diffusion
coefficient/''temperature''-like parameter $D$ for the GBCD at $T_{\infty}$, and hence higher GBCD
peak near misorientation $0$. This is in agreement with
dissipation-fluctuation relations (\ref{eq:2.10}), Section \ref{sec4}.
\par Next, we consider grain growth systems with different number of
grains at initial time $T_0$, Figs.~\ref{fig7}-\ref{fig8}. Namely, we
consider grain growth systems (\ref{eq:6.4n}) with $\mathcal{N}=1000$,
$\mathcal{N}=2500$, $\mathcal{N}=10000$ and with $\mathcal{N}=20000$
grains initially, at time $T_0$. For these systems, we assume no
dynamic misorientation ($\gamma=0$) and Herring condition ($\eta\to
\infty$) at the triple junctions. From the results, Fig.~\ref{fig7}
and \ref{fig8} (middle and right plots) we observe that distributions
for $R_2$, $R_3$ and GBCD exhibit convergence to limiting
distributions with increase in $\mathcal{N}$. In addition, result on
Fig.~\ref{fig8} (left plot), indicates that there is a closer
agreement between distributions $R_2$ and $R_3$ for larger value of
misorientation parameter
$\gamma$. Again, this is consistent with the developed theory, Section
\ref{sec10}. In Fig.~\ref{fig9}, we investigate effect of the mobility
of the triple junctions $\eta$ on the GBCD, however we do not observe
as much effect of $\eta$ on the GBCD as we observed for the
misorientation parameter $\gamma$, see Figs.~\ref{fig6} and
\ref{fig9}. This can be due to more profound effect of the
interactions among triple junctions/correlations effects among triple
junctions that should be taken into account as a part of future
extension of the proposed stochastic model (\ref{eq:2.9}).
\par Finally, in the last test,  Fig.~\ref{fig10}, we 
compare results of ``curvature model''  ($\mu$ is finite) (\ref{eq:6.4n}) with a
results of ``vertex model'' ($\mu \to \infty$), grain boundaries are
straight lines, and hence, only second and third
evolution equations of  (\ref{eq:6.4n}) are considered for the vertex
model, namely model (\ref{eq:2.3}) which is applied to the grain boundary
network is studied. As can be seen from results in Fig.~\ref{fig10},
there is not much effect on the GBCD. However, we observe significant
effect on the distributions of $R_2$ and $R_3$, Fig.~\ref{fig10} (left
and middle plots), namely ``curvature model'' appears to be in closer
agreement with the developed stochastic model  (\ref{eq:2.9}) than
``vertex model''.  This again highlights the importance of
correlations and their effects on grain growth.
Therefore, as a part of future work, we will study
interactions/correlations and their effects on coarsening in
polycrystalline materials. 

\begin{figure}[hbtp]
\centering
\vspace{-1.8cm}
\includegraphics[width=2.1in]{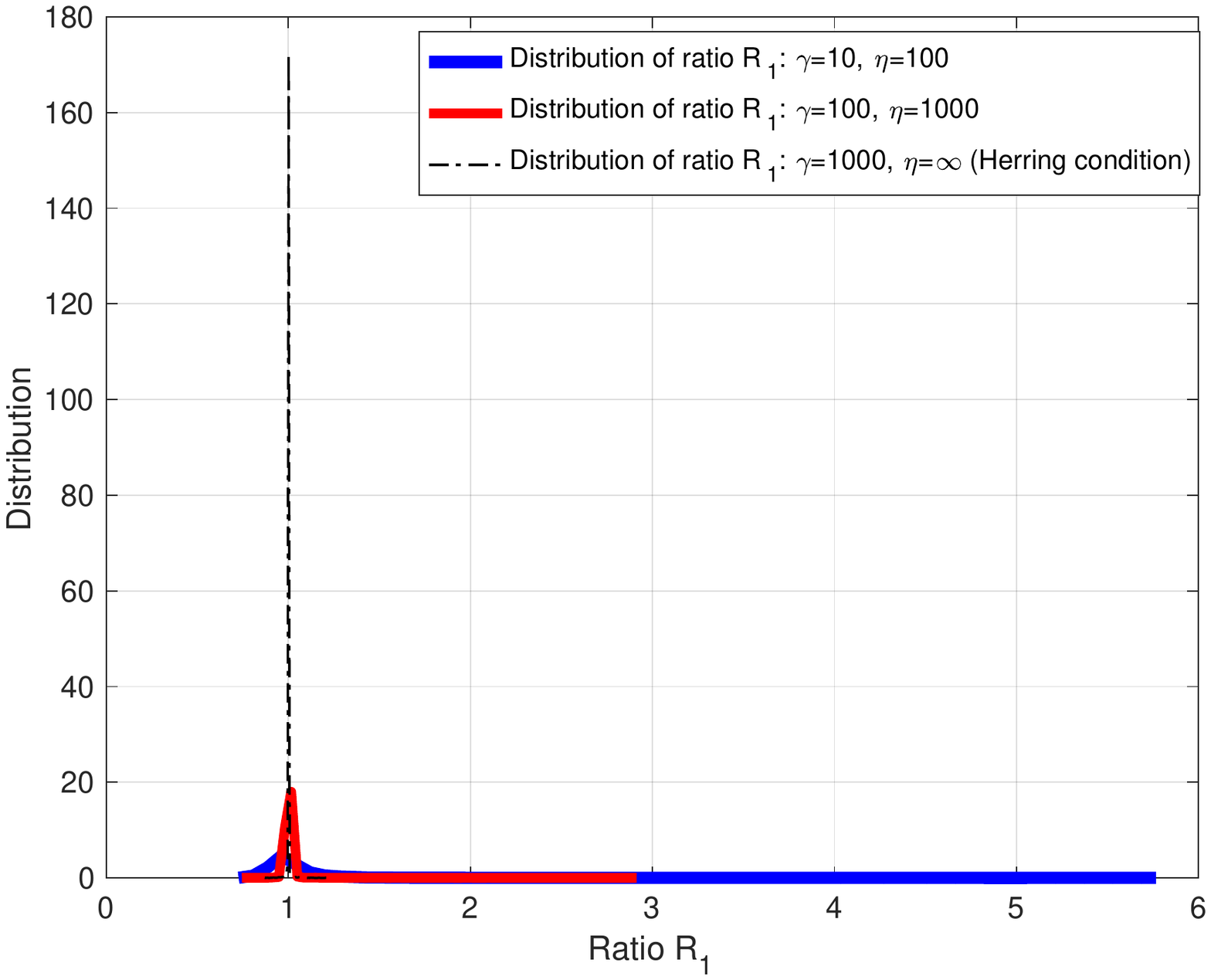}\hspace{-0.7cm}
\includegraphics[width=2.1in]{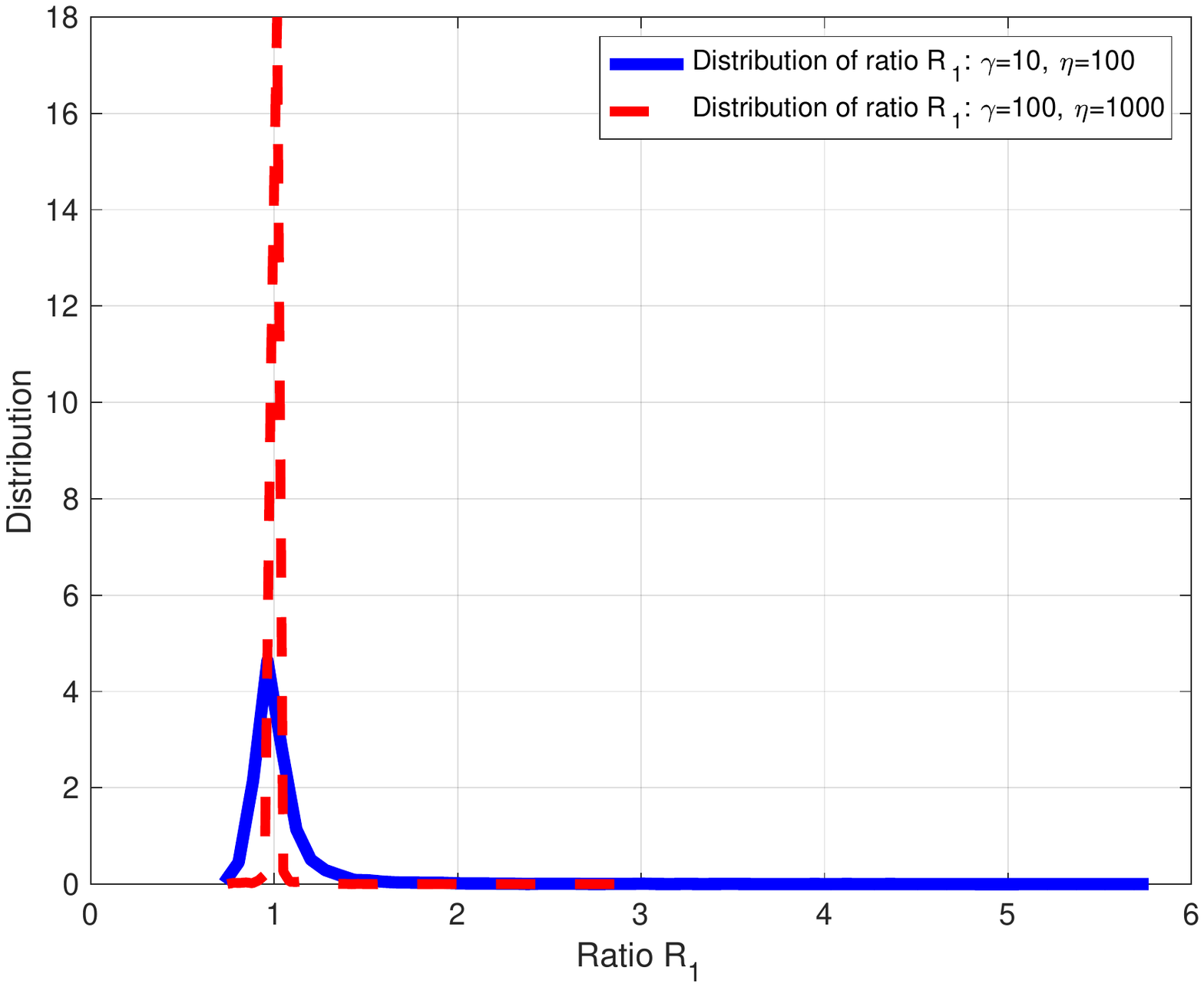}\hspace{-0.7cm}
\includegraphics[width=2.1in]{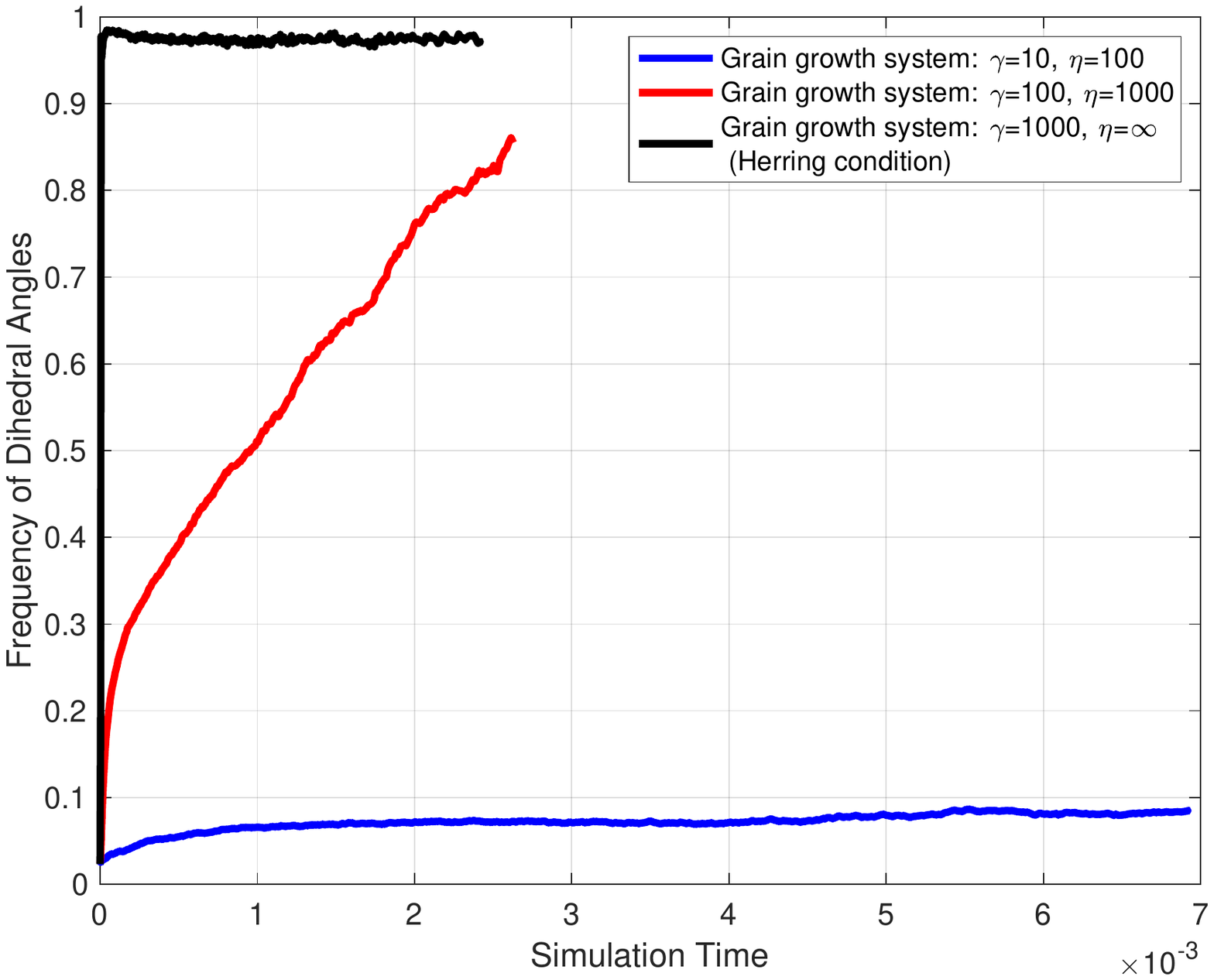}
\vspace{-1.8cm}
\caption{\footnotesize Grain growth system (\ref{eq:6.4n}) with finite
  $\mu$ (with curvature), one run of $2$D trial with $10000$ initial
  grains: {\it (a) Left plot,} distribution of ratio $R_1$ (\ref{eqR1})
  for grain growth systems with mobility of triple junctions $\eta=100$
  and the misorientation parameter $\gamma=10$ (solid blue),  with mobility of triple junctions $\eta=1000$
  and the misorientation parameter $\gamma=100$ (solid red) and with
  mobility of triple junctions $\eta\to \infty$ (Herring condition)
  and the misorientation parameter $\gamma=1000$ (dashed point black).
 {\it (b) Middle plot}, comparison of the two distributions of ratio $R_1$ (\ref{eqR1})
  for grain growth systems with mobility of triple junctions $\eta=100$
  and the misorientation parameter $\gamma=10$ (solid blue) and with mobility of triple junctions $\eta=1000$
  and the misorientation parameter $\gamma=100$ (dashed red). The
  distributions are plotted at $T_{\infty}$.  {\it (c) Right plot},
  time evolution of
  frequency of dihedral angles that satisfy ratio $R_1$ with $0.01$
  accuracy  for grain growth systems with mobility of triple junctions $\eta=100$
  and the misorientation parameter $\gamma=10$ (solid blue),  with mobility of triple junctions $\eta=1000$
  and the misorientation parameter $\gamma=100$ (solid red) and with
  mobility of triple junctions $\eta\to \infty$ (Herring condition)
  and the misorientation parameter $\gamma=1000$ (solid black).}\label{fig1}
\end{figure}

\begin{figure}[hbtp]
\centering
\vspace{-1.8cm}
\includegraphics[width=3.0in]{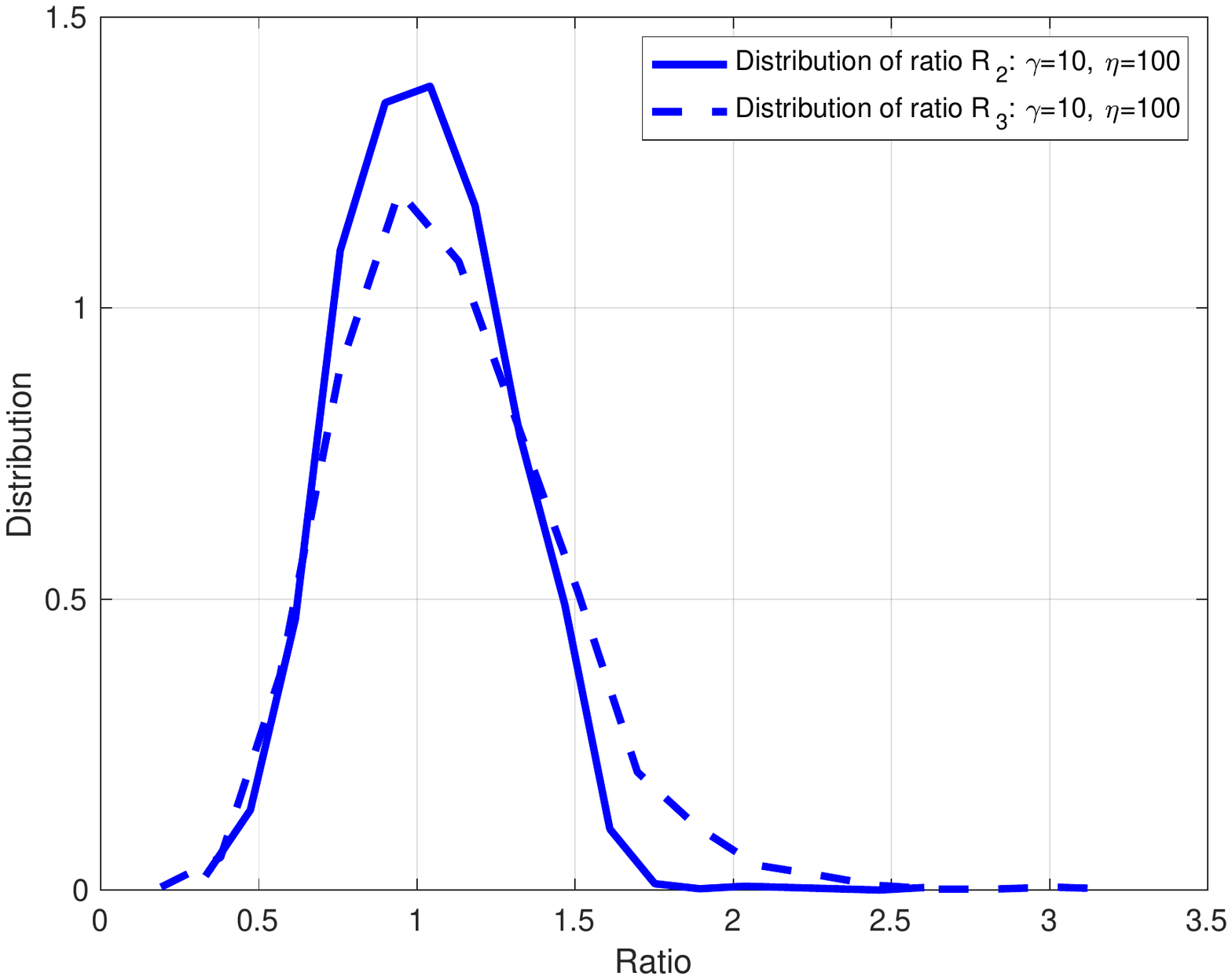}
\includegraphics[width=3.0in]{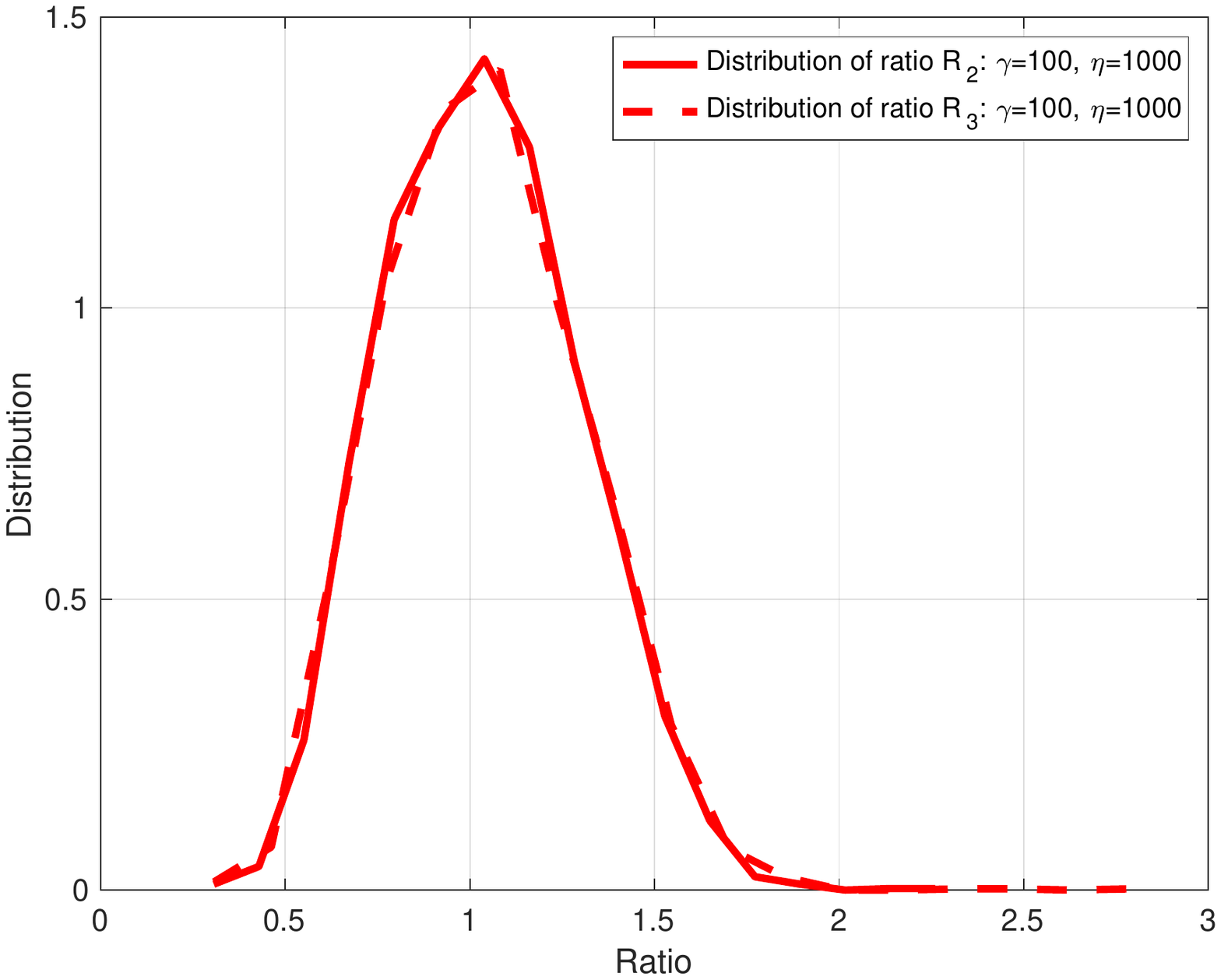}
\vspace{-1.8cm}
\caption{\footnotesize Grain growth system (\ref{eq:6.4n}) with finite
  $\mu$ (with curvature), one run of $2$D trial with $10000$ initial
  grains:  {\it (a) Left plot,} comparison of distributions of ratio
  $R_2$  (\ref{eqR2}) (solid blue) and $R_3$  (\ref{eqR3}) (dashed blue)
  for grain growth system with mobility of triple junctions $\eta=100$
  and the misorientation parameter $\gamma=10$.
 {\it (b) Right plot}, comparison of distributions of ratio
  $R_2$  (\ref{eqR2}) (solid red) and $R_3$  (\ref{eqR3}) (dashed red)
  for grain growth system with mobility of triple junctions $\eta=1000$
  and the misorientation parameter $\gamma=100$. The closest mesh node
  of the grain boundary to
  the triple junction $\vec{a}$ is used as
  $x_i$. The distributions are plotted at $T_{\infty}$.}\label{fig2}
\end{figure}

\begin{figure}[hbtp]
\centering
\vspace{-1.8cm}
\includegraphics[width=3.0in]{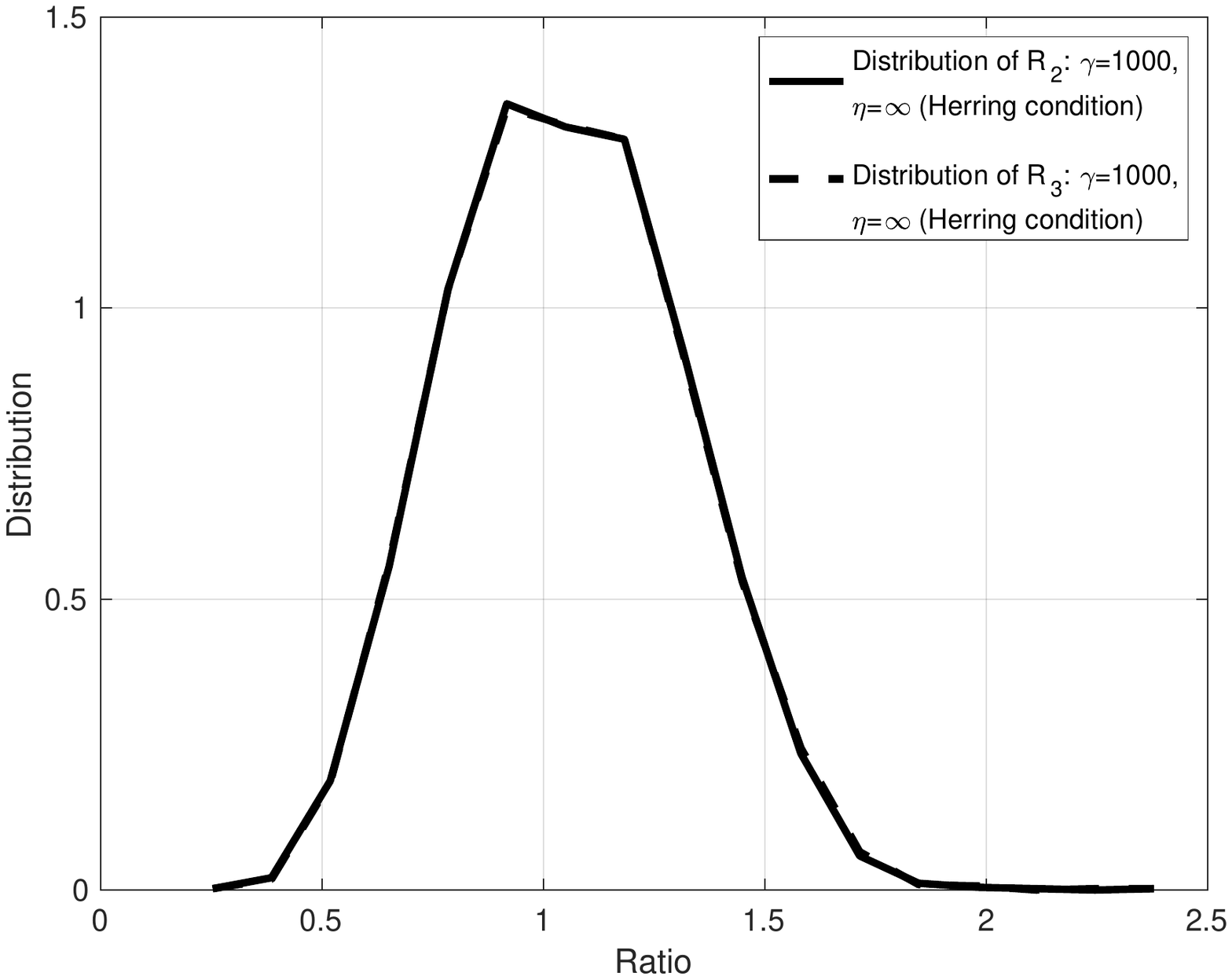}
\includegraphics[width=3.0in]{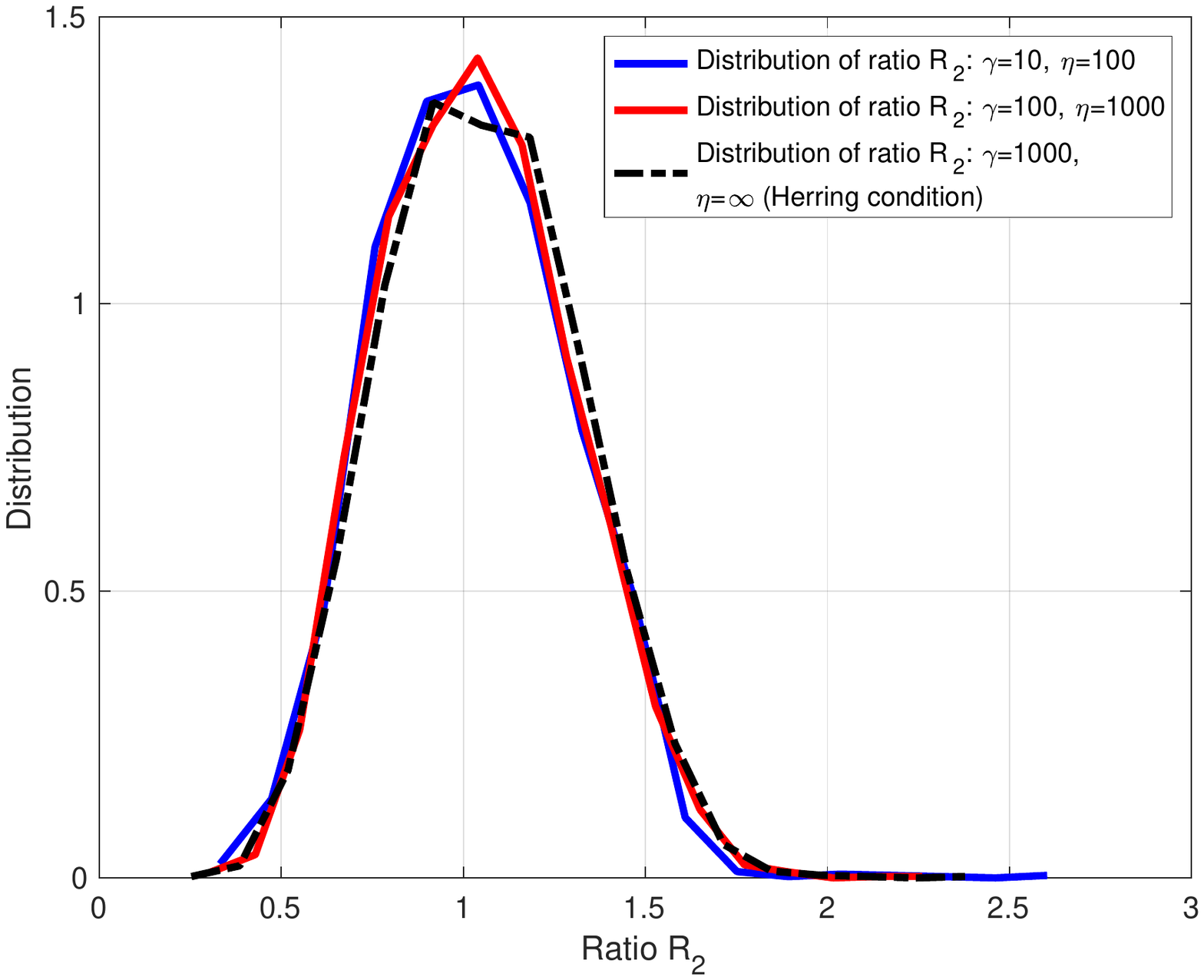}
\vspace{-1.8cm}
\caption{\footnotesize Grain growth system (\ref{eq:6.4n}) with finite
  $\mu$ (with curvature),  one run of $2$D trial with $10000$ initial
  grains: {\it (a) Left plot,} comparison of distributions of ratio
  $R_2$  (\ref{eqR2})  (solid black) and $R_3$  (\ref{eqR3}) (dashed black)
  for grain growth system with mobility of triple junctions $\eta \to
  \infty$ (Herring condition)
  and the misorientation parameter $\gamma=1000$.
 {\it (b) Right plot}, comparison of distributions of ratio
  $R_2$ (\ref{eqR2})  for grain growth systems with mobility of triple junctions $\eta=100$
  and the misorientation parameter $\gamma=10$ (solid blue),  with mobility of triple junctions $\eta=1000$
  and the misorientation parameter $\gamma=100$ (solid red), and with
  mobility of triple junctions $\eta\to \infty$ (Herring condition)
  and the misorientation parameter $\gamma=1000$ (dashed point
  black). The closest mesh node
  of the grain boundary to
  the triple junction $\vec{a}$ is used as
  $x_i$. The distributions are plotted at $T_{\infty}$.}\label{fig3}
\end{figure}

\begin{figure}[hbtp]
\centering
\vspace{-1.8cm}
\includegraphics[width=2.1in]{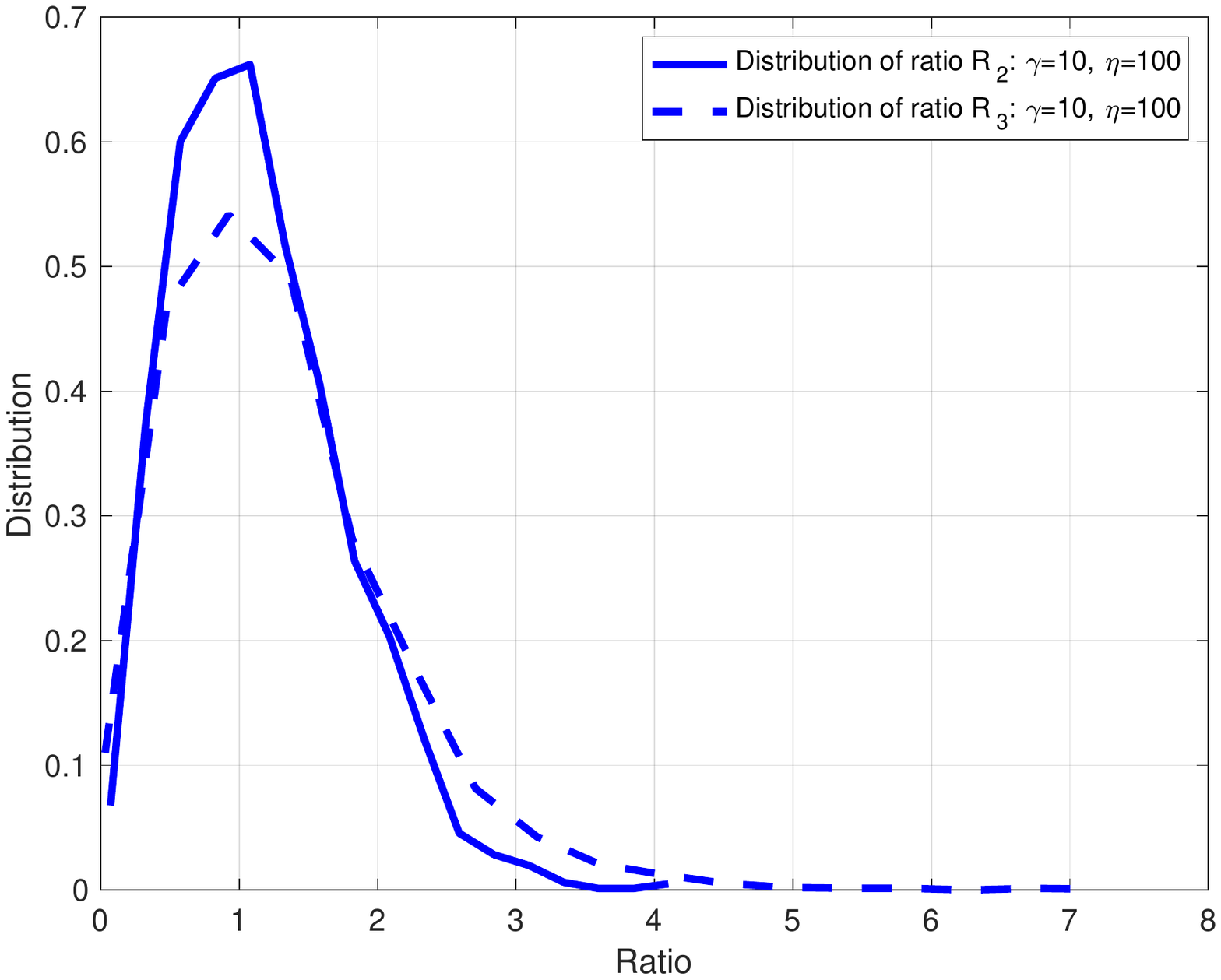}\hspace{-0.7cm}
\includegraphics[width=2.1in]{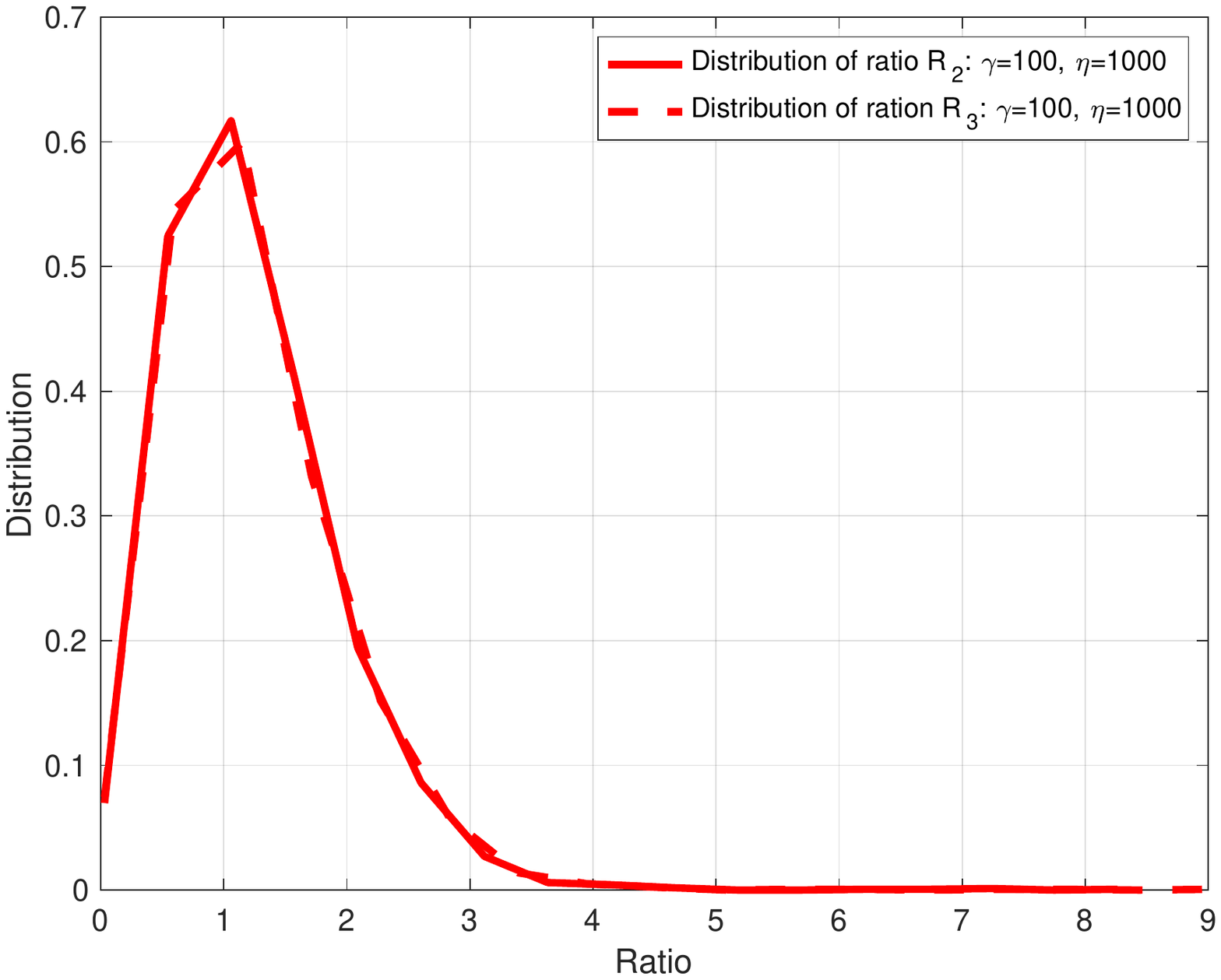}\hspace{-0.7cm}
\includegraphics[width=2.1in]{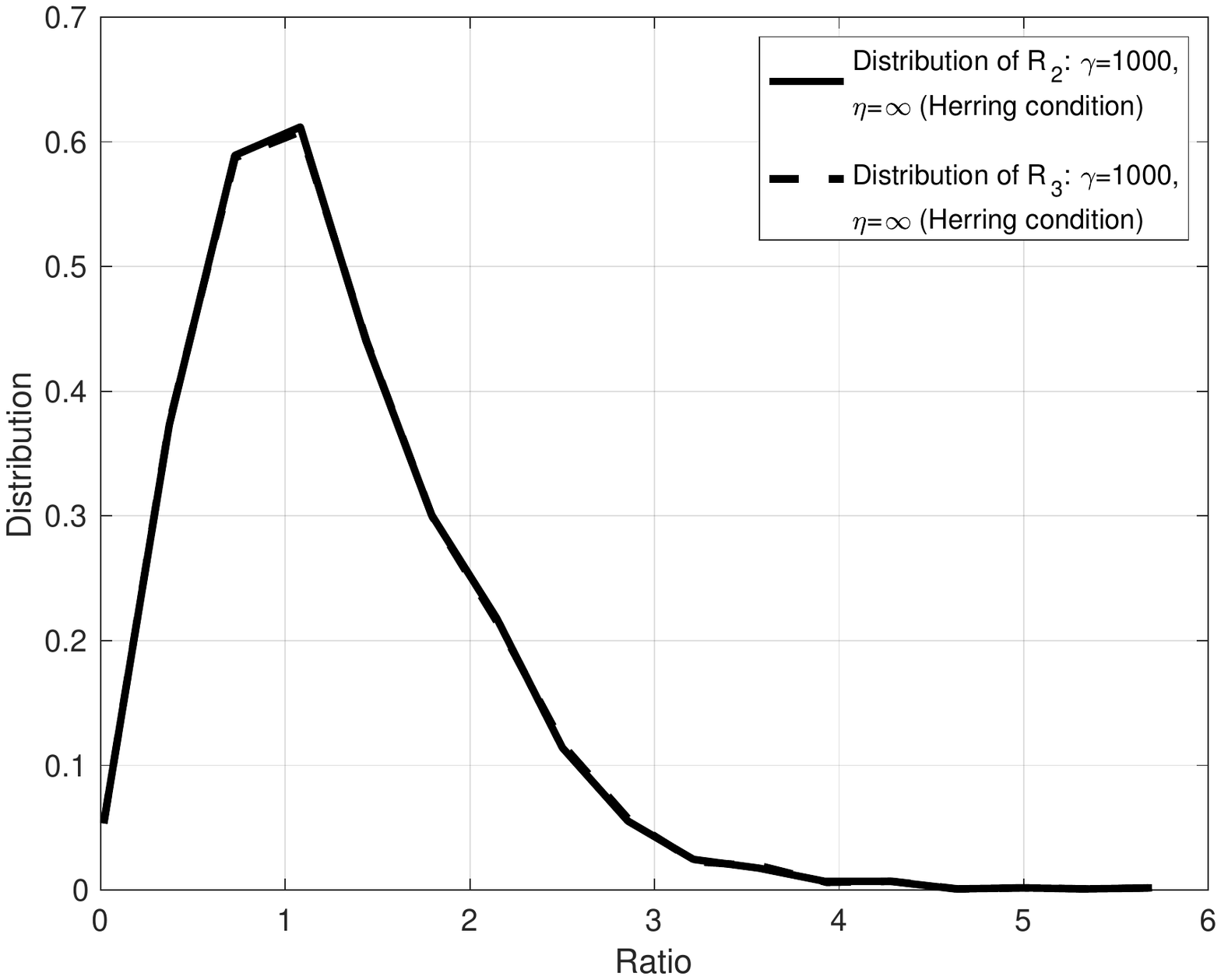}
\vspace{-1.8cm}
\caption{\footnotesize Grain growth system (\ref{eq:6.4n}) with finite
  $\mu$ (with curvature),  one run of $2$D trial with $10000$ initial
  grains:  {\it (a) Left plot,} comparison of distributions of ratio
  $R_2$  (\ref{eqR2}) (solid blue) and $R_3$  (\ref{eqR3}) (dashed blue)
  for grain growth system with mobility of triple junctions $\eta=100$
  and the misorientation parameter $\gamma=10$.  {\it (b) Middle
    plot,} comparison of distributions of ratio
  $R_2$  (\ref{eqR2}) (solid red) and $R_3$  (\ref{eqR3}) (dashed red)
  for grain growth systems with mobility of triple junctions $\eta=1000$
  and the misorientation parameter $\gamma=100$.
 {\it (c) Right plot}, comparison of distributions of ratio
  $R_2$  (\ref{eqR2}) (solid black) and $R_3$  (\ref{eqR3}) (dashed black)
  for grain growth system with mobility of triple junctions  $\eta\to \infty$
  and the misorientation parameter $\gamma=1000$. The other triple
  junction (different from $\vec{a}$) of the given grain boundary is used as
  $x_i$ here. The distributions are plotted at $T_{\infty}$.}\label{fig4}
\end{figure}

\begin{figure}[hbtp]
\centering
\vspace{-1.8cm}
\includegraphics[width=3.2in]{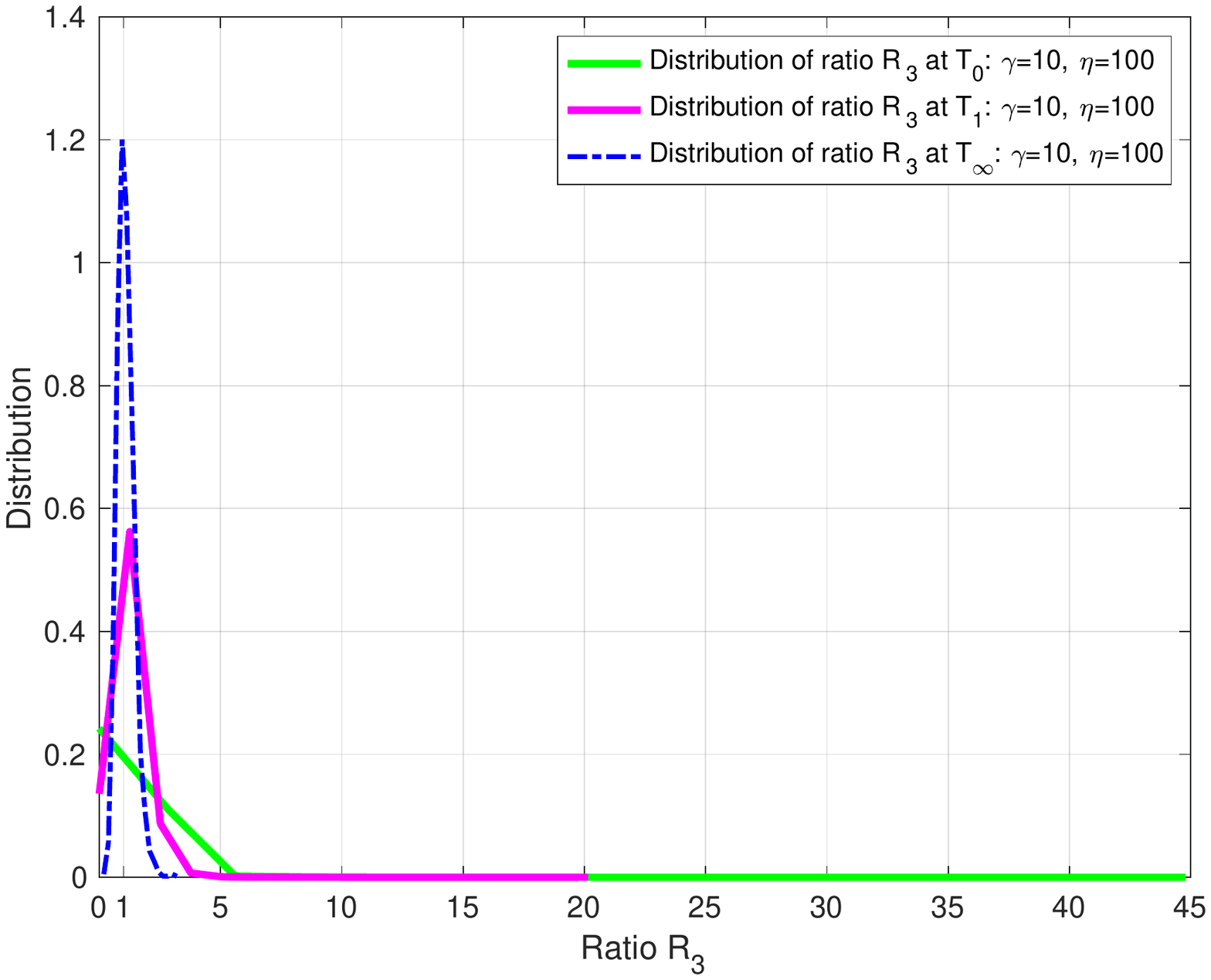}\hspace{-0.7cm}
\includegraphics[width=3.2in]{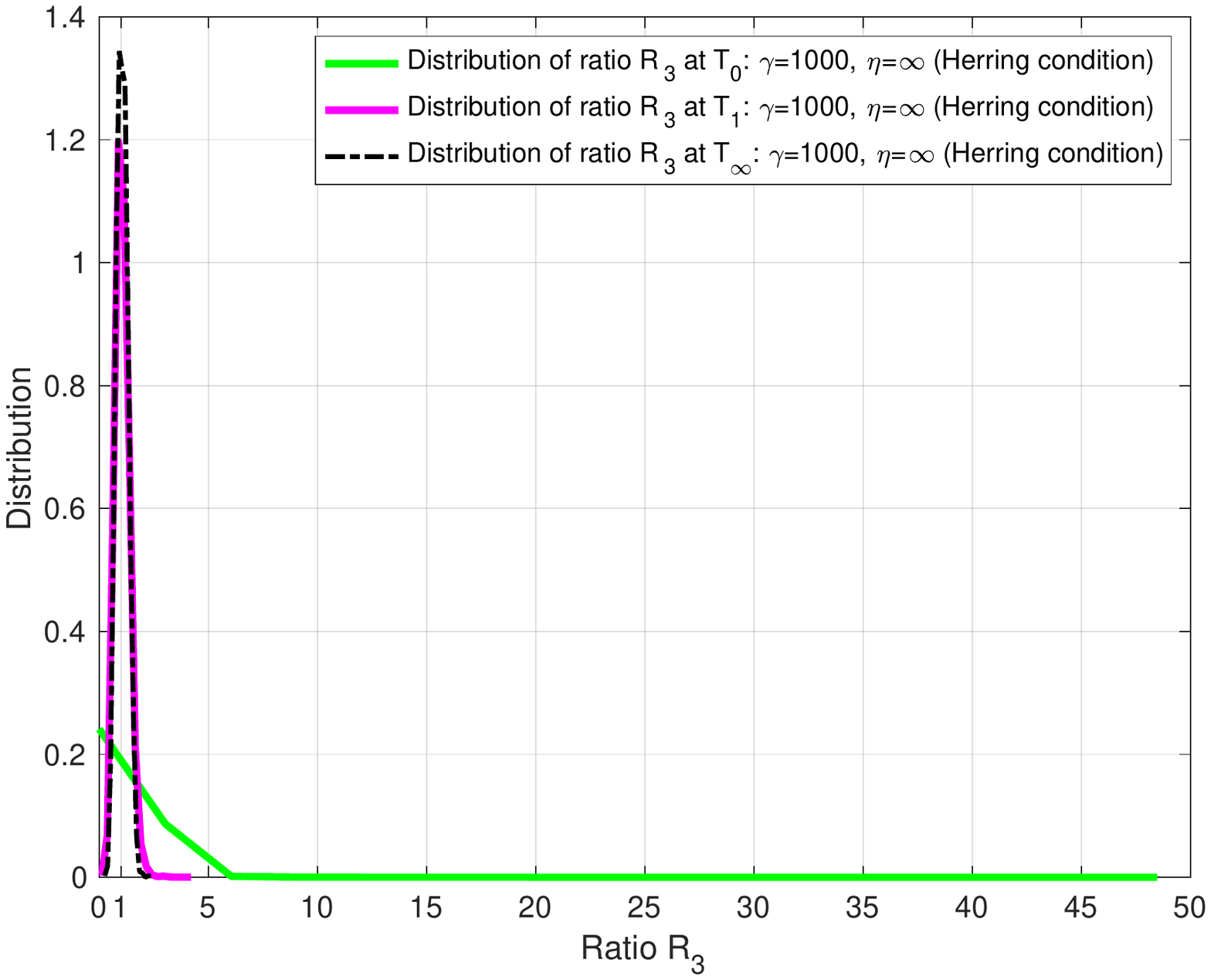}
\vspace{-1.8cm}
\caption{\footnotesize Grain growth system (\ref{eq:6.4n}) with finite
  $\mu$ (with curvature), one run of $2$D trial with $10000$ initial
  grains:  {\it (a) Left plot,} distributions of ratio
  $R_3$  (\ref{eqR3})  at initial time $T_0$ (solid green), at a time
  $T_1$ after a first time step (solid magenta) and at a final time
  $T_{\infty}$ (dashed point blue)
  for grain growth system with mobility of triple junctions $\eta =100$ 
  and the misorientation parameter $\gamma=10$.
 {\it (b) Right plot}, distributions of ratio
  $R_3$  (\ref{eqR3}) at initial time $T_0$  (solid green), at a time
  $T_1$ after a first time step (solid magenta) and at a final time
  $T_{\infty}$ (dashed point black)
  for grain growth system with mobility of triple junctions $\eta \to
  \infty$ (Herring condition)
  and the misorientation parameter $\gamma=1000$.  The closest mesh node
  of the grain boundary to
  the triple junction $\vec{a}$ is used as
  $x_i$.}\label{fig5}
\end{figure}

\begin{figure}[hbtp]
\centering
\vspace{-1.8cm}
\includegraphics[width=2.1in]{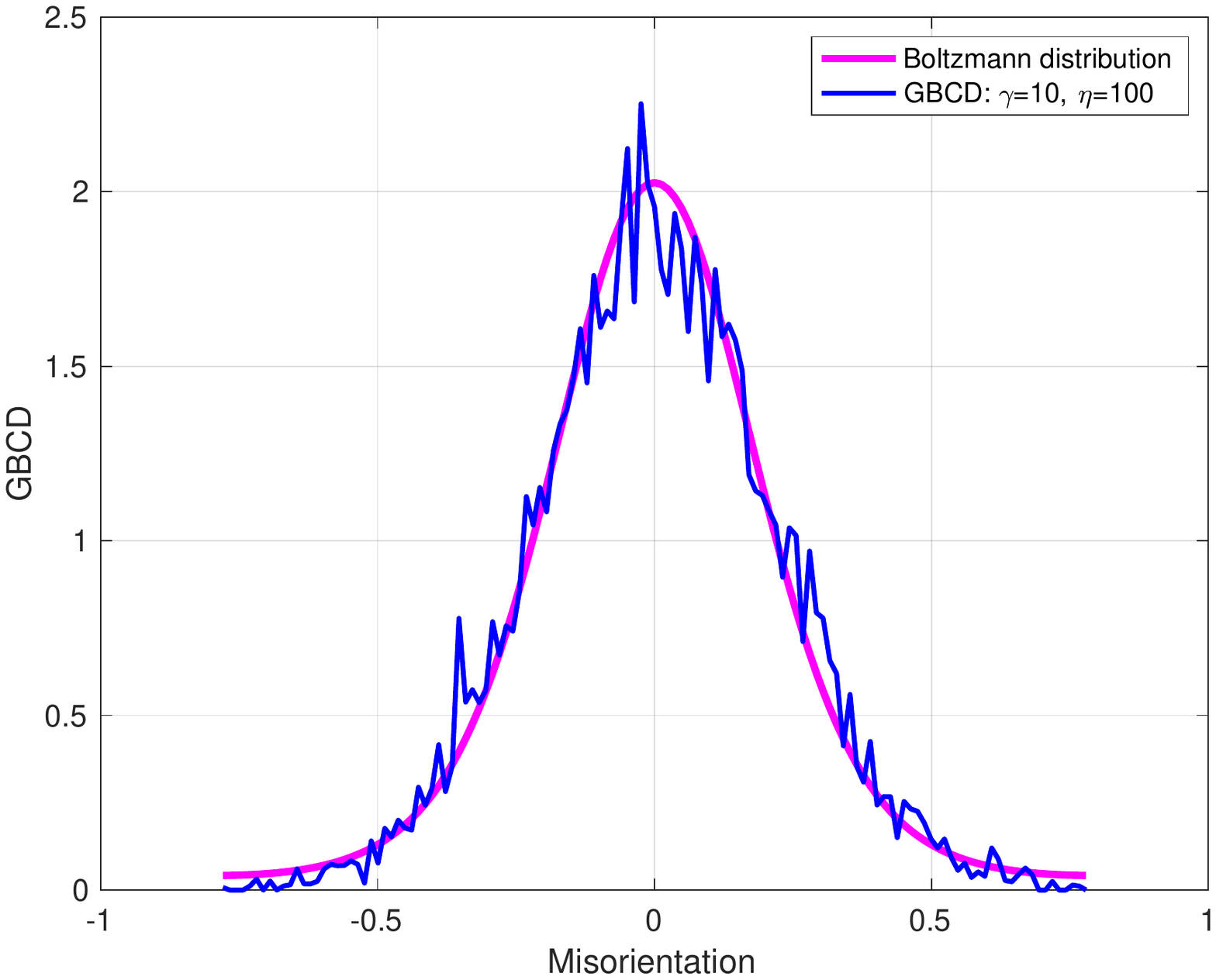}\hspace{-0.7cm}
\includegraphics[width=2.1in]{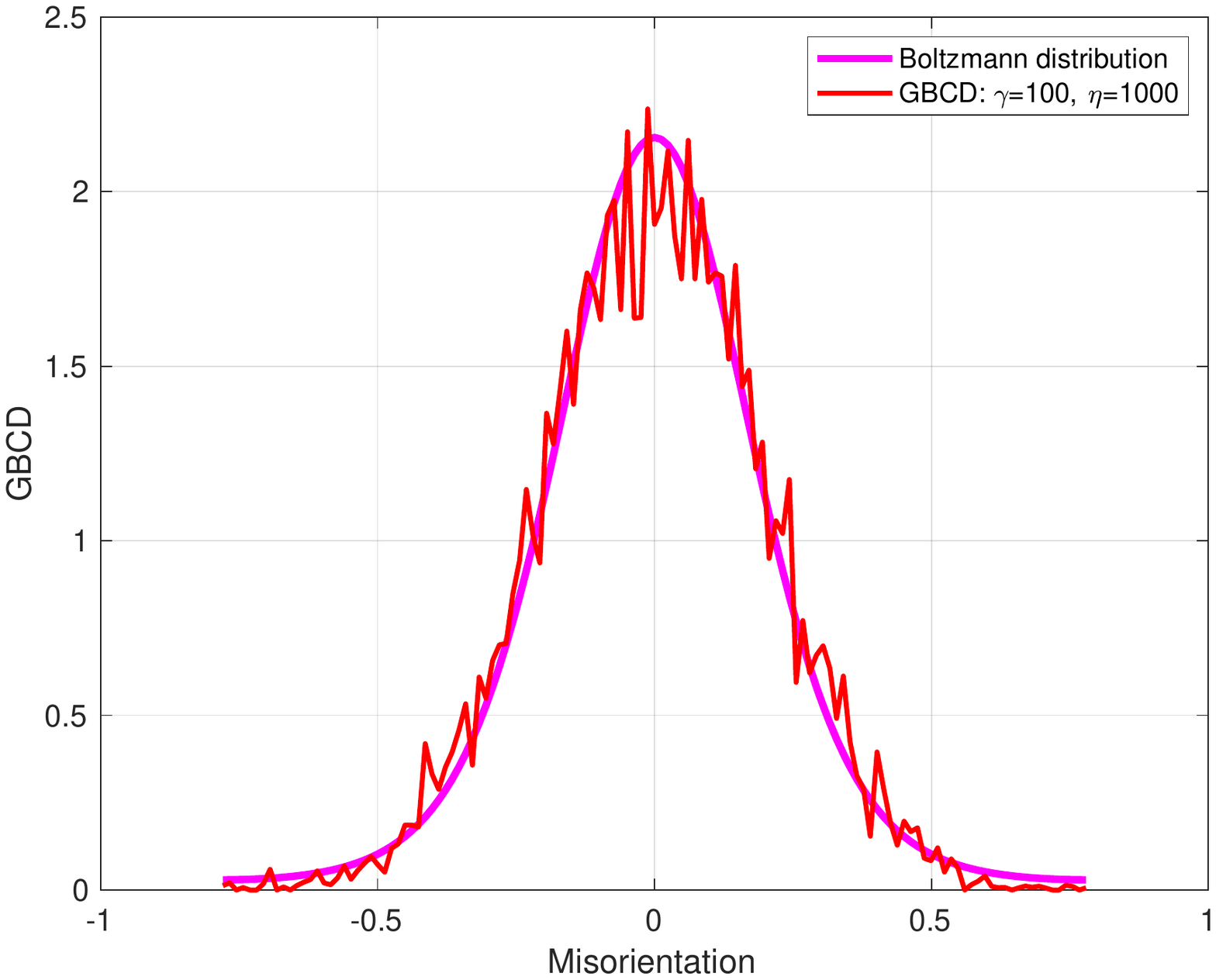}\hspace{-0.7cm}
\includegraphics[width=2.1in]{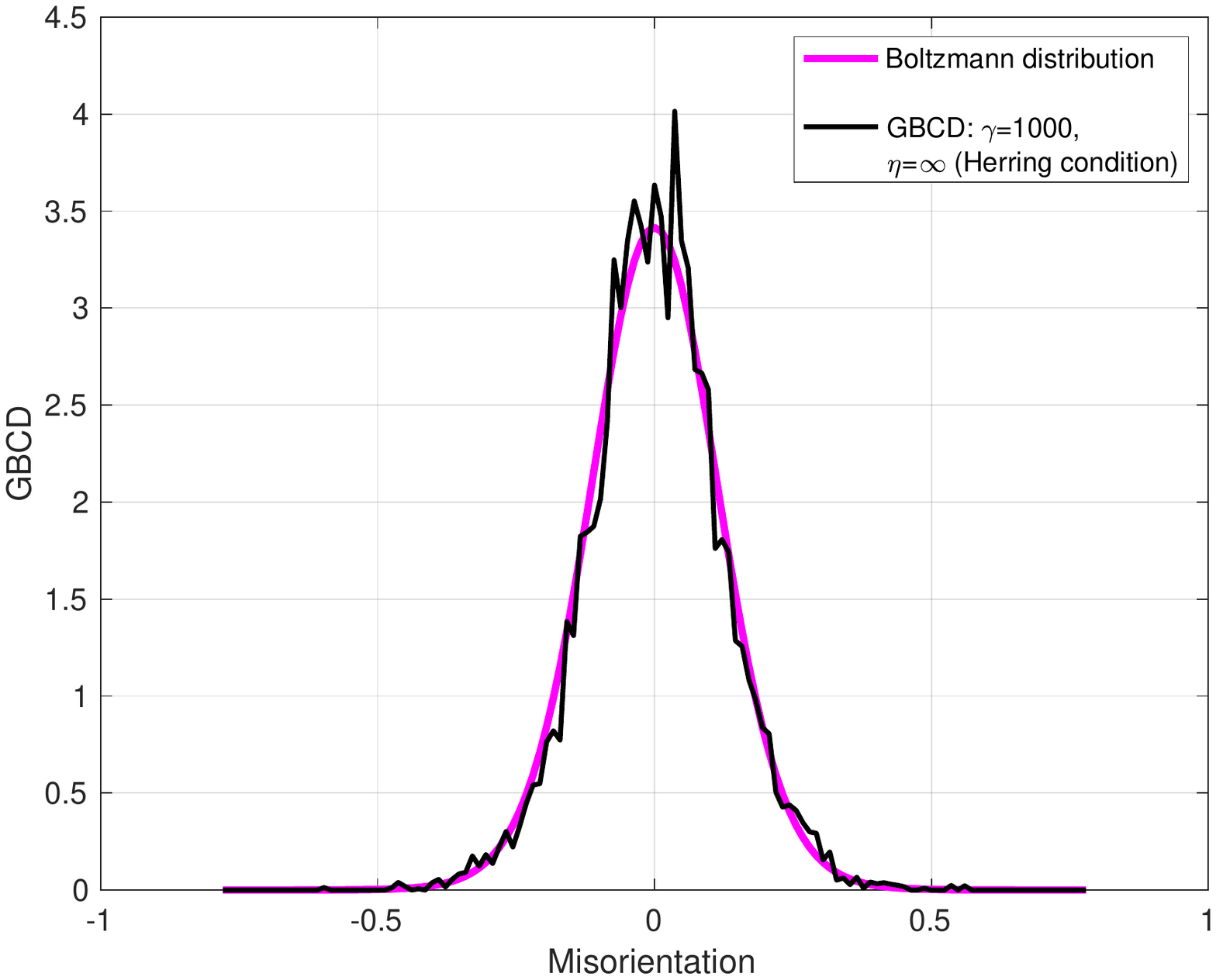}
\vspace{-1.8cm}
\caption{\footnotesize Grain growth system (\ref{eq:6.4n}) with finite
  $\mu$ (with curvature), one run of $2$D trial with $10000$ initial
  grains: {\it (a) Left plot,} GBCD (blue curve) at $T_{\infty}$ versus Boltzmann distribution with ``temperature''-
$D\approx 0.064$ (magenta curve), grain growth system with mobility of triple junctions $\eta =100$ 
  and the misorientation parameter $\gamma=10$. 
 {\it (b) Middle plot}, GBCD (red curve) at $T_{\infty}$ versus Boltzmann distribution with ``temperature''-
$D\approx 0.058$ (magenta curve), grain growth system with mobility of triple junctions $\eta =1000$ 
  and the misorientation parameter $\gamma=100$.  {\it (c) Right plot}, GBCD (black curve) at $T_{\infty}$ versus Boltzmann distribution with ``temperature''-
$D\approx 0.026$ (magenta curve), grain growth system with mobility of
triple junctions $\eta \to \infty$ (Herring condition) 
  and the misorientation parameter $\gamma=1000$.}\label{fig6}
\end{figure}

\begin{figure}[hbtp]
\centering
\vspace{-1.8cm}
\includegraphics[width=2.1in]{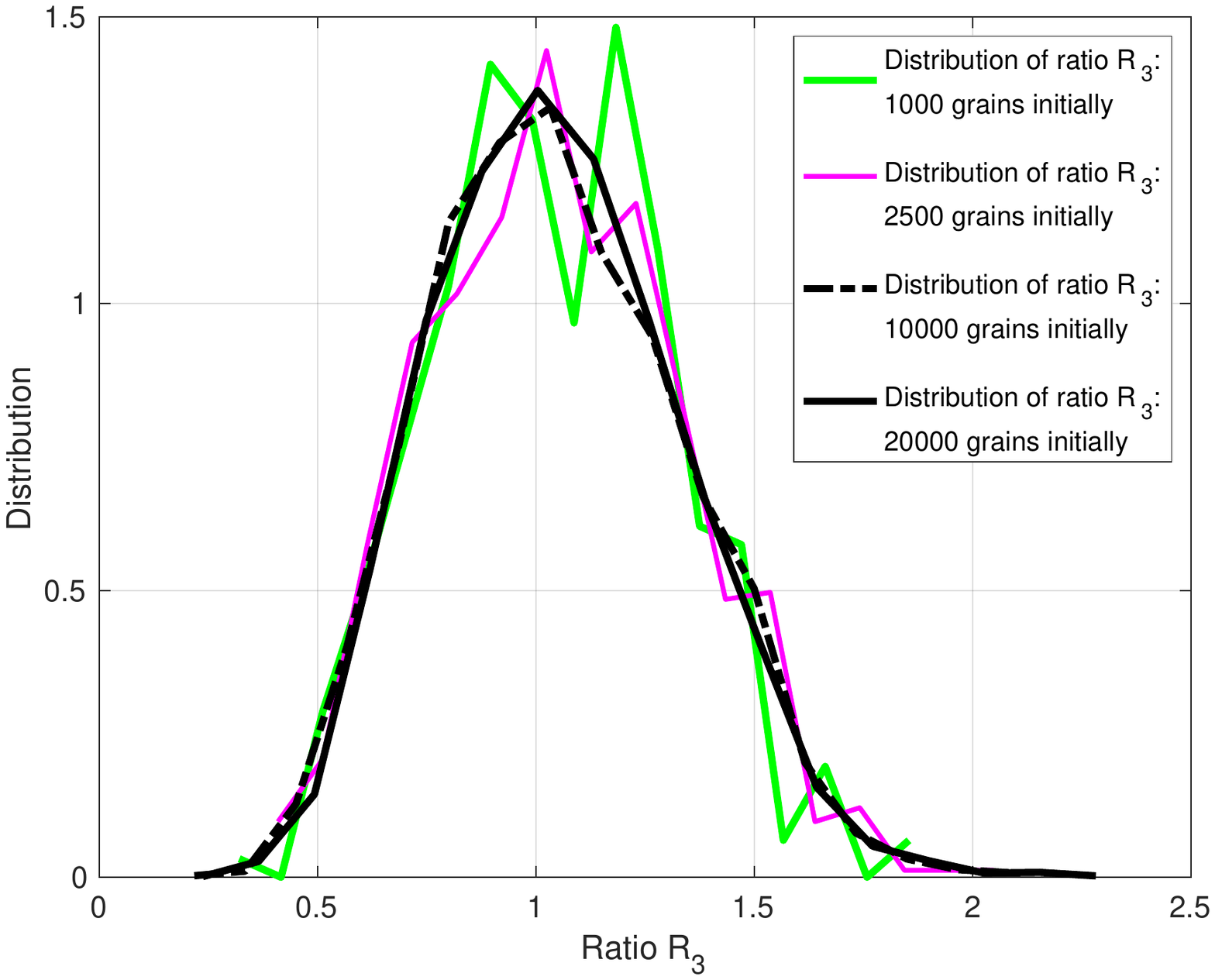}\hspace{-0.7cm}
\includegraphics[width=2.1in]{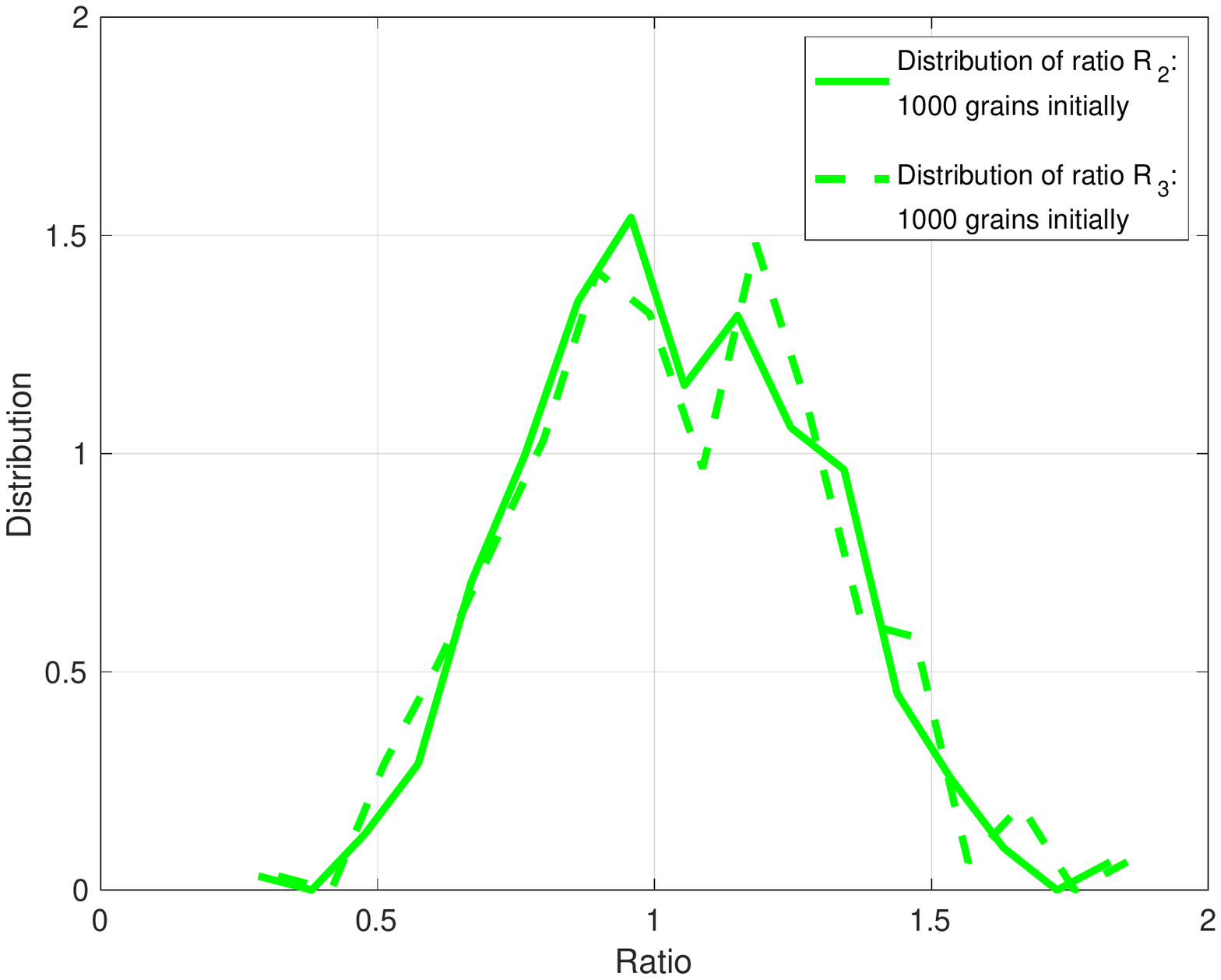}\hspace{-0.7cm}
\includegraphics[width=2.1in]{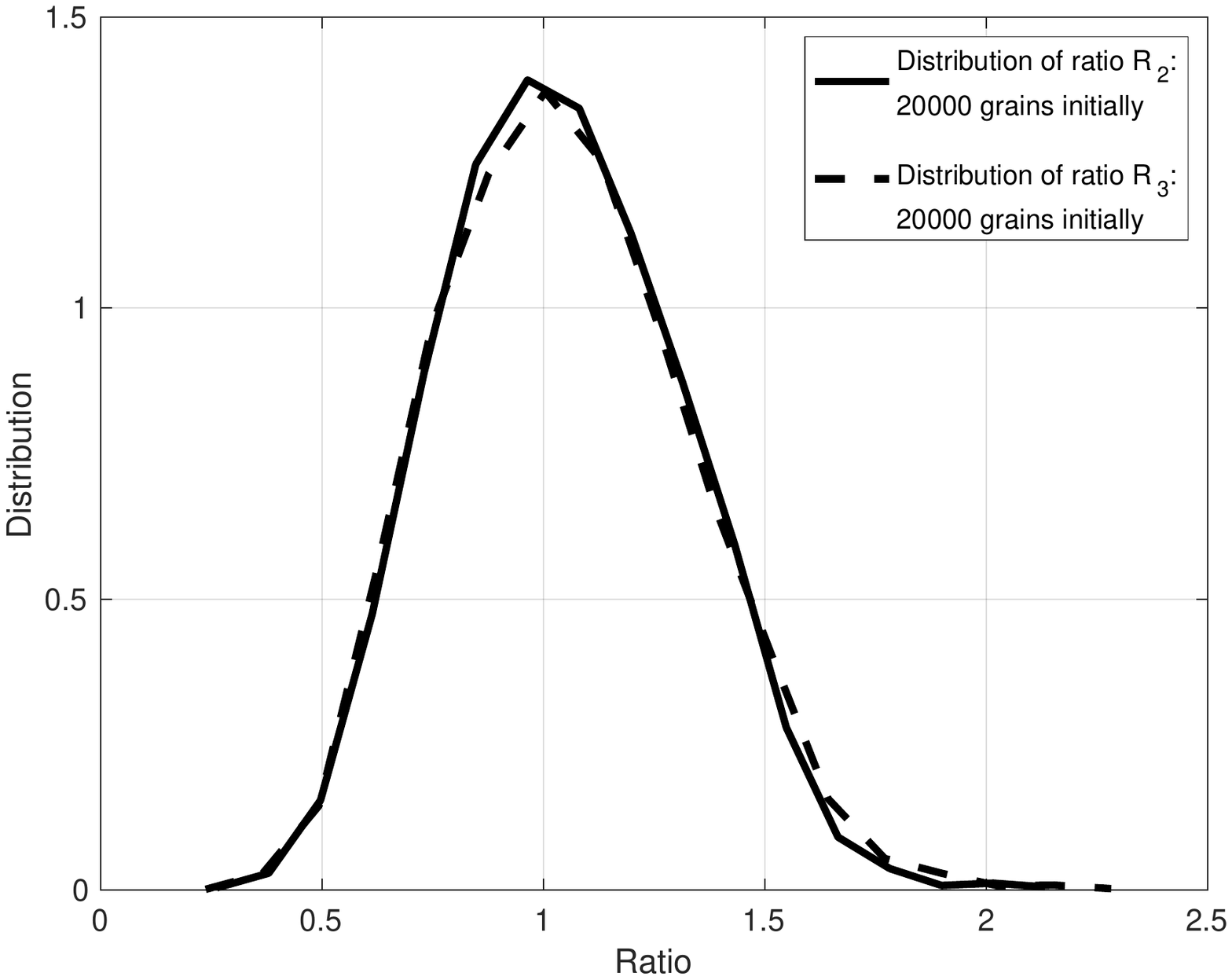}
\vspace{-1.8cm}
\caption{\footnotesize Grain growth system (\ref{eq:6.4n}) with finite
  $\mu$ (with curvature): {\it (a) Left plot,} distributions of ratio
  $R_3$  (\ref{eqR3}) system with 1000 grains initially  (solid green),
 system with 2500 grains initially  (solid magenta), system with 10000
 grains initially  (dashed point
  black),  and system with 20000 grains initially (solid black).
 {\it (b) Middle plot}, comparison of distributions of ratio
  $R_2$ (\ref{eqR2}) (solid green) and $R_3$  (\ref{eqR3})  (dashed green)
  for system with 1000 grains initially. 
{\it (c) Right plot}, comparison of distributions of ratio
  $R_2$ (\ref{eqR2}) (solid black) and $R_3$  (\ref{eqR3})  (dashed black)
  system with 20000 grains initially.
Grain growth systems are considered with mobility of triple junctions $\eta \to
  \infty$ (Herring condition)
  and no dynamic misorientation ($\gamma=0$). The closest mesh node
  of the grain boundary to
  the triple junction $\vec{a}$ is used as
  $x_i$. The distributions are plotted at $T_{\infty}$.}\label{fig7}
\end{figure}

\begin{figure}[hbtp]
\centering
\vspace{-1.8cm}
\includegraphics[width=2.1in]{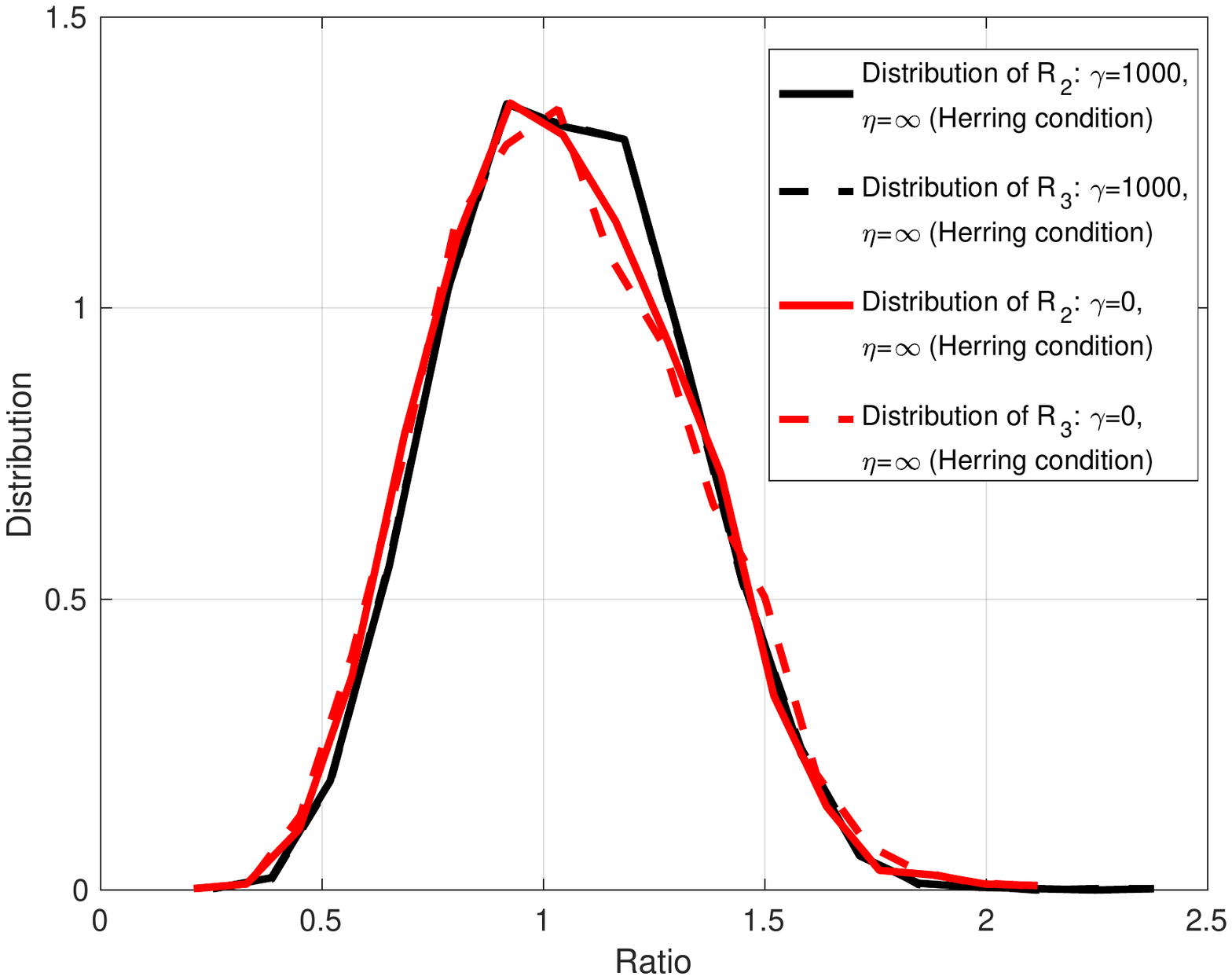}\hspace{-0.7cm}
\includegraphics[width=2.1in]{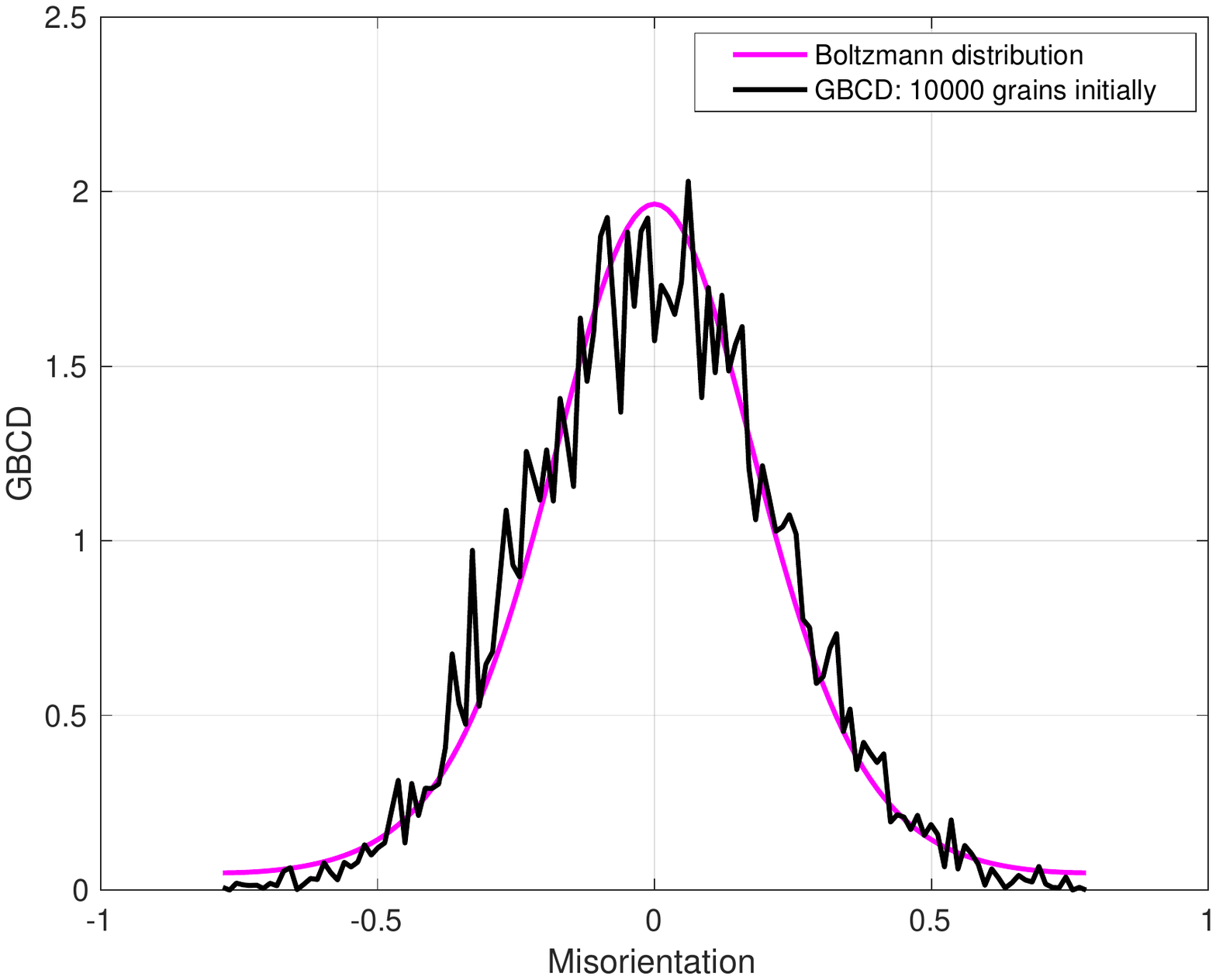}\hspace{-0.7cm}
\includegraphics[width=2.1in]{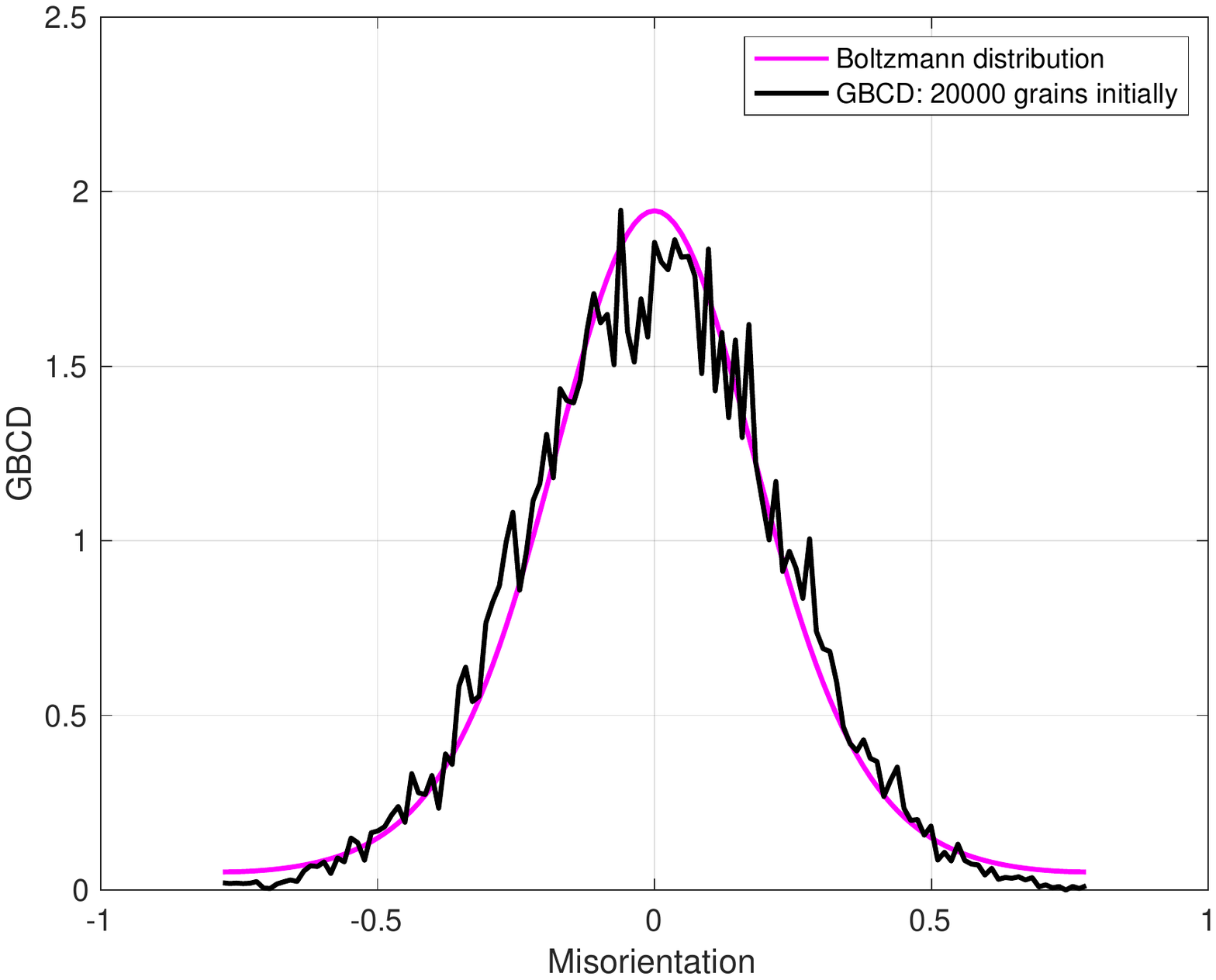}
\vspace{-1.8cm}
\caption{\footnotesize  Grain growth system (\ref{eq:6.4n}) with finite
  $\mu$ (with curvature): {\it (a) Left plot,}  one run of $2$D trial with $10000$ initial
  grains,  comparison of distributions of ratio
  $R_2$ (\ref{eqR2}) (solid black) and $R_3$  (\ref{eqR3})  (dashed
  black), grain growth system with mobility of triple junctions $\eta \to
  \infty$ (Herring condition)
  and $\gamma=1000$. Comparison of distributions of ratio
  $R_2$ (\ref{eqR2}) (solid red) and $R_3$  (\ref{eqR3})  (dashed
  red), grain growth system with mobility of triple junctions $\eta \to
  \infty$ (Herring condition)
  and no dynamic misorientations ($\gamma=0$). The closest mesh node
  of the grain boundary to
  the triple junction $\vec{a}$ is used as
  $x_i$.
The distributions are plotted at $T_{\infty}$ {\it (b) Middle plot},  one run of $2$D trial with $10000$ initial
  grains, GBCD (black curve) at $T_{\infty}$ versus Boltzmann distribution with ``temperature''-
$D\approx 0.068$ (magenta curve). 
 {\it (c) Right plot}, one run of $2$D trial with $20000$ initial
  grains,  GBCD (black curve) at $T_{\infty}$ versus Boltzmann distribution with ``temperature''-
$D\approx 0.069$ (magenta curve). Grain growth system with mobility of triple junctions $\eta \to
  \infty$ (Herring condition)
  and no dynamic misorientation ($\gamma=0$).}\label{fig8}
\end{figure}

\begin{figure}[hbtp]
\centering
\vspace{-1.8cm}
\includegraphics[width=3.0in]{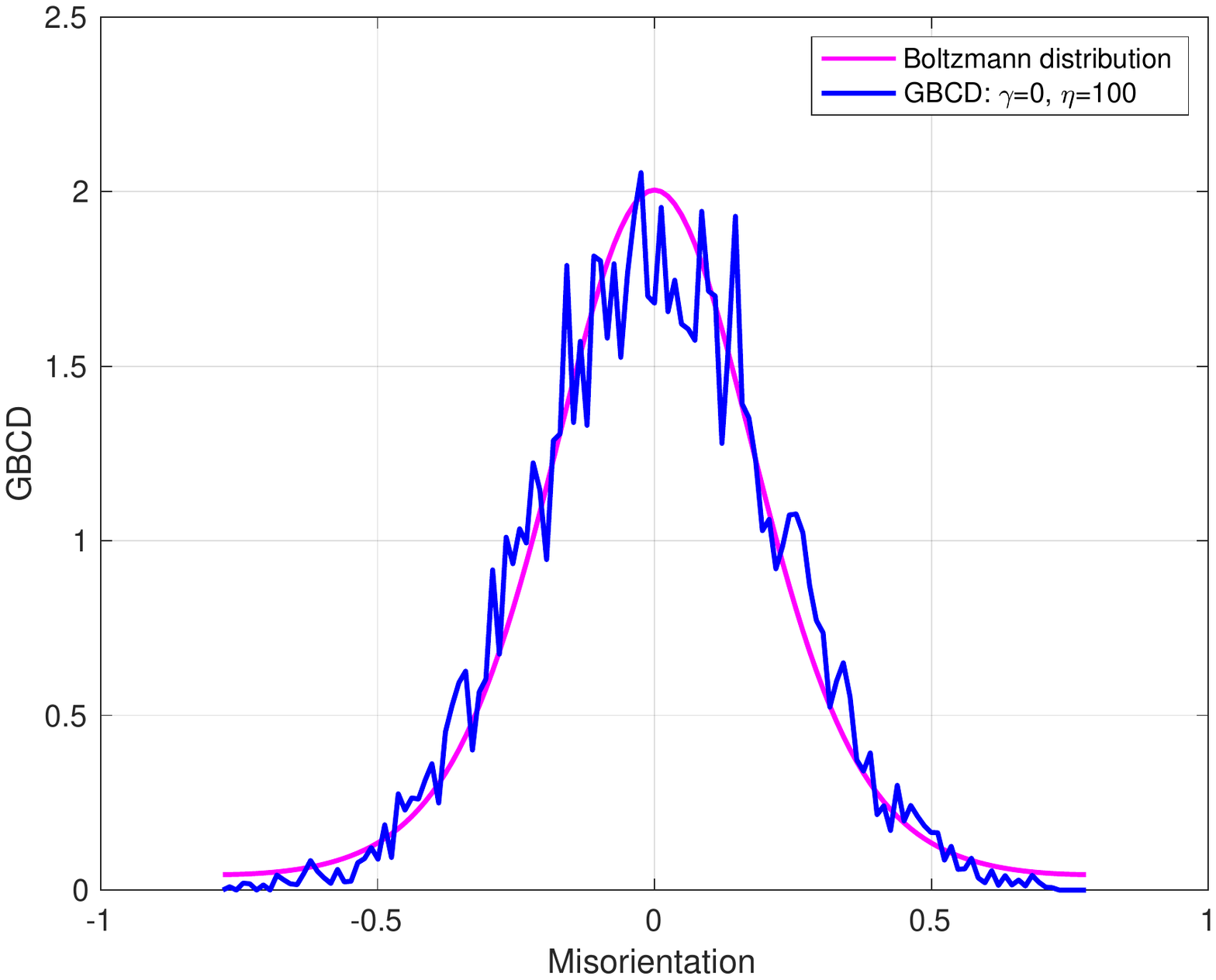}\hspace{-0.7cm}
\includegraphics[width=3.0in]{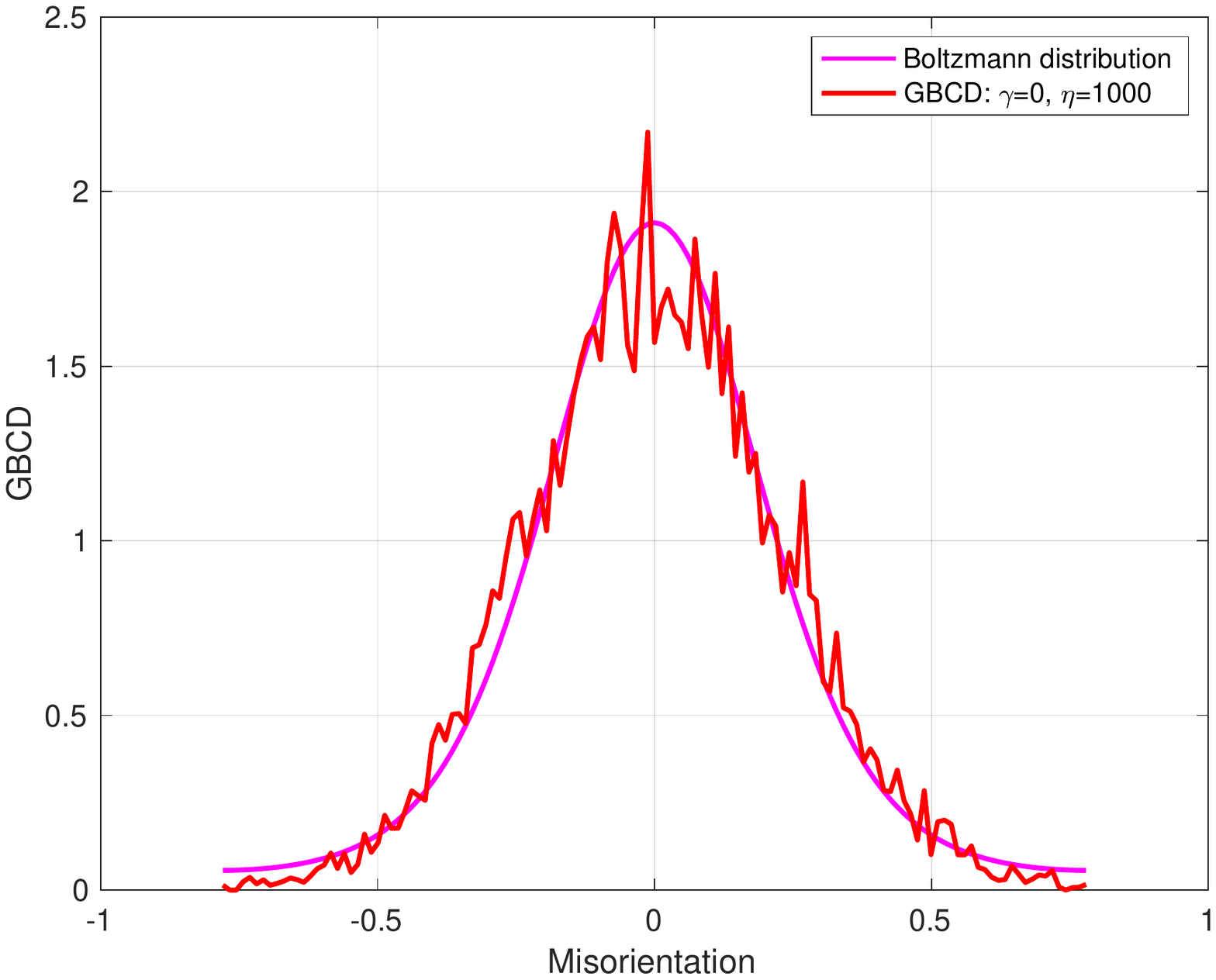}
\vspace{-1.8cm}
\caption{\footnotesize  Grain growth system (\ref{eq:6.4n}) with finite
  $\mu$ (with curvature), one run of $2$D trial with $10000$ initial
  grains:  {\it (a) Left plot},  GBCD (blue curve) at $T_{\infty}$ versus Boltzmann distribution with ``temperature''-
$D\approx 0.066$ (magenta curve), grain growth system with mobility of triple
junctions $\eta=100$ and no dynamic misorientation ($\gamma=0$). 
 {\it (b) Right plot},  GBCD (red curve) at $T_{\infty}$ versus Boltzmann distribution with ``temperature''-
$D\approx 0.071$ (magenta curve), grain growth system with mobility of triple
junctions $\eta=1000$ and no dynamic misorientation ($\gamma=0$).} \label{fig9}
\end{figure}

\begin{figure}[hbtp]
\centering
\vspace{-1.8cm}
\includegraphics[width=2.1in]{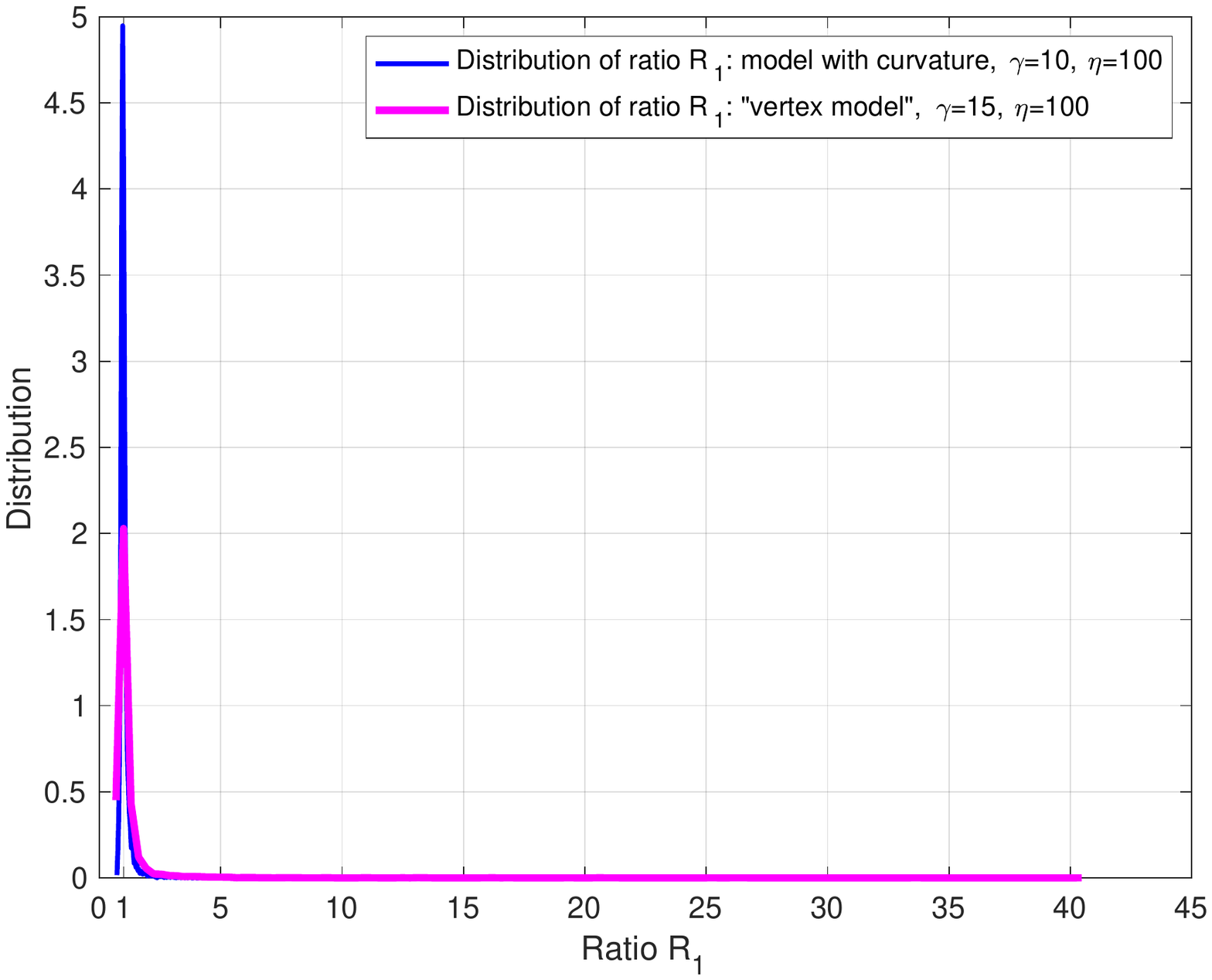}\hspace{-0.7cm}
\includegraphics[width=2.1in]{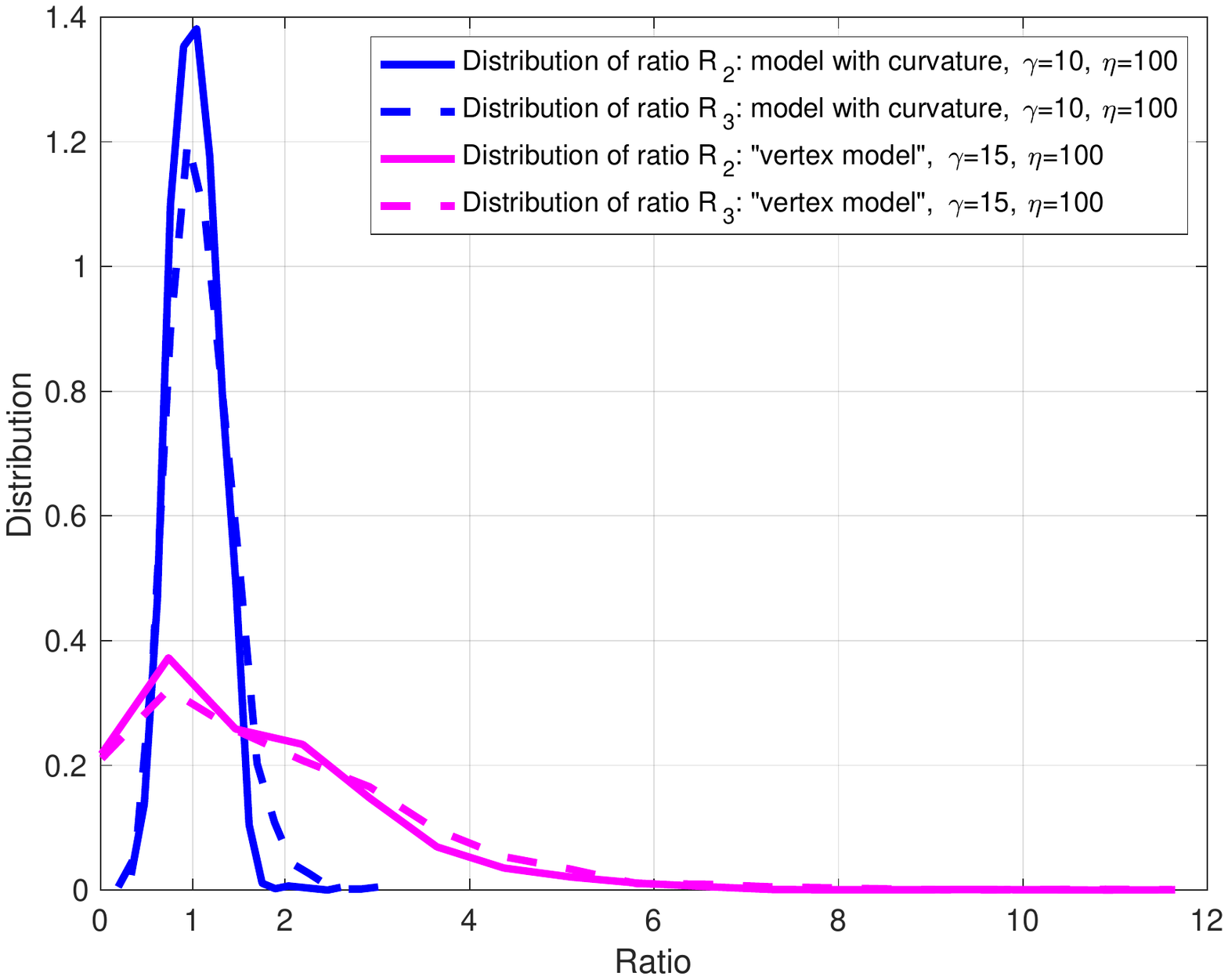}\hspace{-0.7cm}
\includegraphics[width=2.1in]{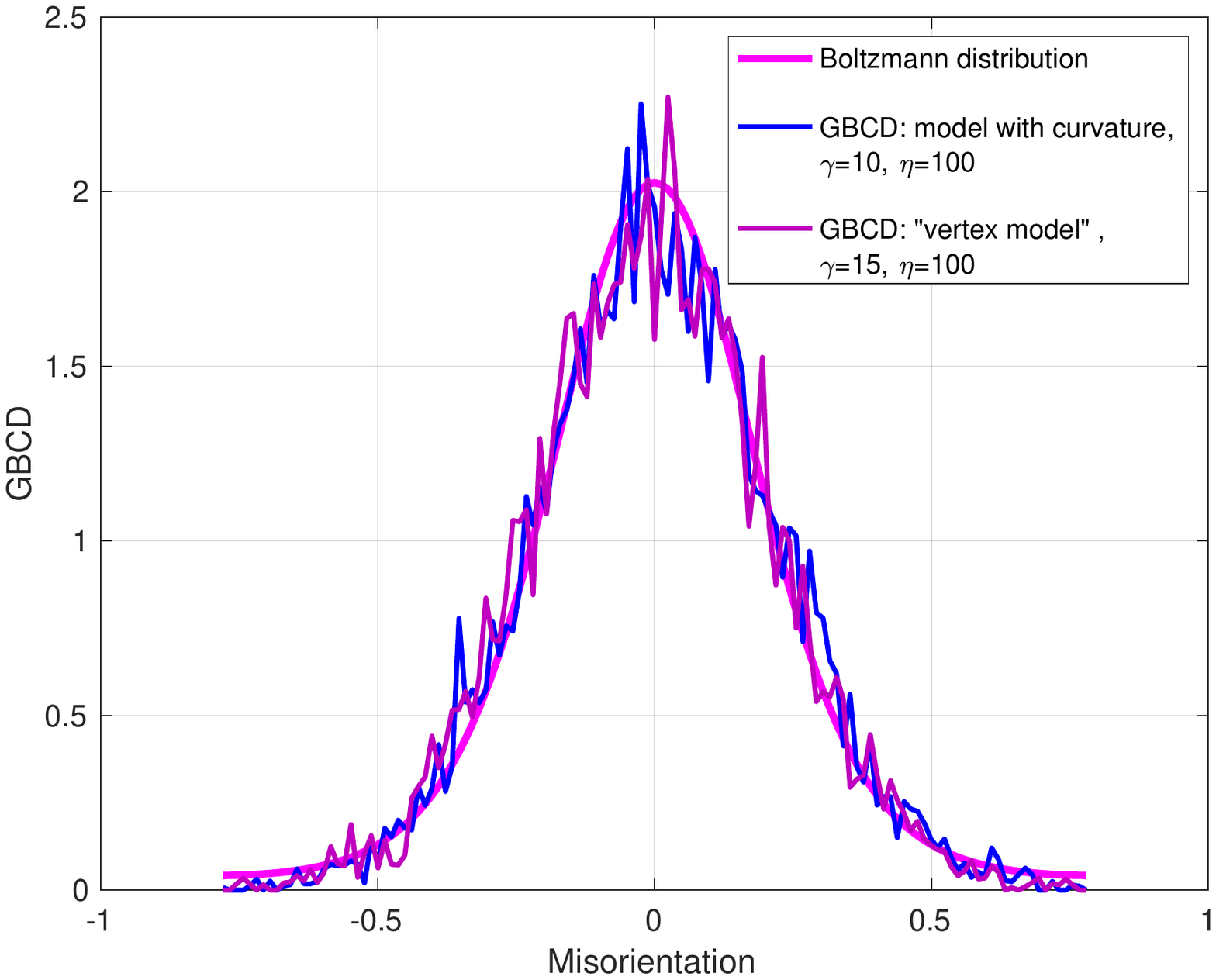}
\vspace{-1.8cm}
\caption{\footnotesize One run of $2$D trial with $10000$ initial
  grains: {\it (a) Left plot}: Comparison of distributions of ratio
  $R_1$ (\ref{eqR1}) (solid blue) for model with curvature (finite
  $\mu$) (\ref{eq:6.4n}) and $R_1$
  (\ref{eqR1})  (solid magenta) ``vertex model'' with ($\mu\to \infty$) (\ref{eq:6.4n}). {\it (b) Middle plot}: comparison of distributions of ratio
  $R_2$ (\ref{eqR2}) (solid blue) and $R_3$  (\ref{eqR3})  (dashed blue) for
  model with curvature (finite
  $\mu$) (\ref{eq:6.4n}), and comparison of distributions of ratio
  $R_2$ (\ref{eqR2}) (solid magenta) and $R_3$  (\ref{eqR3})  (dashed magenta) for
  ``vertex model'' with ($\mu\to \infty$) (\ref{eq:6.4n}). The closest mesh node
  of the grain boundary to
  the triple junction $\vec{a}$ is used as
  $x_i$. The distributions are plotted at $T_{\infty}$. {\it (c) Right plot}: One run of $2$D trial with $10000$ initial
  grains, GBCD (blue curve) ``curvature model''  (finite
  $\mu$) (\ref{eq:6.4n}),  GBCD (dark magenta
  curve) ``vertex model'' ($\mu\to \infty$) (\ref{eq:6.4n}) at $T_{\infty}$ versus Boltzmann distribution with ``temperature''-
$D\approx 0.064$ (magenta curve). Grain growth ``curvature model'' is considered with
mobility of triple junctions $\eta =100$ and $\gamma=10$,  and grain
growth ``vertex model'' is considered with
mobility of triple junctions $\eta =100$ and $\gamma=15$.}\label{fig10}
\end{figure}

\section{Conclusion}\label{sec14}
 \par In this paper, we study a stochastic model for the evolution of planar grain boundary
network in order to be able to incorporate and model the effect of the critical
events during grain growth (coarsening).  We start with a simplified model and, hence,  consider the Langevin
equation analog of the model from \cite{MR4263432},  with the interactions among
triple junctions and misorientations modeled as white noise.
The proposed system considers
anisotropic grain boundary energy which depends on lattice misorientation and
takes into account mobility of the triple junctions, as well as
independent dynamics of the misorientations. We derive the associated
Fokker-Planck equation and establish fluctuation-dissipation
principle. Next,  due to degeneracy and singularity of the system
energy, we use weighted $L^2$
space to establish long time asymptotics of the solution to the
Fokker-Planck equation, 
the joint probability density function of misorientations and triple
junctions,  as well as  of the closely related marginal probability density of misorientations
(the results are obtained under fluctuation-dissipation assumption).  As a part of our future work, we will study the
logarithmic-Sobolev inequality  \cite{MR1842428,MR0889476,MR3497125,MR1760620} and construct the $L^1$ theory of the
system.
 \par Furthermore, for an equilibrium configuration of a boundary network, we derive explicit
local algebraic relations, a generalized Herring Condition formula, as well as
formula that connects grain boundary energy density with the geometry
of the grain boundaries that share a triple junction.  Even though the
considered simplified stochastic model neglects
the explicit interactions and correlations among triple junctions, the considered
specific form of the noise, under the fluctuation-dissipation assumption,
provides partial information about evolution of a grain boundary
network, and is consistent with presented results of extensive grain
growth simulations.  As a part of our future research, we also plan to
identify and model explicitly correlations, including nucleation \cite{MR2442975,ZHANG20106574} and interactions
among triple junctions, as well as extend theory to different
statistical metrics of grain growth.
\section*{Acknowledgments}\label{sec:Ack}

The authors are grateful to David Kinderlehrer for the fruitful
discussions, inspiration and motivation of the work. The authors are
also grateful to colleagues Katayun Barmak and Lajos Horvath for
their collaboration and helpful discussions. Yekaterina Epshteyn acknowledges
partial support of NSF DMS-1905463, Masashi Mizuno
acknowledges partial support of JSPS KAKENHI Grant No.18K13446, Chun Liu acknowledges partial support of
NSF DMS-1759535 and NSF DMS-1759536.

\bibliographystyle{plain}
\bibliography{references}

\end{document}